\documentclass{sig-alt-full}

\usepackage{amsmath}
\usepackage{amssymb,latexsym}
\usepackage{amsfonts}
\usepackage{eucal}
\usepackage{graphicx}
\usepackage{hyperref}
\usepackage[usenames,dvipsnames]{color}
\usepackage{pdfcolmk}
\usepackage[cmtip,arrow]{xy}
\usepackage{pb-diagram,pb-xy}

\usepackage{subfigure}
\usepackage{rotating}
\usepackage{array}

\newcommand{\Aa}{\mathbb{A}}
\newcommand{\Zz}{\mathbb{Z}}
\newcommand{\Ff}{\mathbb{F}}
\newcommand{\Rr}{\mathbb{R}}
\newcommand{\Cc} {\mathbb{C}}

\newcommand{\Hh}{\operatorname{\mathrm{H}}}
\newcommand{\Ch}{\operatorname{\mathrm{C}}}

\newcommand\im{\mathop{\rm Im}}
\renewcommand\ker{\mathop{\rm Ker}}

\newcommand\pairing[2]{\langle #1 , #2 \rangle}

\newcommand{\nldr}{\textsc{nldr}}
\newcommand{\Rips}{\operatorname{\rm Rips}}
\newcommand{\Witness}{\operatorname{\rm Witness}}

\newtheorem{Lemma}{Lemma}

\newtheorem{Proposition}[Lemma]{Proposition}

\begin{document}
\conferenceinfo{SCG'09,} {June 8--10, 2009, Aarhus, Denmark.} 
\CopyrightYear{2009}
\crdata{978-1-60558-501-7/09/06}

\title{Persistent Cohomology and Circular Coordinates%
}

\numberofauthors{2}
\author{
\alignauthor
Vin de Silva%
\titlenote{VdS has been partially supported by DARPA, through grants HR0011-07-1-0002 and HR0011-05-1-0007.}
\\
  \affaddr{Department of Mathematics}\\
  \affaddr{Pomona College}\\
  \affaddr{Claremont, California}\\
  \email{vin.desilva@pomona.edu}
\alignauthor
Mikael Vejdemo-Johansson%
\titlenote{MVJ has been partially supported by the Office of Naval Research, through grant N00014-08-1-0931.}
\\
  \affaddr{Department of Mathematics}\\
  \affaddr{Stanford University}\\
  \affaddr{Stanford, California}\\
  \email{mik@math.stanford.edu}
}

\maketitle
\begin{abstract}
Nonlinear dimensionality reduction (NLDR) algorithms such as Isomap, LLE and Laplacian Eigenmaps address the problem of representing high-dimensional nonlinear data in terms of low-dimensional coordinates which represent the intrinsic structure of the data. This paradigm incorporates the assumption that real-valued coordinates provide a rich enough class of functions to represent the data faithfully and efficiently. On the other hand, there are simple structures which challenge this assumption: the circle, for example, is one-dimensional but its faithful representation requires two real coordinates. In this work, we present a strategy for constructing circle-valued functions on a statistical data set. We develop a machinery of persistent cohomology to identify candidates for significant circle-structures in the data, and we use harmonic smoothing and integration to obtain the circle-valued coordinate functions themselves. We suggest that this enriched class of coordinate functions permits a precise NLDR analysis of a broader range of realistic data sets.
\end{abstract}




\section{Introduction}
\label{sec:introduction}

Nonlinear dimensionality reduction (\nldr) algorithms address the following problem: given a high-dimensional collection of data points $X \subset \Rr^N$, find a low-dimensional embedding $\phi: X \to \Rr^n$ (for some $n \ll N$) which faithfully preserves the `intrinsic' structure of the data. For instance, if the data have been obtained by sampling from some unknown manifold $M \subset \Rr^N$ --- perhaps the parameter space of some physical system --- then $\phi$ might correspond to an $n$-dimensional coordinate system on~$M$. If $M$ is completely and non-redundantly parametrized by these $n$~coordinates, then the \nldr\ is regarded as having succeeded completely.

Principal components analysis, or linear regression, is the simplest form of dimensionality reduction; the embedding function~$\phi$ is taken to be a linear projection. This is closely related to (and sometimes identifed with) classical multidimensional scaling~\cite{Cox_Cox_1994}.

When there are no satisfactory linear projections, it becomes necessary to use \nldr. Prominent algorithms for \nldr\ include Locally Linear Embedding~\cite{Roweis_Saul_2000}, Isomap~\cite{Tenenbaum_dS_L_2000}, Laplacian Eigenmaps~\cite{Belkin_Niyogi_2002}, Hessian Eigenmaps~\cite{Donoho_Grimes_2003}, and many more.

These techniques share an implicit assumption that the unknown manifold~$M$ is well-described by a finite set of coordinate functions $\phi_1, \phi_2, \dots, \phi_n: M \to \Rr$. Explicitly, some of the correctness theorems in these studies depend on the hypothesis that $M$ has the topological structure of a convex domain in some~$\Rr^n$. This hypothesis guarantees that good coordinates exist, and shifts the burden of proof onto showing that the algorithm recovers these coordinates.

In this paper we ask what happens when this assumption fails.
The simplest space which challenges the assumption is the circle, which is one-dimensional but requires two real coordinates for a faithful embedding. Other simple examples include the annulus, the torus, the figure eight, the 2-sphere, the last three of which present topological obstructions to being embedded in the Euclidean space of their natural dimension. We propose that an appropriate response to the problem is to enlarge the class of coordinate functions to include circle-valued coordinates $\theta: M \to S^1$. In a physical setting, circular coordinates occur naturally as angular and phase variables. Spaces like the annulus and the torus are well described by a combination of real and circular coordinates. (The 2-sphere is not so lucky, and must await its day.)

The goal of this paper is to describe a natural procedure for constructing circular coordinates on a nonlinear data set using techniques from classical algebraic topology and its 21st-century grandchild, persistent topology. We direct the reader to~\cite{Hatcher_2002} as a general reference for algebraic topology, and to~\cite{Zomorodian_Carlsson_2005} for a streamlined account of persistent homology.

\subsection{Related work}
There have been other attempts to address the problem of finding good coordinate representations of simple non-Euclidean data spaces. One approach~\cite{Pless_Simon_2002} is to use modified versions of multidimensional scaling specifically devised to find the best embedding of a data set into the cylinder, the sphere and so on. The target space has to be chosen in advance. Another class of approaches~\cite{Lee_Verleysen_2004,Dixon_J_P_2006} involves cutting the data manifold along arcs and curves until it has trivial topology. The resulting configuration can then be embedded in Euclidean space in the usual way.
In our approach, the number of circular coordinates is not fixed in advance, but is determined experimentally after a persistent homology calculation. Moreover, there is no cutting involved; the coordinate functions respect the original topology of the data.

\subsection{Overview}

The principle behind our algorithm is the following equation from homotopy theory, valid for topological spaces~$X$ with the homotopy type of a cell complex (which covers everything we normally encounter):
\begin{equation}
\label{eq:main}
[X, S^1] = \Hh^1(X; \Zz)
\end{equation}
The left-hand side denotes the set of equivalence classes of continuous maps from $X$ to the circle~$S^1$; two maps are equivalent if they are homotopic (meaning that one map can be deformed continuously into the other); the right-hand side denotes the 1-dimensional cohomology of~$X$, taken with integer coefficients.
In other language: $S^1$ is the {classifying space} for $\Hh^1$, or equivalently $S^1$ is the Eilenberg--MacLane space $K(\Zz,1)$.
See section~4.3 of~\cite{Hatcher_2002}.

If $X$ is a contractible space (such as a convex subset of~$\Rr^n$), then $\Hh^1(X; \Zz) = 0$ and Equation~\eqref{eq:main} tells us not to bother looking for circular functions: all such functions are homotopic to a constant function. On the other hand, if $X$ has nontrivial topology then there may well exist a nonzero cohomology class $[\alpha] \in \Hh^1(X; \Zz)$; we can then build a continuous function $X \to S^1$ which in some sense reveals $[\alpha]$.

Our strategy divides into the following steps.
\begin{enumerate}
\item
Represent the given discrete data set as a simplicial complex or filtered simplicial complex.
\item
Use persistent cohomology to identify a `significant' cohomology class in the data. For technical reasons, we carry this out with coefficients in the field $\Ff_p$ of integers modulo~$p$, for some prime~$p$. This gives us $[\alpha_p] \in \Hh^1(X; \Ff_p)$.
\item
Lift $[\alpha_p]$ to a cohomology class with integer coefficients: $[\alpha] \in \Hh^1(X; \Zz)$.
\item
Smoothing: replace the integer cocycle~$\alpha$ by a harmonic cocycle in the same cohomology class: $\bar{\alpha} \in \Ch^1(X; \Rr)$.
\item
Integrate the harmonic cocycle $\bar{\alpha}$ to a circle-valued function $\theta: X \to S^1$.
\end{enumerate}

The paper is organized as follows. In Section~\ref{sec:coho}, we derive what we need of equation~\eqref{eq:main}. Steps (1--5) of the algorithm are addressed in Sections \ref{sec:pcd2sc}--\ref{sec:integration}, respectively. In Section~\ref{sec:examples} we report some experimental results.
%

\section{Algorithm Details}
\label{sec:algorithm}


\subsection{Cohomology and circular functions}
\label{sec:coho}

Let $X$ be a finite simplicial complex. Let $X^0, X^1, X^2$ denote the sets of vertices, edges and triangles of~$X$, respectively. We suppose that the vertices are totally ordered (in an arbitrary way). If $a < b$ then the edge between vertices $a,b$ is always written $ab$ and not $ba$. Similarly, if $a < b < c$ then the triangle with vertices $a,b,c$ is always written $abc$.

Cohomology can be defined as follows. Let $\Aa$ be a commutative ring (for example $\Aa = \Zz, \Ff_p, \Rr$). We define 0-cochains, 1-cochains, and 2-cochains as follows:
\begin{eqnarray*}
\Ch^0 \,=\, \Ch^0(X; \Aa) &=&
\left\{
\mbox{functions $f: X^0 \to \Aa$}
\right\}
\\
\Ch^1 \,=\, \Ch^1(X; \Aa) &=&
\left\{
\mbox{functions $\alpha: X^1 \to \Aa$}
\right\}
\\
\Ch^2 \,=\, \Ch^2(X; \Aa) &=&
\left\{
\mbox{functions $A: X^2 \to \Aa$}
\right\}
\end{eqnarray*}
These are modules over~$\Aa$. We now define coboundary maps $d_0: \Ch^0 \to \Ch^1$ and $d_1: \Ch^1 \to \Ch^2$.
\begin{eqnarray*}
(d_0 f) (ab) &=& f(b) - f(a)
\\
(d_1 \alpha) (abc) &=& \alpha(bc) - \alpha(ac) + \alpha(ab)
\end{eqnarray*}

Let $\alpha \in \Ch^1$. If $d_1 \alpha = 0$ we say that $\alpha$ is a \emph{cocycle}. If $d_0 f = \alpha$ admits a solution $f \in \Ch^0$ we say that $\alpha$ is a \emph{coboundary}. The solution~$f$, if it exists, can be thought of as the discrete integral of~$\alpha$. It is unique up to adding constants on each connected component of~$X$.

It is easily verified that $d_1 d_0 f = 0$ for any $f \in \Ch^0$. Thus, coboundaries are always cocycles, or equivalently $\im(d_0) \subseteq \ker(d_1)$. We can measure the difference between coboundaries and cocycles by defining the 1-cohomology of~$X$ to be the quotient module
\[
\Hh^1(X; \Aa) = \ker(d_1) / \im(d_0).
\]
We say that two cocycles $\alpha, \beta$ are \emph{cohomologous} if $\alpha-\beta$ is a coboundary.


We now consider integer coefficients. The following proposition fulfils part of the promise of equation~\eqref{eq:main}, by producing circle-valued functions from integer cocycles. It will be helpful to think of $S^1$ as the quotient group $\Rr/\Zz$.

\begin{Proposition}
\label{prop:integer}
Let $\alpha \in \Ch^1(X; \Zz)$ be a cocycle. Then there exists a continuous function $\theta: X \to \Rr/\Zz$ which maps each vertex to 0, and each edge $ab$ around the entire circle with winding number $\alpha(ab)$.
\end{Proposition}

\begin{proof}
We can define $\theta$ inductively on the vertices, edges, triangles, \dots\ of~$X$. The vertices and edges follow the prescription in the statement of the proposition. To extend $\theta$ to the triangles, it is necessary that the winding number of $\theta$ along the boundary of each triangle $abc$ is zero. And indeed this is $\alpha(bc) - \alpha(ac) + \alpha(ab) = d_1 \alpha (abc) = 0$. Since the higher homotopy groups of $S^1$ are all zero (\cite{Hatcher_2002}, section~4.3), $\theta$ can then be extended to the higher cells of~$X$ without obstruction.
\end{proof}

The construction in Proposition~\ref{prop:integer} is unsatisfactory in the sense that all vertices are mapped to the same point. All variation in the circle parameter takes place in the interior of the edges (and higher cells). This is rather unsmooth. For more leeway, we consider real coefficients.

\begin{Proposition}
\label{prop:real}
Let $\bar\alpha \in \Ch^1(X; \Rr)$ be a cocycle. Suppose we can find $\alpha \in \Ch^1(X; \Zz)$ and $f \in \Ch^0(X; \Rr)$ such that $\bar{\alpha} = \alpha + d_0 f$. Then there exists a continuous function $\theta: X \to \Rr/\Zz$ which maps each edge $ab$ linearly to an interval of length $\bar\alpha(ab)$, measured with sign.
\end{Proposition}

In other words, we can construct a circle-valued function out of any real cocycle $\bar\alpha$ whose cohomology class $[\bar\alpha]$ lies in the image of the natural homomorphism $\Hh^1(X; \Zz) \to \Hh^1(X; \Rr)$.

\begin{proof}
Define $\theta$ on the vertices of~$X$ by setting $\theta(a)$ to be $f(a)$ mod~$\Zz$. For each edge $ab$,
we have
\begin{eqnarray*}
\theta(b) - \theta(a) &=& f(b) - f(a)
\\
&=& d_0 f(ab)
\\
&=&
\bar{\alpha}(ab) - \alpha(ab)
\end{eqnarray*}
which is congruent to $\bar{\alpha}(ab)$ mod~$\Zz$, since $\alpha(ab)$ is an integer.

It follows that $\theta$ can be taken to map $ab$ linearly onto an interval of signed length $\bar\alpha(ab)$. Since $\bar\alpha$ is a cocyle, $\theta$ can be extended to the triangles as before; then to the higher cells.
\end{proof}

Proposition~\ref{prop:real} suggests the following tactic: from an integer cocycle~$\alpha$ we construct a cohomologous real cocycle $\bar\alpha = \alpha + d_0 f$, and then define $\theta = f$ mod~$\Zz$ on the vertices of~$X$. If we can construct $\bar\alpha$ so that the edge-lengths $|\bar\alpha(ab)|$ are small, then the behaviour of $\theta$ will be apparent from its restriction to the vertices. See Section~\ref{sec:smooth}.

\subsection{Point-cloud data to simplicial complex}
\label{sec:pcd2sc}

We now begin describing the workflow in detail. The input is a point-cloud data set: in other words, a finite set $S \subset \Rr^N$ or more generally a finite metric space. The first step is to convert $S$ into a simplicial complex and to identify a stable-looking integer cohomology class. This will occupy the next three subsections.

The first lesson of point-cloud topology~\cite{Ghrist_2008AMS} is that point-clouds are best represented by 1-parameter nested families of simplicial complexes. There are several candidate constructions: the Vietoris--Rips complex $X^\epsilon = \Rips(S, \epsilon)$ has vertex set~$S$ and includes a $k$-simplex whenever all $k+1$ vertices lie pairwise within distance $\epsilon$ of each other. The witness complex $X^\epsilon = \Witness(L, S, \epsilon)$ uses a smaller vertex set~$L \subset S$ and includes a $k$-simplex when the $k+1$ vertices lie close to other points of~$S$, in a certain precise sense (see \cite{deSilva_Carlsson_2004,Guibas_Oudot_2007}). In both cases, $X^\epsilon \subseteq X^{\epsilon'}$ whenever $\epsilon \leq \epsilon'$.
Either of these constructions will serve our purposes, but the witness complex has the computational advantage of being considerably smaller.

We determine $X^\epsilon$ only up to its 2-skeleton, since we are interested in $\Hh^1$.

\subsection{Persistent cohomology}
\label{sec:pcoho}

Having constructed a 1-parameter family $\{ X^\epsilon \}$, we apply the principle of persistence to identify cocycles that are stable across a large range for~$\epsilon$. Suppose that $\epsilon_1, \epsilon_2, \dots, \epsilon_m$ are the critical values where the complex $X^\epsilon$ gains new cells. The family can be represented as a diagram
\[
X^{\epsilon_1} \longrightarrow
X^{\epsilon_2} \longrightarrow
\dots \longrightarrow
X^{\epsilon_m}
\]
of simplicial complexes and inclusion maps. For any coefficient field~$\Ff$, the cohomology functor $\Hh^1(-; \Ff)$ converts this diagram into a diagram of vector spaces and linear maps over~$\Ff$; the arrows are reversed:
\[
\Hh^1(X^{\epsilon_1}; \Ff) \longleftarrow
\Hh^1(X^{\epsilon_2}; \Ff) \longleftarrow
\dots \longleftarrow
\Hh^1(X^{\epsilon_m}; \Ff)
\]
According to the theory of persistence~\cite{Edelsbrunner_L_Z_2002,Zomorodian_Carlsson_2005}, such a diagram decomposes as a direct sum of 1-dimensional terms indexed by half-open intervals of the form $[\epsilon_i, \epsilon_j)$. Each such term corresponds to a cochain $\alpha \in \Ch^i(X^\epsilon)$ that satisfies the cocycle condition for $\epsilon < \epsilon_j$ and becomes a coboundary for $\epsilon < \epsilon_i$.
The collection of intervals can be displayed graphically as a \emph{persistence diagram}, by representing each interval $[\epsilon_i,\epsilon_j)$ as a point $(\epsilon_i,\epsilon_j)$ in the Cartesian plane above the main diagonal.
We think of long intervals as representing trustworthy (i.e.\ stable) topological information.

\smallskip
\textsc{Choice of coefficients.}
The persistence decomposition theorem applies to diagrams of vector spaces over a field. When we work over the ring of integers~$\Zz$, however, the result is known to fail: there need not be an interval decomposition. This is unfortunate, since we require integer cocycles to construct circle maps. To finesse this problem, we pick an arbitrary prime number~$p$ (such as $p=47$) and carry out our persistence calculations over the finite field $\Ff = \Ff_p$. The resulting $\Ff_p$ cocyle must then be converted to integer coefficients: we address this in Section~\ref{sec:lift}.
\medskip

In principle we can use the ideas in~\cite{Zomorodian_Carlsson_2005} to calculate the persistent cohomology intervals and then select a long interval $[\epsilon_i, \epsilon_j)$ and a specific $\delta \in [\epsilon_i, \epsilon_j)$. We then let $X = X^\delta$ and take $\alpha$ to be the cocycle in $\Ch^1(X; \Ff)$ corresponding to the interval. 

Explicitly, persistent cocycles can be calculated in the following way.
We thank Dmitriy Morozov for this algorithm.
Suppose that the simplices in the filtered complex are totally ordered, and labelled $\sigma_1, \sigma_2, \dots, \sigma_m$ so that $\sigma_i$ arrives at time~$\epsilon_i$. For $k = 0, 1, \dots, m$ we maintain the following information:
\begin{itemize}
\item
a set of indices $I_k \subseteq \{ 1, 2, \dots, k\}$ associated with `live' cocycles;
\item
a list of cocycles $( \alpha_i : i \in I_k)$ in $\Ch^*(X^{\epsilon_k}; \Ff)$.
\end{itemize}
The cocycle $\alpha_i$ involves only $\sigma_i$ and those simplices of the same dimension that appear later in the filtration sequence (thus only $\sigma_j$ with $j \geq i$).

Initially $I_0 = \emptyset$ and the list of cocycles is empty.

To update from $k-1$ to $k$, we compute the coboundaries of the cocycles $( \alpha_i : i \in I_{k-1} )$ of~$X^{\epsilon_{k-1}}$ within the larger complex $X^{\epsilon_{k}}$ obtained by including the simplex $\sigma_k$. In fact, these coboundaries must be multiples of the elementary cocycle $\alpha = [\sigma_k]$ defined by $\alpha(\sigma_k) =1$, and $\alpha(\sigma_j) = 0$ otherwise. We can write $d\alpha_i = c_i [\sigma_k]$. If all the $c_i$ are zero, then we have one new cocycle: let $I_k = I_{k-1} \cup \{  k\}$ and define $\alpha_k = [\sigma_k]$. Otherwise, we must lose a cocycle. Let $j \in I_{k-1}$ be the largest index for which $c_j \ne 0$. We delete $\alpha_j$ by setting $I_k = I_{k-1} \setminus \{ j\}$, and we restore the earlier cocycles by setting $\alpha_i \leftarrow \alpha_i - (c_i/c_j)\alpha_j$. In this latter case, we write the persistence interval $[\epsilon_j, \epsilon_k)$ to the output.

At the end of the process, surviving cocycles are associated with semi-infinite intervals: $[\epsilon_i, \infty)$ for $i \in I_m$.

\smallskip
\textsc{Remark.}
The reader may be more familiar with persistence diagrams in homology rather than cohomology. In fact, the universal coefficient theorem~\cite{Hatcher_2002} implies that the two diagrams are identical. The salient point is that cohomology is the vector-space dual of homology, when working with field coefficients. That said, we cannot simply use the usual algorithm for persistent homology: we are interested in obtaining explicit cocycles, whereas the classical algorithm~\cite{Zomorodian_Carlsson_2005} returns cycles.
\smallskip

We will establish the correctness of this algorithm in the archival version of this paper. The expert reader may regard this as an exercise in the theory of persistence.

\subsection{Lifting to integer coefficients}
\label{sec:lift}


We now have a simplicial complex $X = X^\delta$ and a cocycle $\alpha_p \in \Ch^1(X; \Ff_p)$. The next step is to `lift' $\alpha_p$ by constructing an integer cocycle~$\alpha$ which reduces to $\alpha_p$ modulo~$p$.

To show that this is (almost) always possible, note that the short exact sequence of coefficient rings
$
0 \longrightarrow
\Zz \stackrel{\cdot p}{\longrightarrow}
\Zz \longrightarrow
\Ff_p \longrightarrow
0
$
gives rise to a long exact sequence, called the Bockstein sequence (see Section 3.E of~\cite{Hatcher_2002}). Here is the relevant section of the sequence:
%
\[
\to \Hh^1(X; \Zz) \to \Hh^1(X; \Ff_p)
\stackrel{\beta}{\to}
\Hh^2(X; \Zz) \stackrel{\cdot p}{\to} \Hh^2(X; \Zz) \to
\]
By exactness, the Bockstein homomorphism~$\beta$ induces an isomorphism between the cokernel of $\Hh^1(X; \Zz) \to \Hh^1(X; \Ff_p)$ and the kernel of $\Hh^2(X; \Zz) \stackrel{\cdot p}{\to} \Hh^2(X; \Zz)$, and this kernel is precisely the set of $p$-torsion elements of $\Hh^2(X; \Zz)$. If there is no $p$-torsion, then it follows immediately that the cokernel of the first map is zero. In other words $\Hh^1(X; \Zz) \to \Hh^1(X; \Ff_p)$ is surjective; any cocycle $\alpha_p \in \Ch^1(X; \Ff_p)$ can be lifted to a cocycle $\alpha \in \Ch^1(X; \Zz)$.

If we are unluckily sabotaged by $p$-torsion, then we pick another prime and redo the calculation from scratch: it is enough to pick a prime that does not divide the order of the torsion subgroup of $\Hh^2(X; \Zz)$, so almost any prime will do.

In practice, we construct $\alpha$ by taking the coefficients of $\alpha_p$ in $\Ff_p$ and replacing them with integers in the correct congruence class modulo~$p$. The default choice is to choose coefficients close to zero. If $d_1 \alpha = 0$ then we are done; otherwise it becomes necessary to do some repair work. Certainly $d_1\alpha \equiv 0$ modulo~$p$, so we can write $d_1 \alpha = p\eta$ for some $\eta\in \Ch^2(X; \Zz)$. In the absence of $p$-torsion, we can then solve $\eta = d_1 \zeta$ for $\zeta \in \Ch^1(X; \Zz)$, and then the required lift is $\alpha - p \zeta$. Fortunately, this has not proved necessary in any of our examples.

\smallskip
\textsc{Remark.}
We expect that $p$-torsion is extremely rare in `real' data sets, since it is symptomatic of rather subtle topological phenomena. For instance, the simplest examples which exhibit 2-torsion are the nonorientable closed surfaces (such as the projective plane and the Klein bottle).
\smallskip

\subsection{Harmonic smoothing}
\label{sec:smooth}

Given an integer cocycle $\alpha \in \Ch^1(X; \Zz)$, or indeed a real cocycle $\alpha \in \Ch^1(X; \Rr)$, we wish to find the `smoothest' real cocycle $\bar\alpha \in \Ch^1(X; \Rr)$ cohomologous to~$\alpha$. It turns out that what we want is the harmonic cocycle representing the cohomology class $[\alpha]$.

We define smoothness. Each of the spaces $\Ch^i(X; \Rr)$ comes with a natural Euclidean metric:
\begin{eqnarray*}
\| f \|^2 &=& \sum_{a \,\in X^0} |f(a)|^2,
\\
\| \alpha \|^2 &=& \sum_{ab \,\in X^1} |\alpha(ab)|^2,
\\
\| A \|^2 &=& \sum_{abc \,\in X^2} |A(abc)|^2.
\end{eqnarray*}
A circle-valued function $\theta$ is `smooth' if its total variation across the edges of~$X$ is small. The terms $| \alpha(ab)|^2$ capture the variation across individual edges; therefore what we must minimize is $\| \bar\alpha\|^2$.

\begin{Proposition}
Let $\alpha \in \Ch^1(X; \Rr)$. There is a unique solution $\bar\alpha$ to the least-squares minimization problem
\begin{equation}
\label{eq:lsqr}
\mathop{\rm argmin}_{\bar\alpha} 
\left\{ 
\| \bar\alpha \|^2
\mid
\exists f \in \Ch^0(X; \Rr),\, \bar\alpha = \alpha + d_0 f
\right\}.
\end{equation}
Moreover, $\bar\alpha$ is characterized by the equation $d_0^*\,\bar \alpha = 0$, where $d_0^*$ is the adjoint of $d_0$ with respect to the inner products on $\Ch^0, \Ch^1$.
\end{Proposition}

\begin{proof}
Note that if $d_0^*\, \bar\alpha = 0$ then for any $f \in \Ch^0$ we have
\begin{eqnarray*}
\| \bar\alpha + d_0 f \|^2
&=&
\| \bar\alpha \|^2 + 2 \pairing{\bar\alpha}{d_0 f} + \| d_0 f \|^2
\\
&=&
\| \bar\alpha \|^2 + 2 \pairing{d_0^*\,\bar\alpha}{f} + \| d_0 f \|^2
\\
&=& 
\| \bar\alpha \|^2 + \| d_0 f \|^2
\end{eqnarray*}
which implies that such an $\bar\alpha$ must be the unique minimizer.
For existence, note that
\[
d_0^*\, \alpha + d_0^*\,d_0 f = 0
\]
certainly has a solution~$f$ if $\im(d_0^*) = \im(d_0^*\,d_0)$. But this is a standard fact in finite-dimensional linear algebra: $\im(A^\text{\sc t}) = \im(A^\text{\sc t} A)$ for any real matrix~$A$; this follows from the singular value decomposition, for instance.
\end{proof}

\smallskip
\textsc{Remark.}
It is customary to construct the Laplacian $\Delta = d_1^*\,d_1 + d_0\,d_0^*$. The twin equations $d_1\,\bar\alpha = 0$ and $d_0^*\,\bar\alpha=0$ immediately imply (and conversely, can be deduced from) the single equation $\Delta \bar\alpha =0$; in other words $\bar\alpha$ is \emph{harmonic}.
\smallskip

\subsection{Integration}
\label{sec:integration}

The least-squares problem in equation~\eqref{eq:lsqr} can be solved using a standard algorithm such as LSQR~\cite{Paige_Saunders_1982a}. By Proposition~\ref{prop:real} we can use the solution parameter~$f$ to define the circular coordinate $\theta$ on the vertices of~$X$. This works because the original cocycle $\alpha$ has integer coefficients.

More generally, if $\bar\alpha$ is an arbitrary real cocycle such that $[\bar\alpha]  \in \im(\Hh^1(X; \Zz) \to \Hh^1(X; \Rr))$, it is a straightforward matter to integrate $\bar\alpha$ to a circle-valued function $\theta$ on the vertex set~$X^0$. Suppose that $X$ is connected (if not, each connected component can be treated separately) and pick a starting vertex~$x_0$ and assign~$\theta(x_0)=0$. One can use Dijkstra's algorithm to find shortest paths to each remaining vertex from~$x_0$. When a new vertex $b$ enters the structure via an edge $ab$, we assign $\theta(b) = \theta(a) + \bar\alpha(ab)$ (or $\theta(a) - \bar\alpha(ba)$ if the edge is correctly identified as $ba$). If a vertex $a$ is connected to $x_0$ by multiple paths then the different possible values of $\theta(a)$ differ by an integer; this is where we use the hypothesis that $\bar\alpha$ is cohomologous to an integer cocyle.

\section{Experiments}
\label{sec:examples}

\subsection{Software}
The following experiments were carried out using the Java-based jPlex simplicial complex software~\cite{jPlex_2008}, with high-level scripting and numerical analysis in MATLAB. We ran a development version of jPlex to obtain explicit persistent cohomology cocycles. We expect to include the code in the next release of jPlex. We used Paige and Saunders' implementation of LSQR~\cite{LSQR_2007} for the least-squares problem in the harmonic smoothing step. Timings were determined using MATLAB's built-in `tic' and `toc' commands, and are included for relative comparison against each other.

\subsection{General procedure}

We tested our methods on several synthetic data sets with known topology, ranging from the humble circle itself to a genus-2 surface (`double torus'). Most of the examples were embedded in $\Rr^2$ or~$\Rr^3$, with the exception of a sample from a complex projective curve (embedded in $\Cc P^2$) and a synthetic image-like data set (embedded in $\Rr^{120000}$).

In each case we selected vertices for the filtered simplicial complex: either the whole set, or a smaller well-distributed subset of `landmarks' selected by iterative furthest-point sampling. We then built a Rips or witness complex, with maximum radius generally chosen to ensure around $10^5$ simplices in the complex.

In most cases, we show the persistence diagram produced by the cocycle computation. The chosen value~$\delta$ is marked on the diagonal, with its upper-left quadrant indicated in green lines. The persistent cocycles available at that parameter value are precisely those contained in that quadrant. Each of those cocycles then produces a circular coordinate.

There are various figures associated with each example. Most important are the correlation scatter plots: each scatter plot compares two circular coordinate functions. These may be functions produced by the computation (`inferred coordinates') or known parameters. These scatter plots are drawn in the unit square, which is of course really a torus $S^1 \times S^1$.

When the original data are embedded in $\Rr^2$ or~$\Rr^3$, we also display the circular coordinates directly on the data set, plotting each point in color according to its coordinate value interpreted on the standard hue-circle. This works less well in grayscale reproductions, of course.

Finally, in certain cases we plot coordinate values against frequency, as a histogram. This distributional information can sometimes be useful in the absence of other information.

\smallskip
\textsc{Remark.} When the goal is to infer the topology of a data set whose structure is unknown, we do not have any `known parameters' available to us. We can still construct correlation scatter plots between pairs of inferred coordinates, and the distributional histograms for each coordinate individually. We exhort the reader to view the following examples through the lens of the topological inference problem: what structures can be distinguished using scatter plots and histograms (and persistence diagrams) alone?
\smallskip

\subsection{Noisy circle}
\label{sec:noisy-circle}

We begin with the circle itself, and its tautological circle-valued coordinate. 

We picked 400 points distributed along the unit circle. We added a uniform
random variable from $[0.0,0.4]$ to each coordinate. A Rips complex was
constructed with maximal radius 0.5, resulting in 23475 simplices. The
computation of cohomology finished in 237 seconds.

Parametrizing at 0.4 yielded a single coordinate function, which very closely reproduces the tautological angle function. Parametrizing at 0.14 yielded several possible cocycles. We selected one of those with low persistence; this produced a parametrization which `snags' around a small gap in the data.

See Figure~\ref{fig:noisyCircle}.
\begin{figure*}
  \centering
  \begin{minipage}{0.3\linewidth}
    \includegraphics[width=\linewidth]{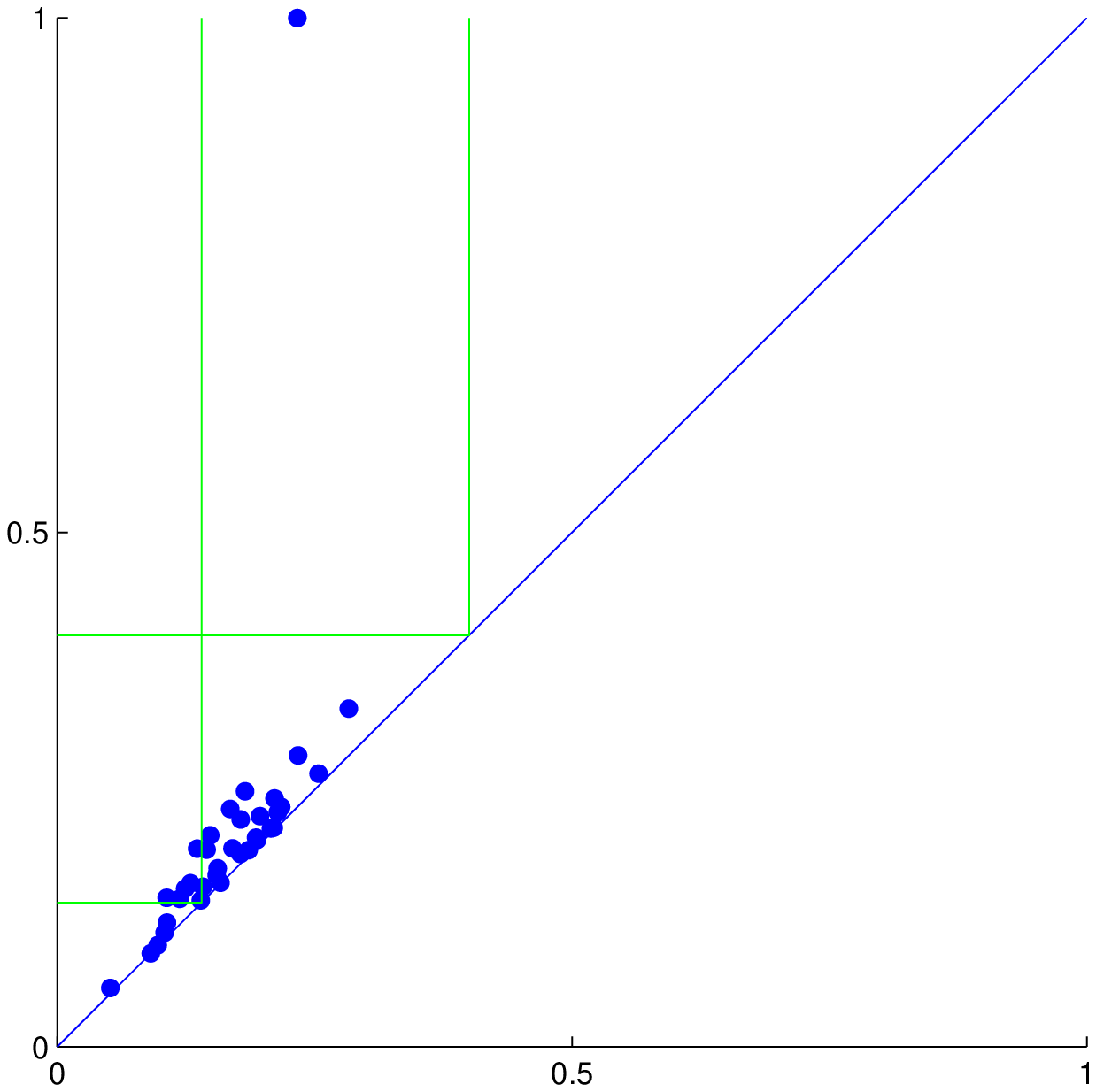}
  \end{minipage}
  \begin{minipage}{0.65\linewidth}
    \includegraphics[width=0.3\linewidth]{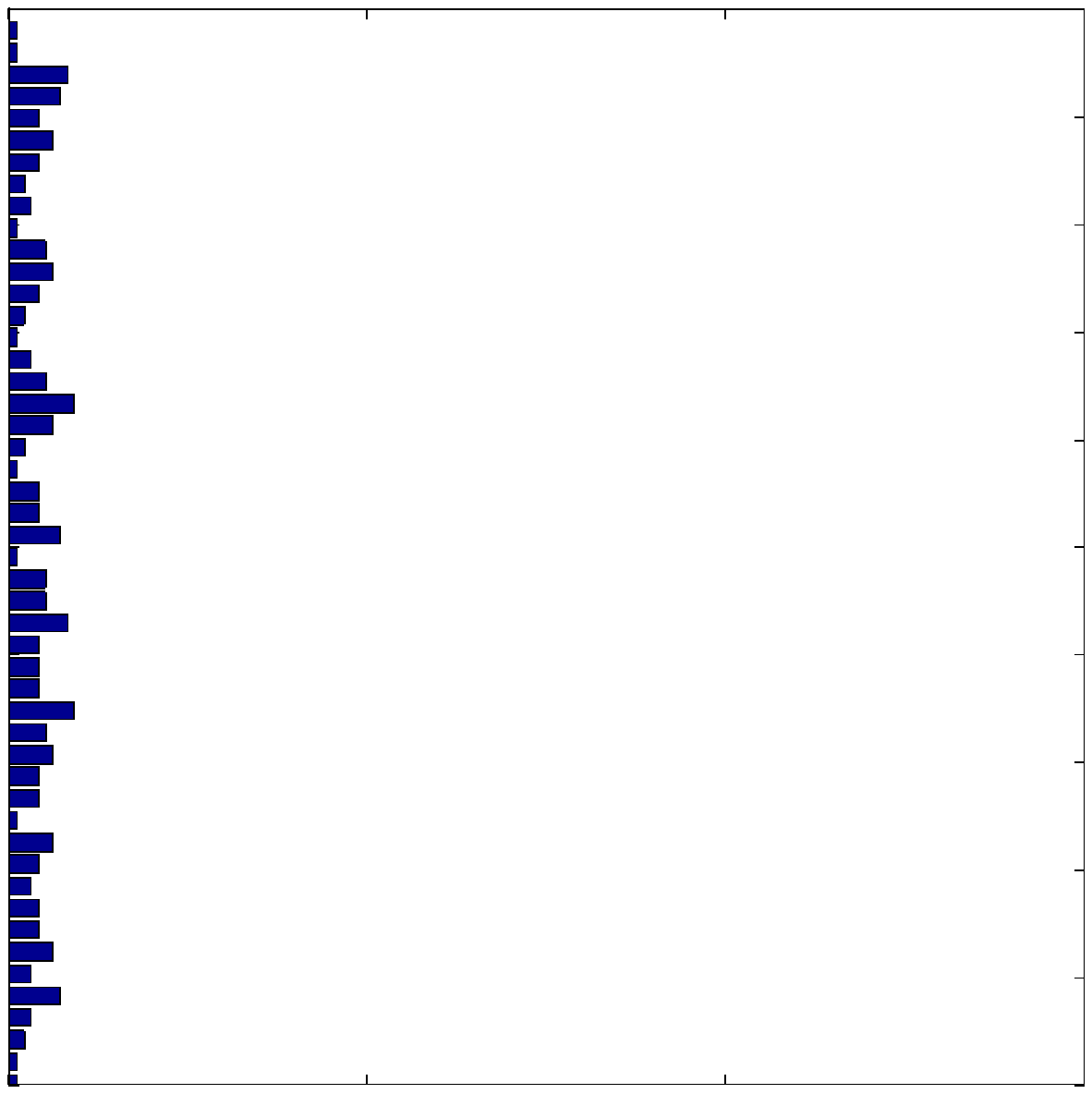}
    \includegraphics[width=0.3\linewidth]{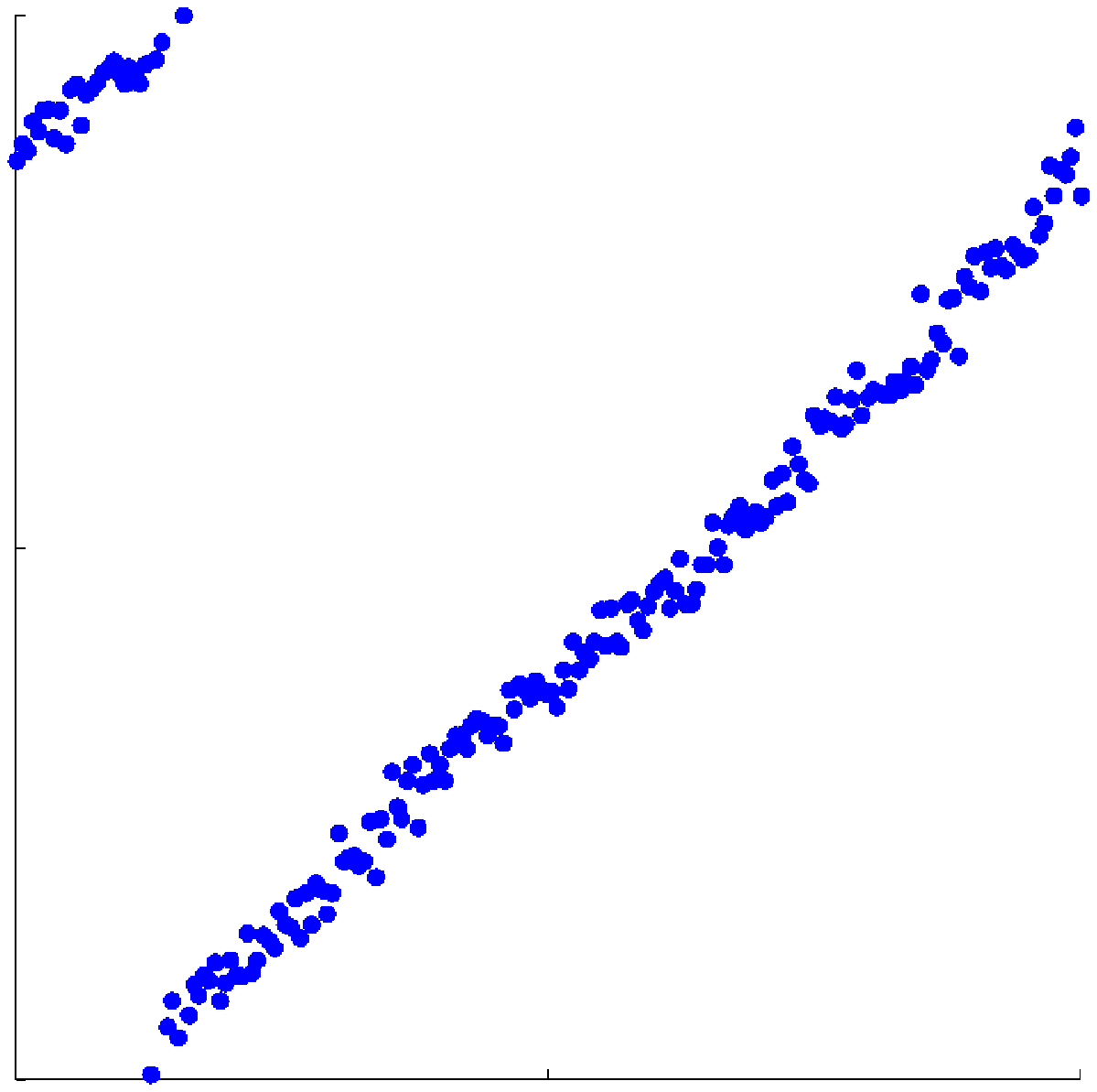}
    \includegraphics[width=0.33\linewidth]{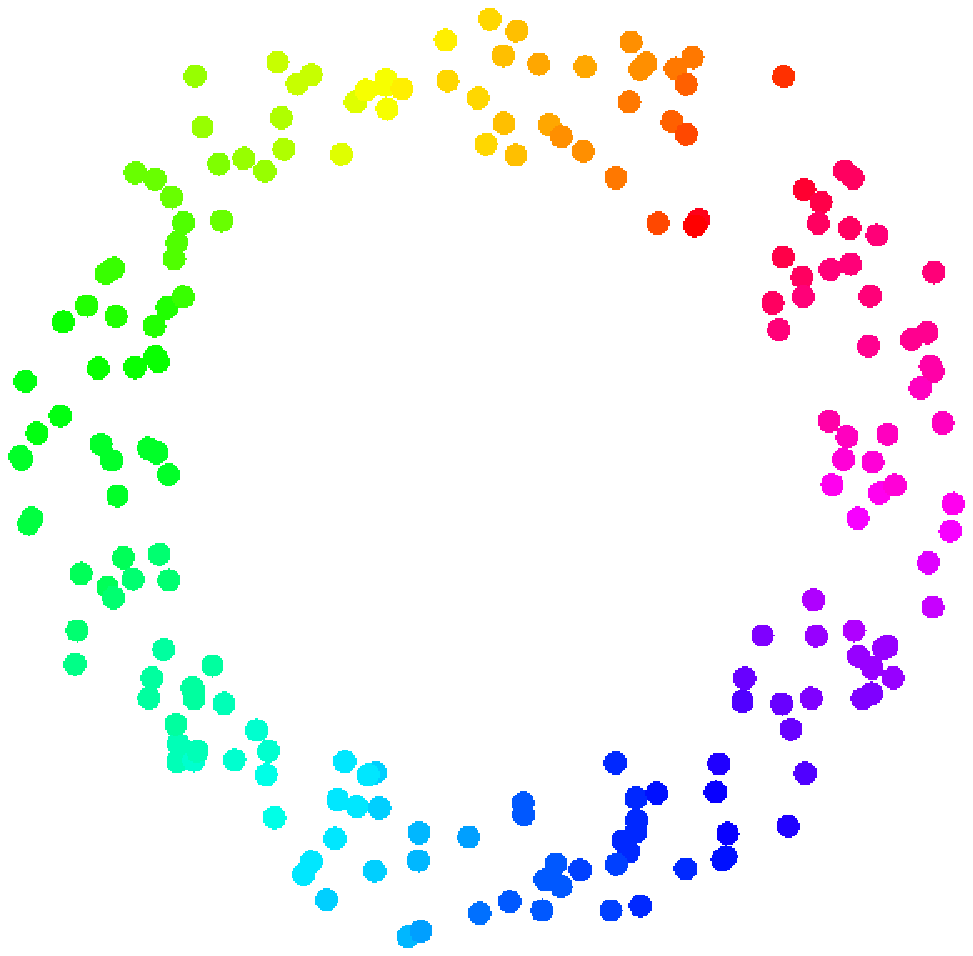}
    \\
    \includegraphics[width=0.3\linewidth]{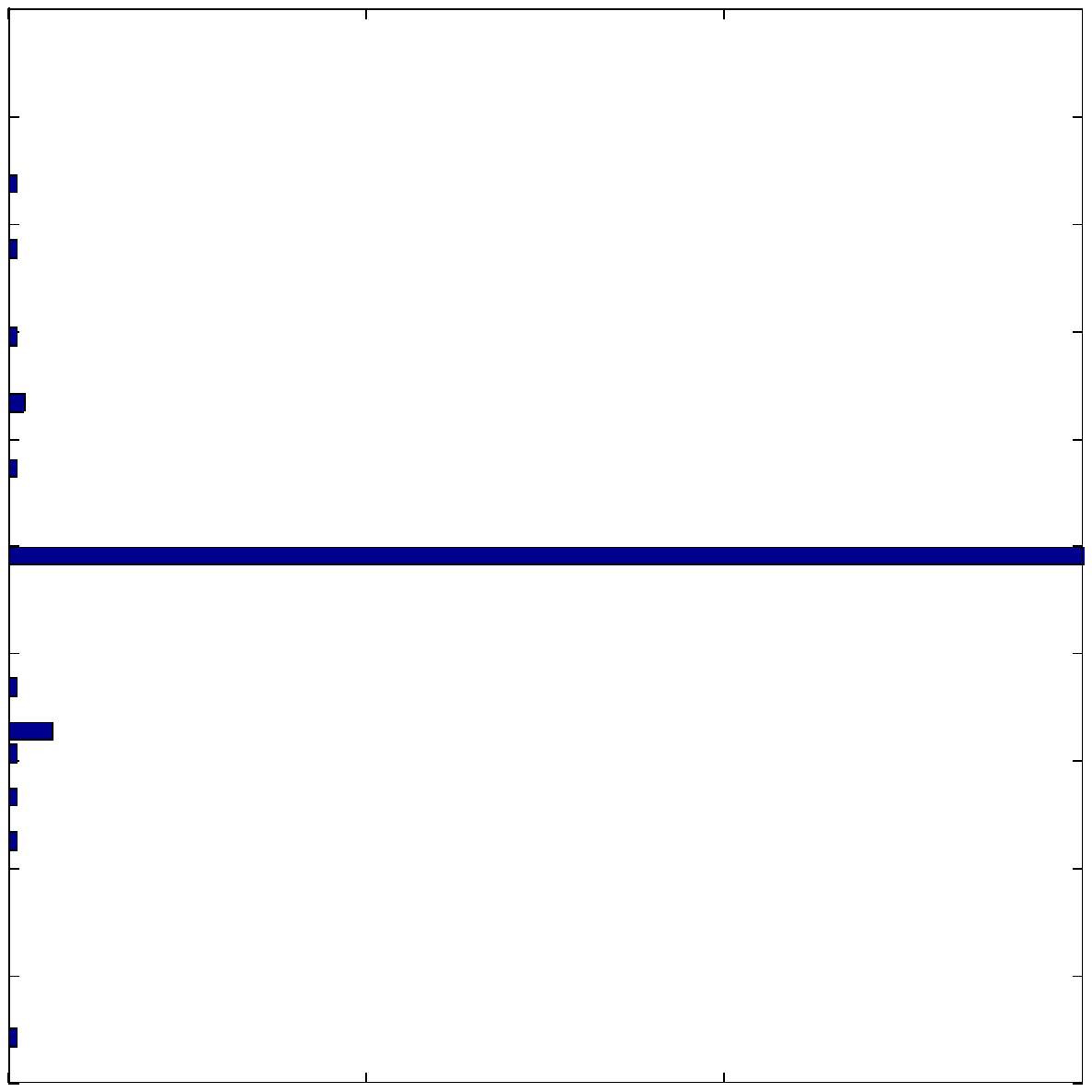}
    \includegraphics[width=0.3\linewidth]{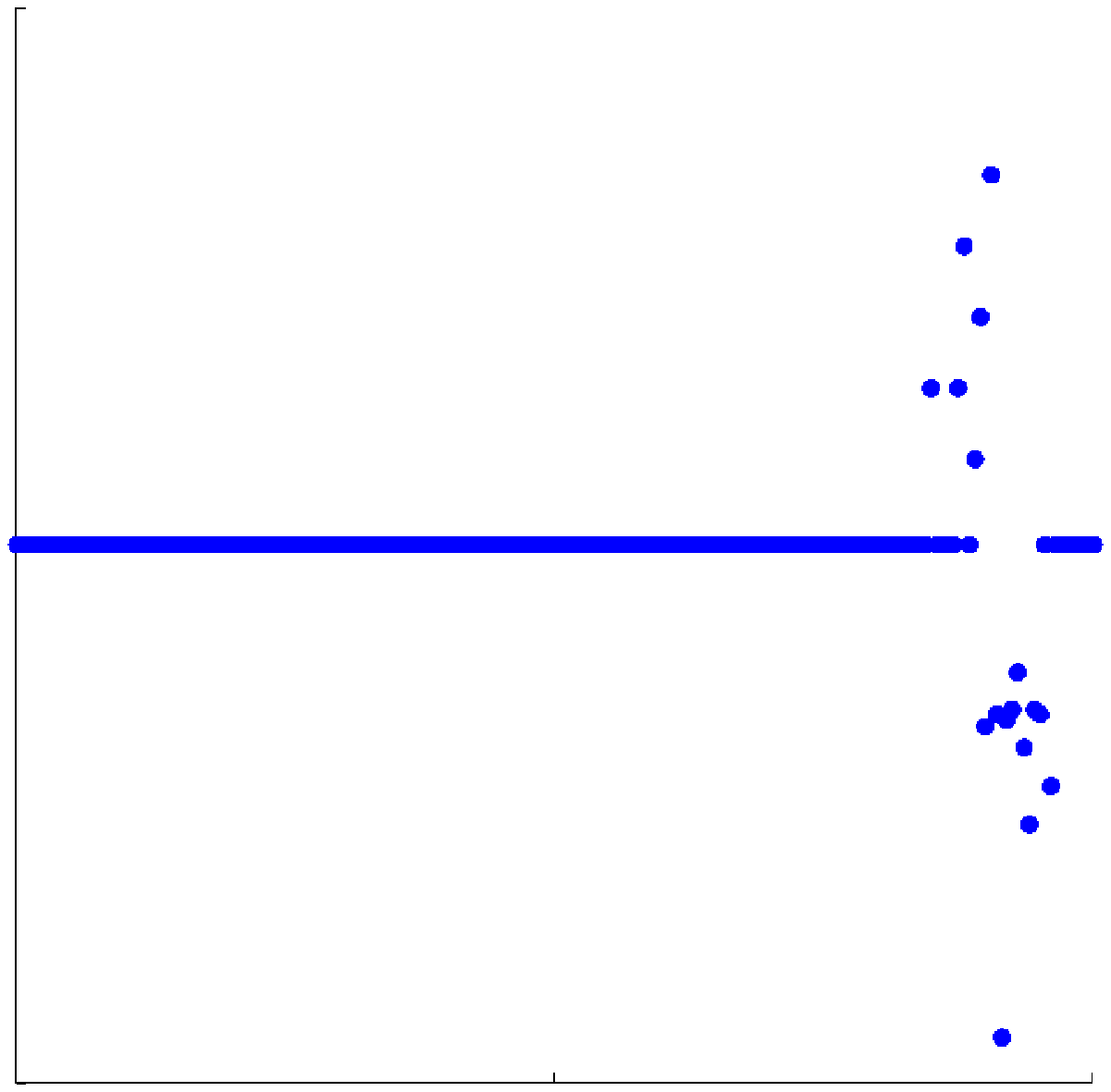}
    \includegraphics[width=0.33\linewidth]{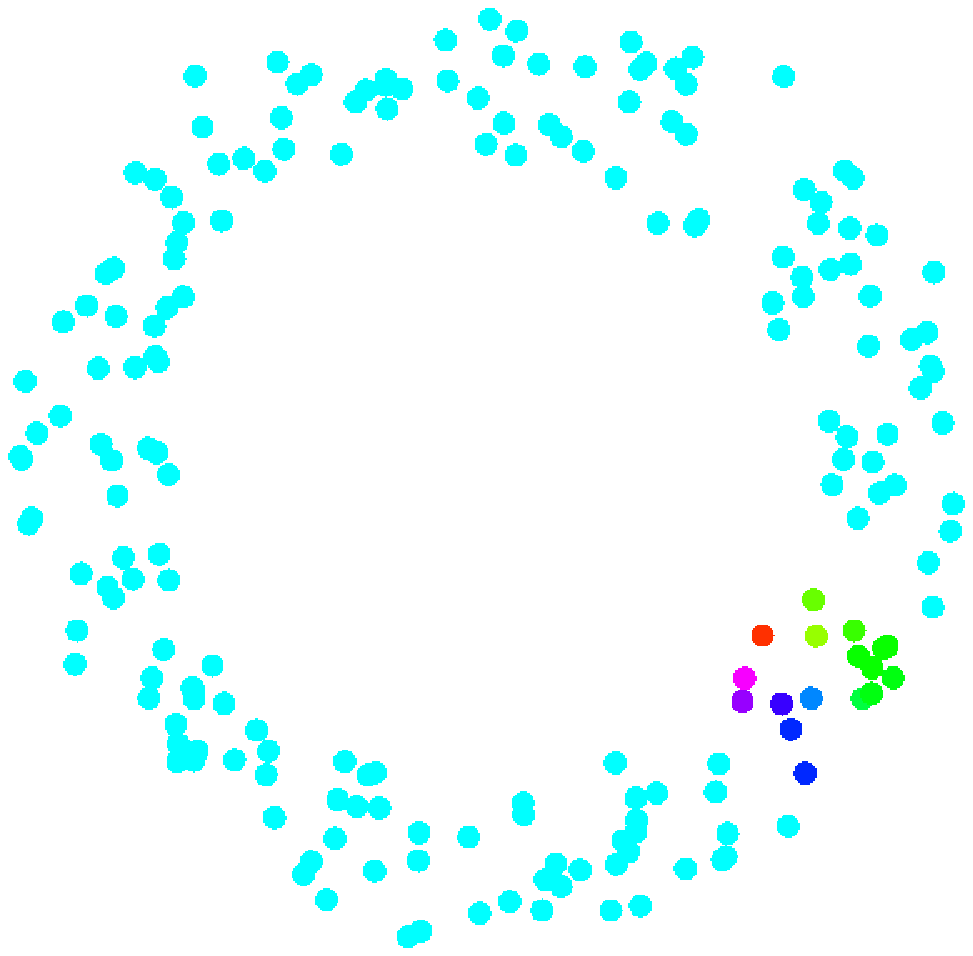}
  \end{minipage}
  \caption{Noisy circle. Persistence diagram (left). Global coordinate (top row), local coordinate (bottom row). In each row: histogram of coordinate values (left), correlation scatter plot against known angle function (middle), inferred coordinate in color (right).}
  \label{fig:noisyCircle}
\end{figure*}
The left panel in each row shows the histogram of coordinate values; the middle panel shows the correlation scatter plot against the known angle function; the right panel displays the coordinate using color. The high-persistence (`global') coordinate correlates with the angle function with topological degree~1. Variation in that coordinate is uniformly distributed, as seen in the histogram. In contrast, the low-persistence (`local') coordinate has a spiky distribution.

\subsection{Trefoil torus knot}
\label{sec:trefoil-torus-knot}

Another example with circle topology: see Figure~\ref{fig:trefoil}.
We picked 400 points distributed along the $(2,3)$ torus knot on a torus
with radii 2.0 and 1.0. We jittered them by a uniform random variable
from $[0.0,0.2]$ added to each coordinate. We generated a Rips complex up to radius~1.0, acquiring 36936 simplices. We computed persistent
cohomology in 70 seconds.
\begin{figure*}
  \centering
  \includegraphics[width=0.25\linewidth]{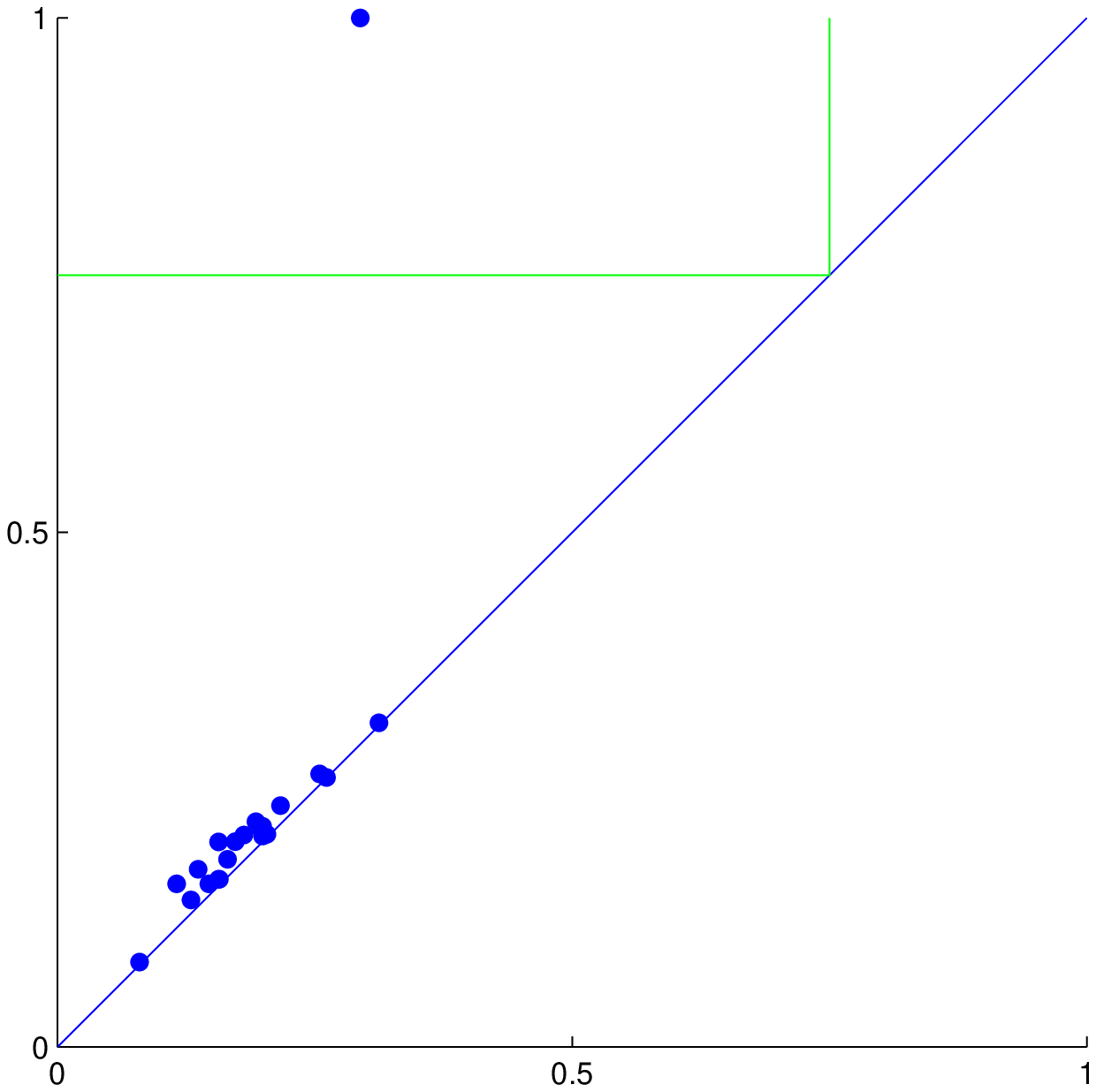}
  \includegraphics[width=0.25\linewidth]{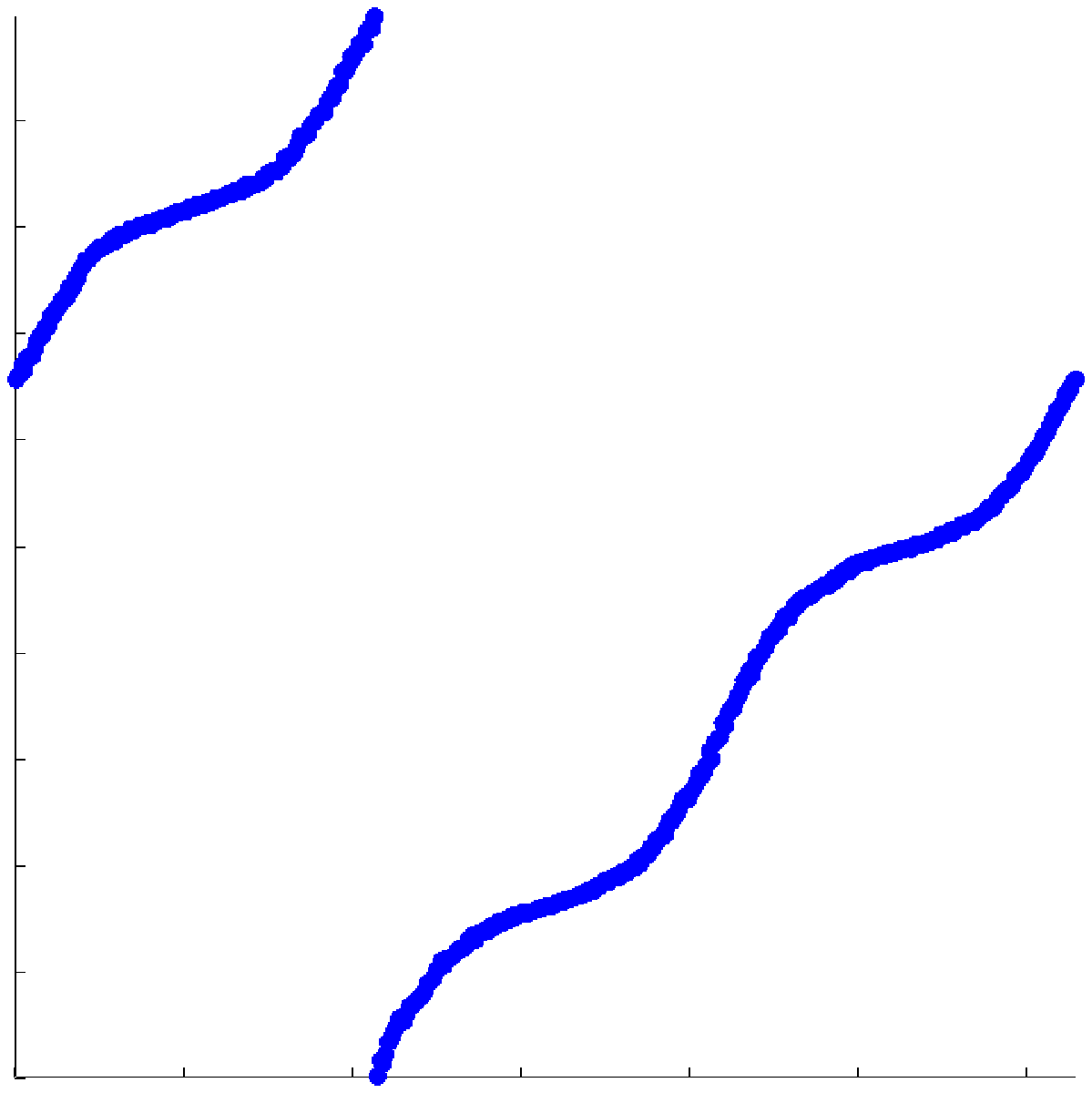}
  \includegraphics[width=0.3\linewidth]{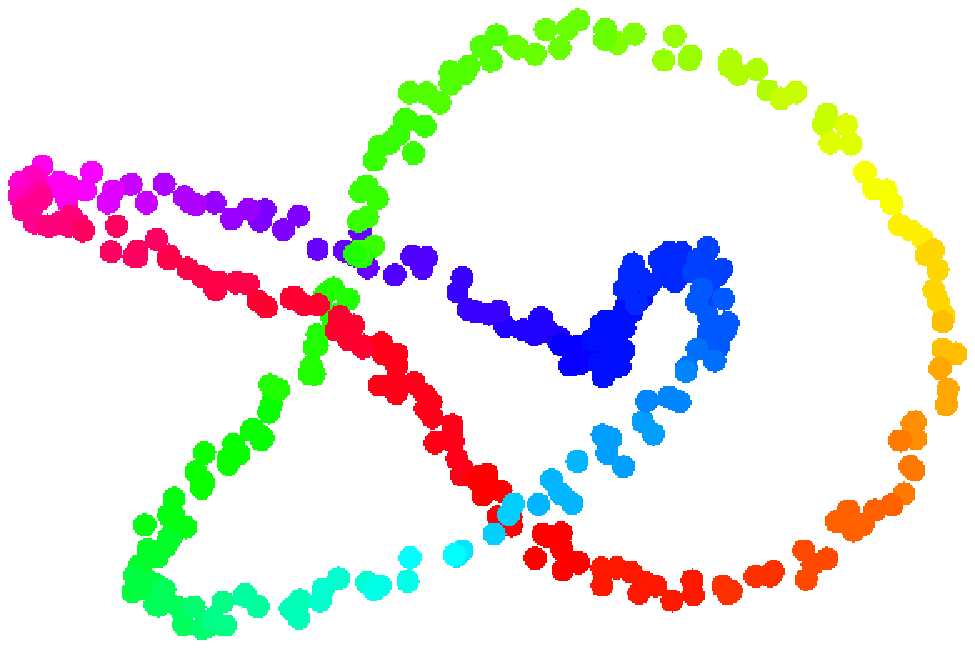}
  \caption{Trefoil torus knot. Persistence diagram (left), correlation scatter plot of inferred coordinate against known parametrization (middle), inferred coordinate in color (right).}
  \label{fig:trefoil}
\end{figure*}
As expected, the inferred coordinate correlates strongly with the known parameter with topological degree~1. The histogram shows three `bulges' corresponding to the three high-density regions of the sampled curve, which occur when the curve approaches the central axis of the torus.

\subsection{Rotating cube}
\label{sec:rotatingCube}

For a more elaborate data set with $S^1$-topology, we generated a sequence of 657 rendered images of a colorful cube rotating around one axis. Each image was regarded as a vector in the Euclidean space $\mathbb R^{200\cdot200\cdot3}$. From this data we built a witness complex with 50 landmark points and constructed a single circular coordinate.
Interpolating the resulting function linearly between the landmarks gave us coordinates for all the points in the family.

\begin{figure*}
\centering
\parbox{0.90\linewidth}{
	\raisebox{0.45in}{
	\includegraphics[width=0.25\linewidth]{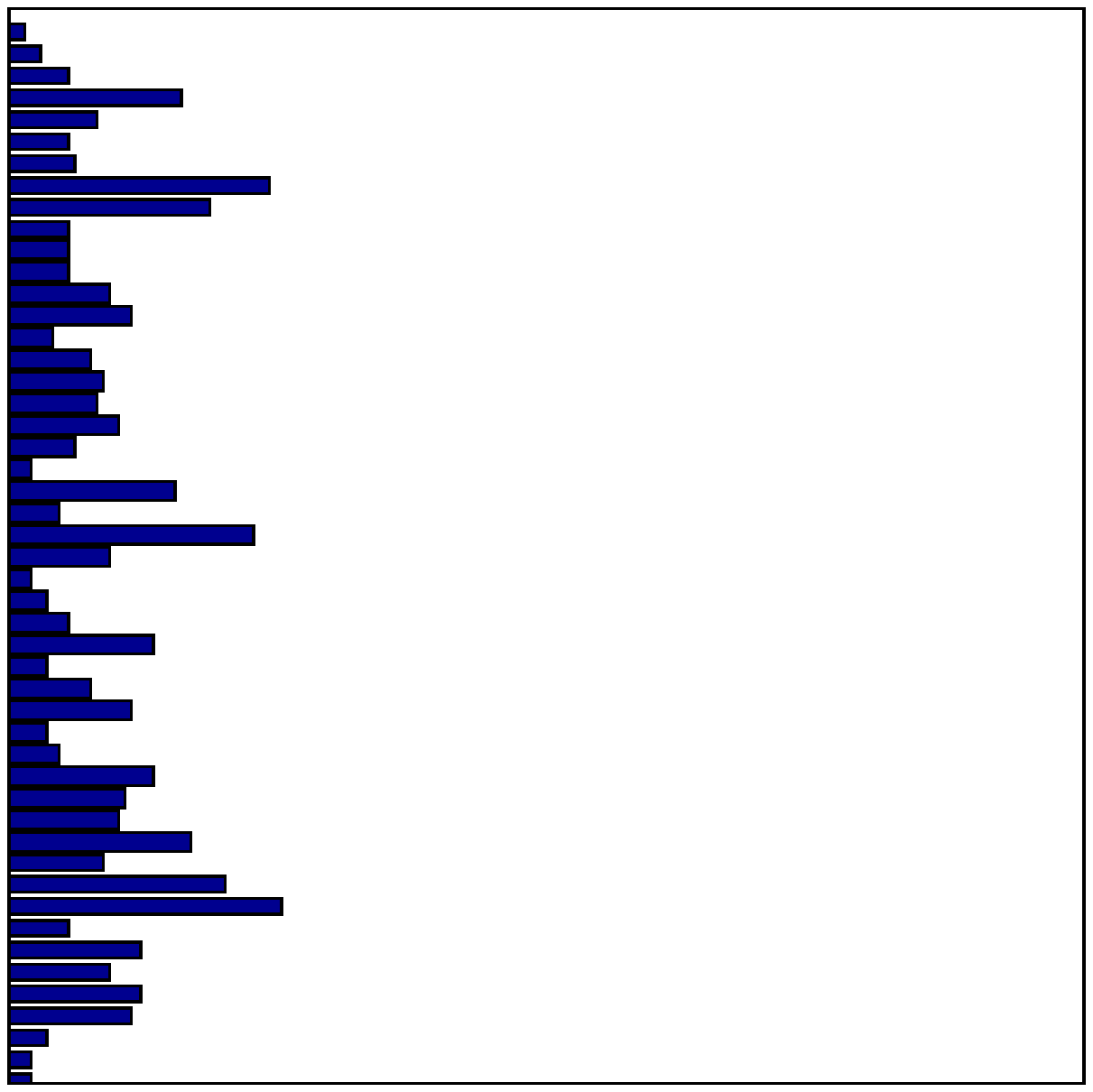}
	}
	\hfill
	\raisebox{0.45in}{
	\includegraphics[width=0.25\linewidth]{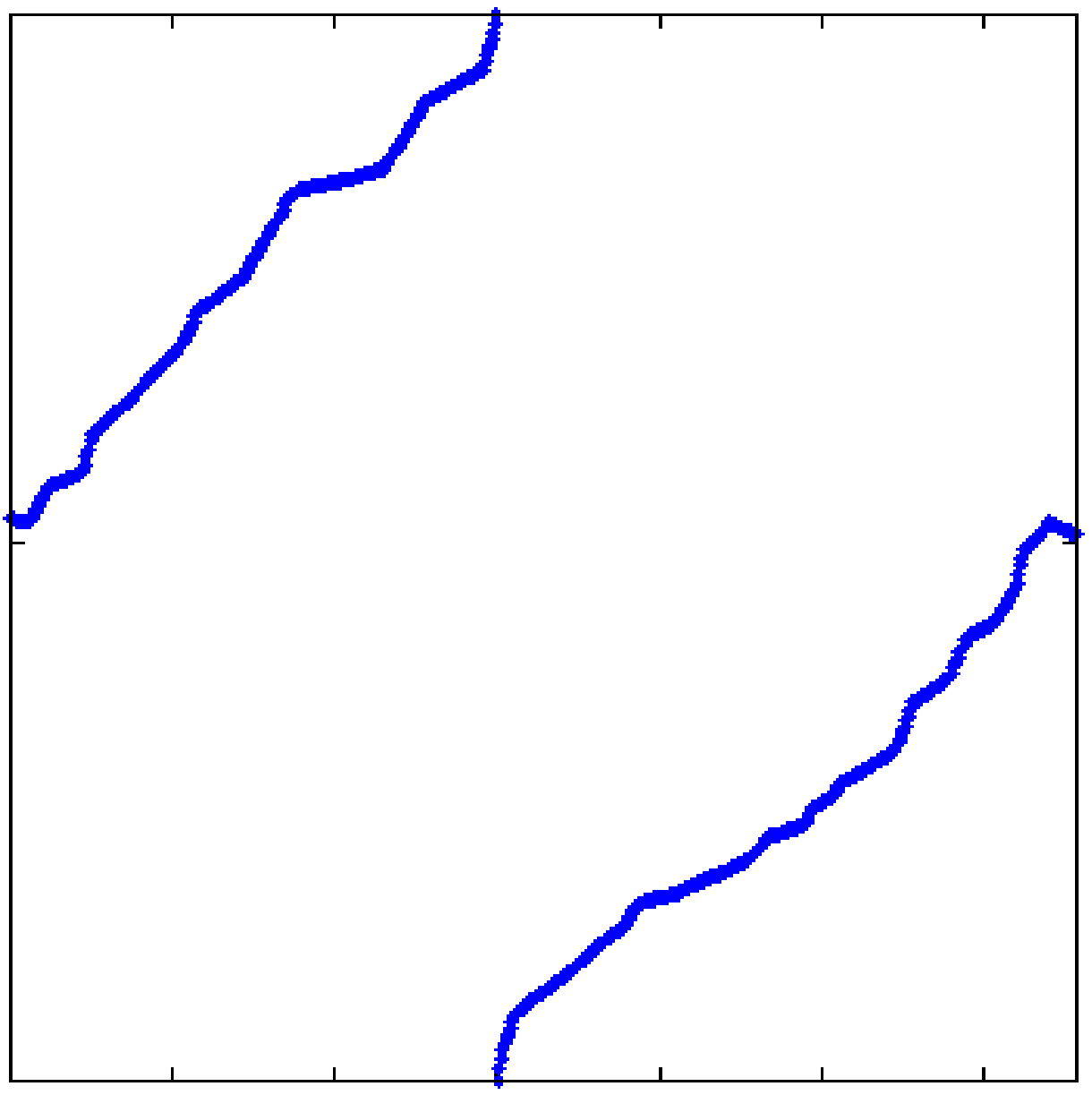}
	}
	\hfill
	\includegraphics[width=0.4\linewidth]{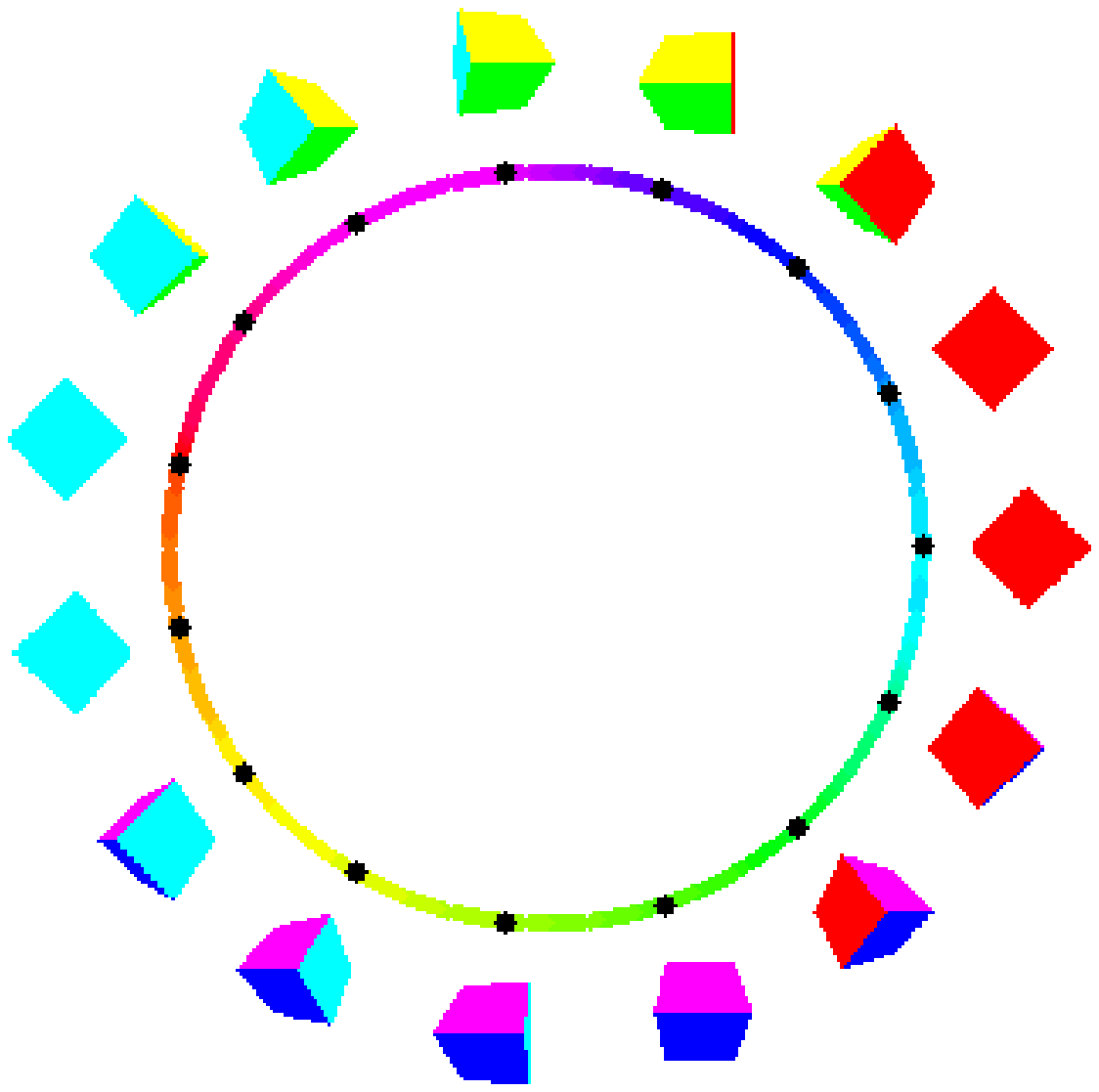}
}
\caption{Images of a rotating cube. Histogram of coordinate values (left); scatter plot against known angle function (middle); a selection of images matched to recovered circle coordinate (right).}
\label{fig:cube}
\end{figure*}
See Figure~\ref{fig:cube}. The frequency distribution is comparatively smooth (by which we mean that there are no large spikes in the histogram), which indicates that the coordinate does not have large static regions. The correlation plot of the inferred coordinate against the original known sequence of the cube images shows a correlation with topological degree~1. We show the progression of the animation on an evenly-spaced sample of representative points around the circle.

\subsection{Pair of circles}
\label{sec:wedge-two-circles}

See Figure~\ref{fig:wedgedisjoint} for these two examples.

Conjoined circles: we picked 400 points distributed along circles in the plane with radius 1 and with centres at $(\pm 1,0)$. The points were then jittered by
adding noise to each coordinate taken uniformly randomly from the interval $[0.0,0.3]$. A Rips complex was constructed with maximal radius 0.5, resulting in
76763 simplices. The cohomology was computed in 378 seconds.

Disjoint circles: 400 points were distributed on circles of radius 1 centered around
$(\pm 2, 0)$ in the plane. These points were subsequently disturbed by
a uniform random variable from $[0.0,0.5]$. We constructed a Rips
complex with maximum radius 0.5, which gave us 45809 simplices. The
cohomology computation finished in about 117 seconds.

\begin{figure*}
\centering
\begin{minipage}{0.35\linewidth}
\centering
\begin{tabular}{m{0.49\linewidth}m{0.49\linewidth}}
  \includegraphics[width=\linewidth]{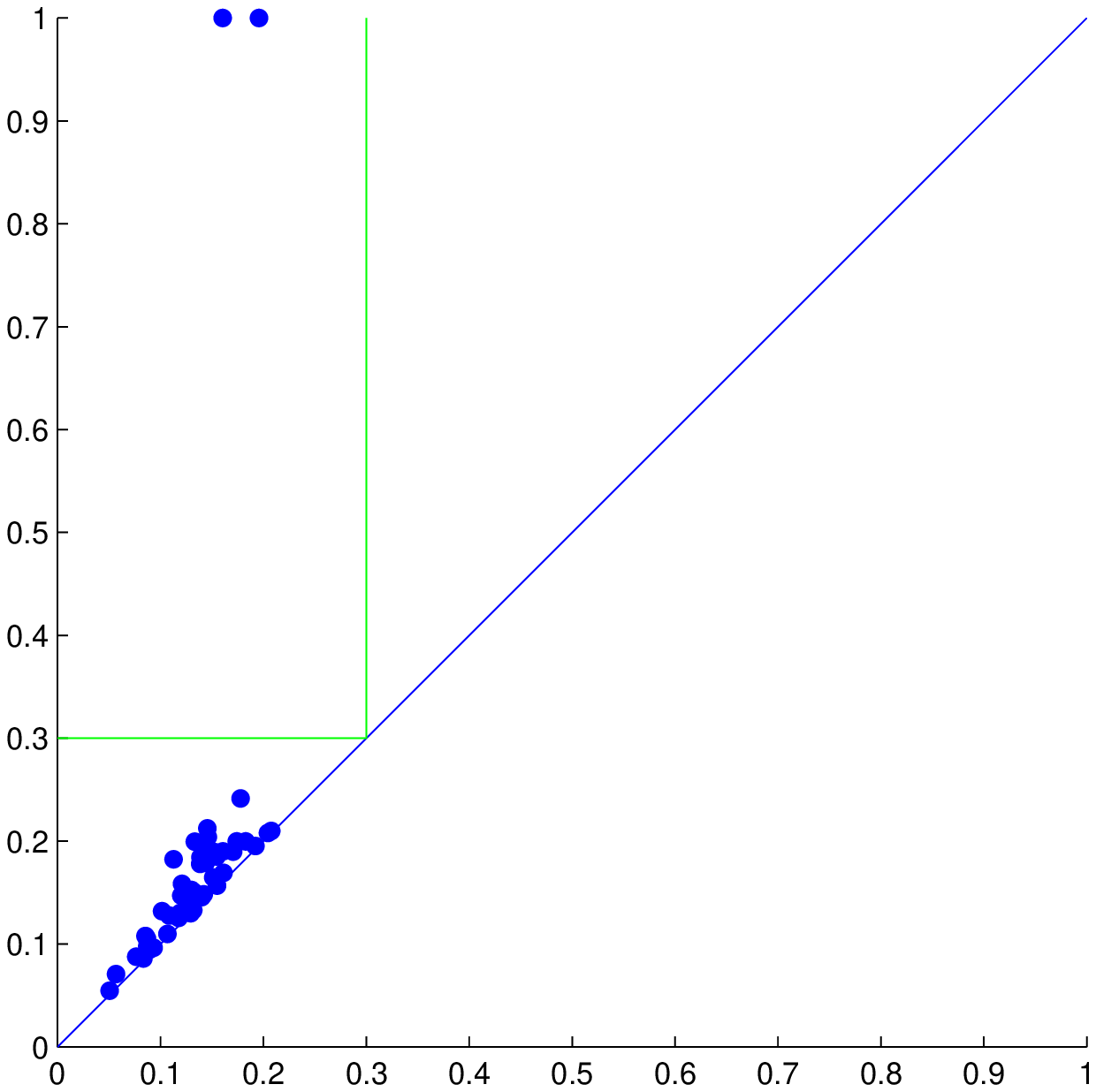}
  &\includegraphics[height=0.5\linewidth]{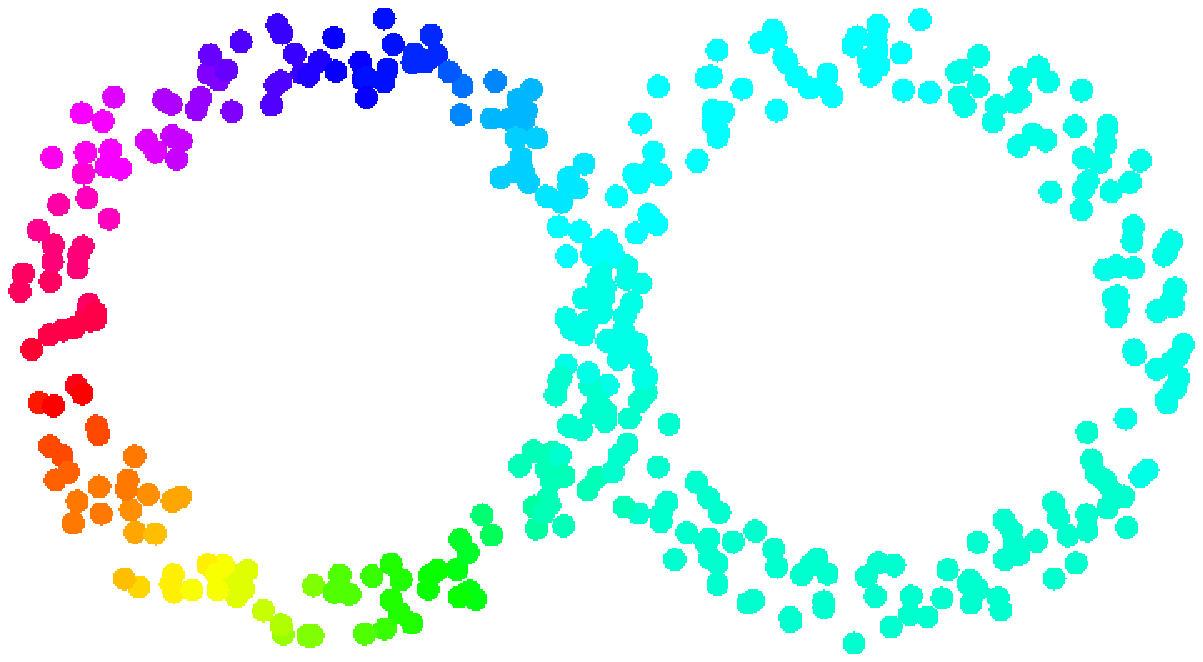}
  \\
  \includegraphics[width=\linewidth]{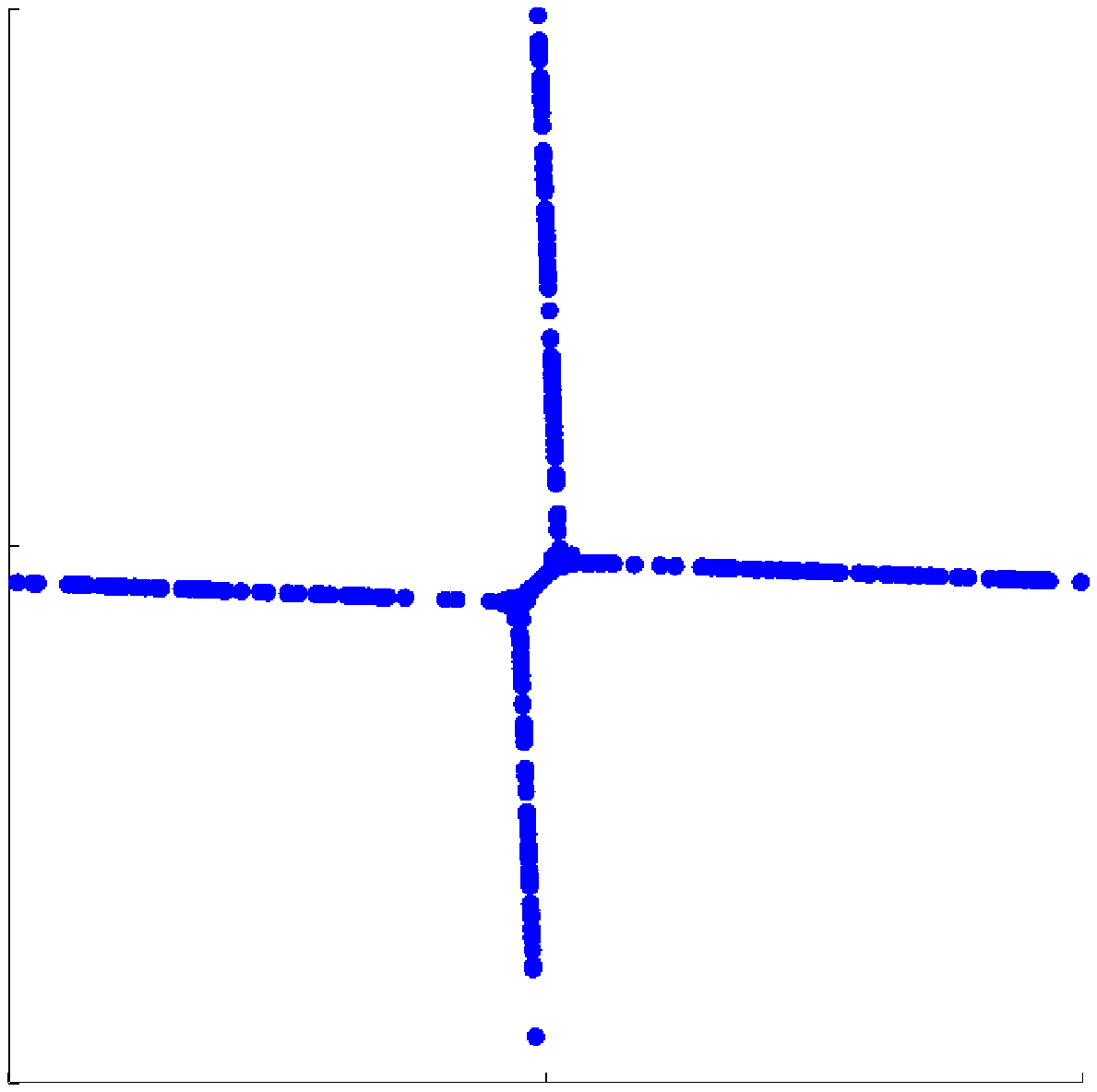}
  &\includegraphics[height=0.5\linewidth]{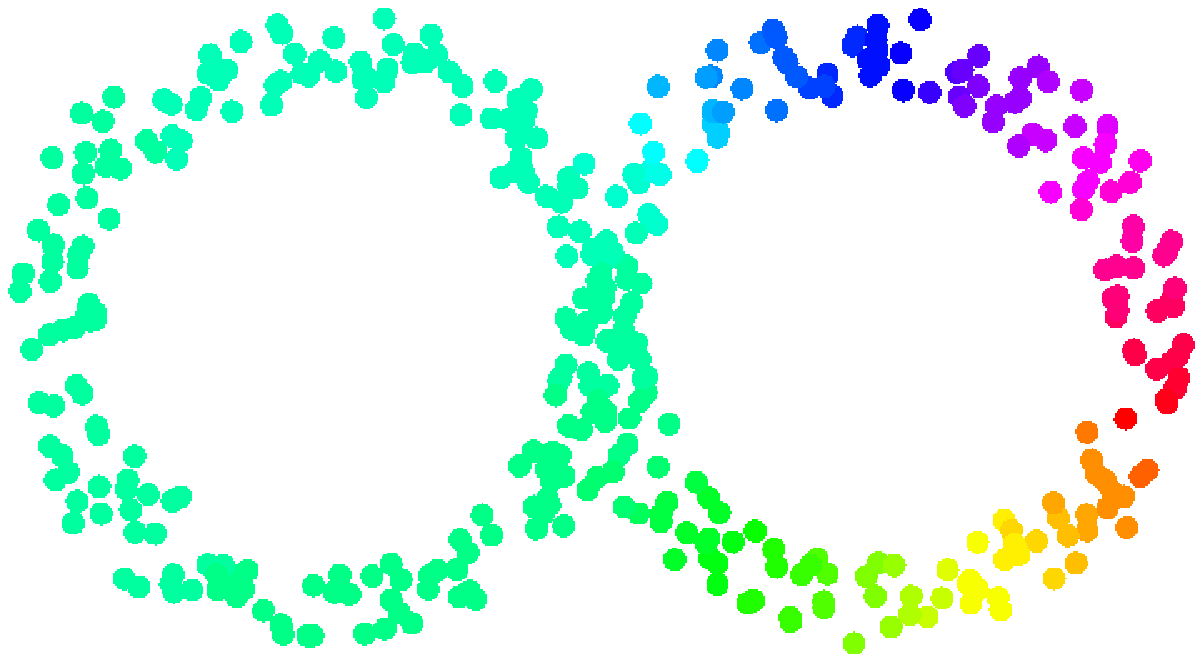}
\end{tabular}
\end{minipage}
\hspace{0.1\linewidth}
\begin{minipage}{0.525\linewidth}
\centering
\begin{tabular}{m{0.33\linewidth}m{0.66\linewidth}}
  \includegraphics[width=\linewidth]{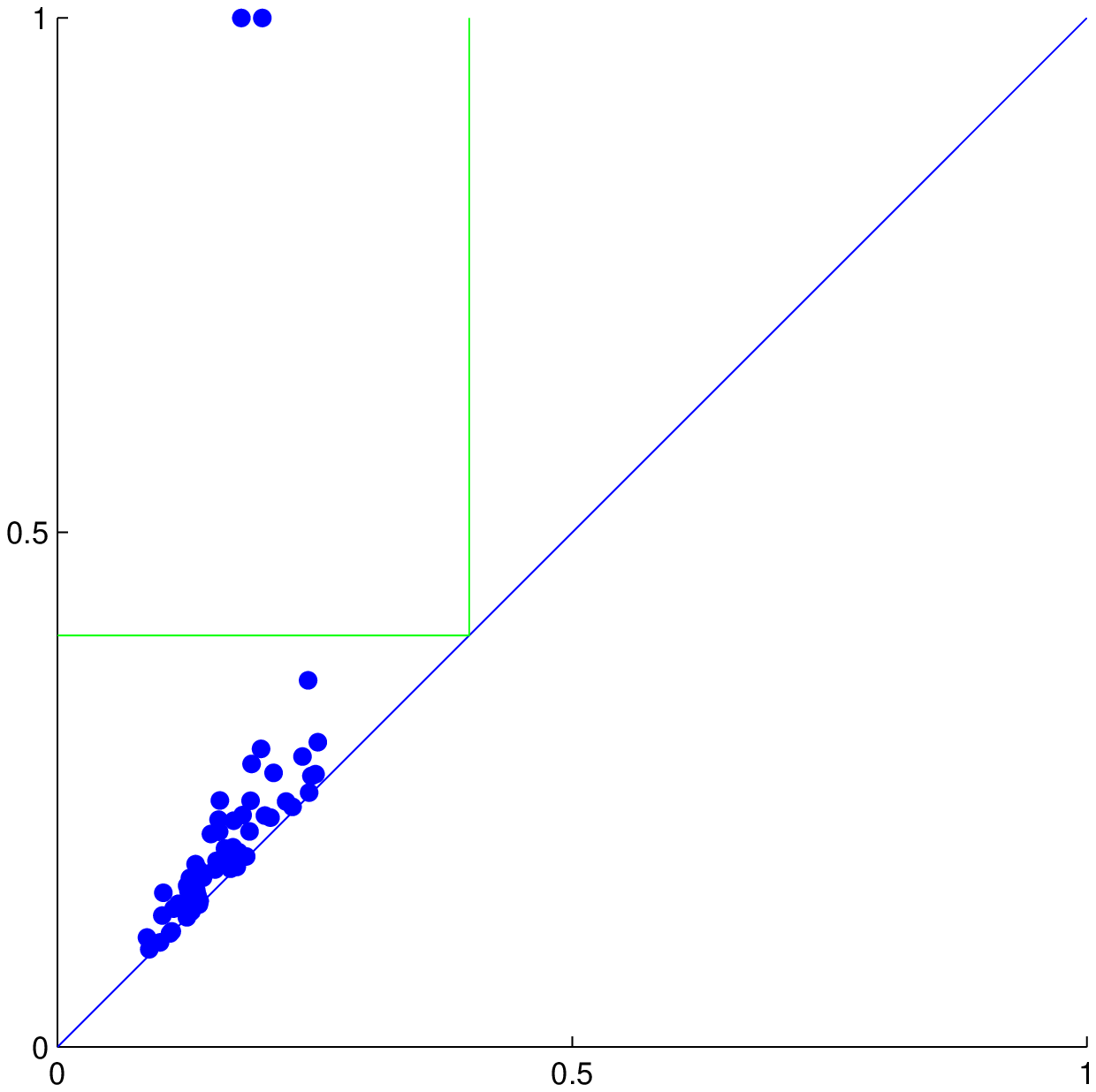}
  &\includegraphics[height=0.275\linewidth]{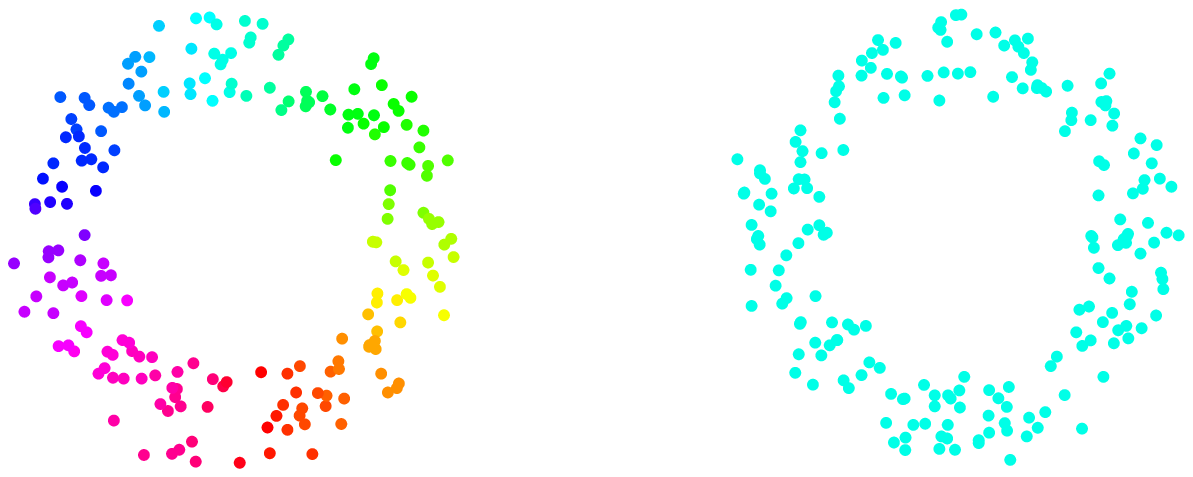}
  \\
  \includegraphics[width=\linewidth]{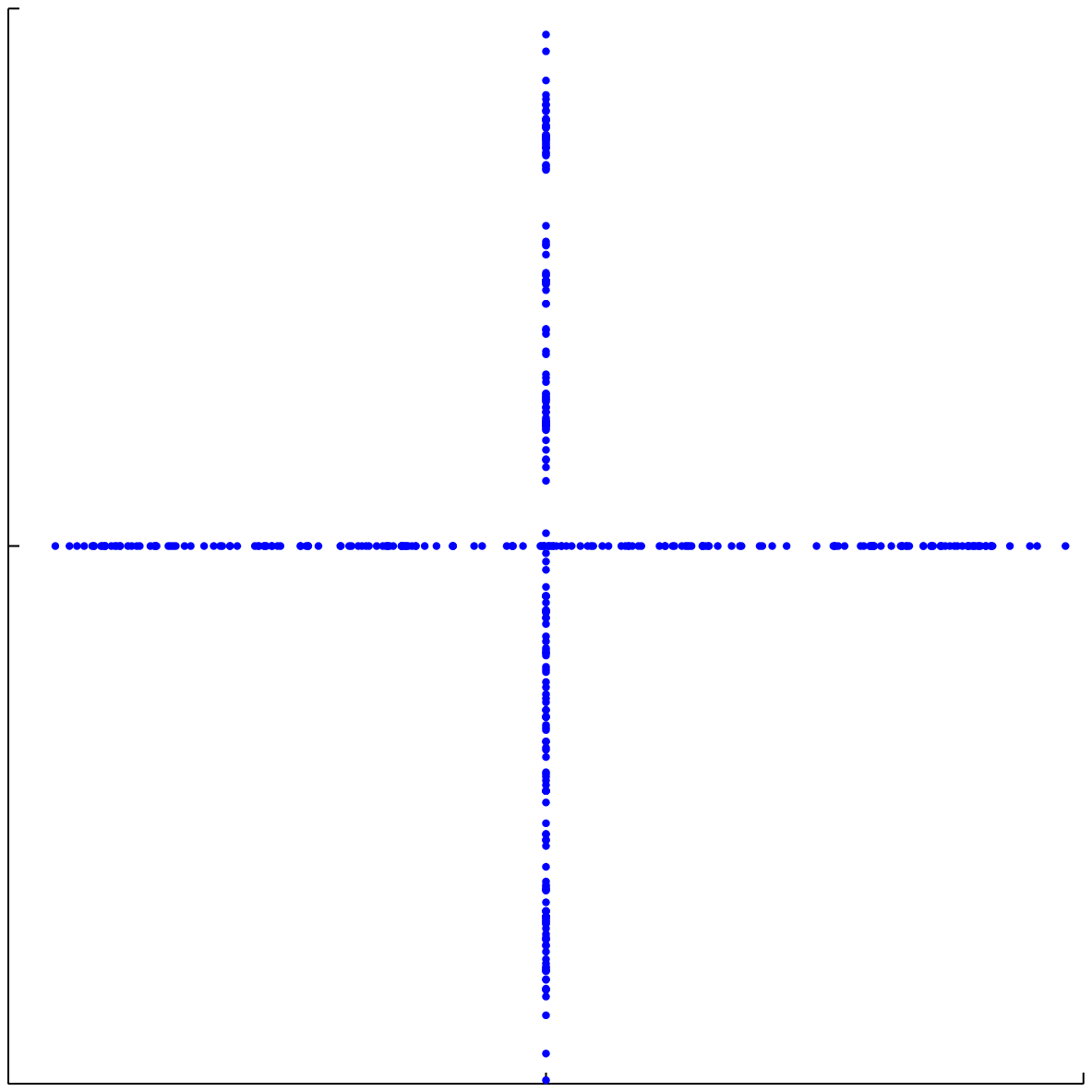}
  &\includegraphics[height=0.275\linewidth]{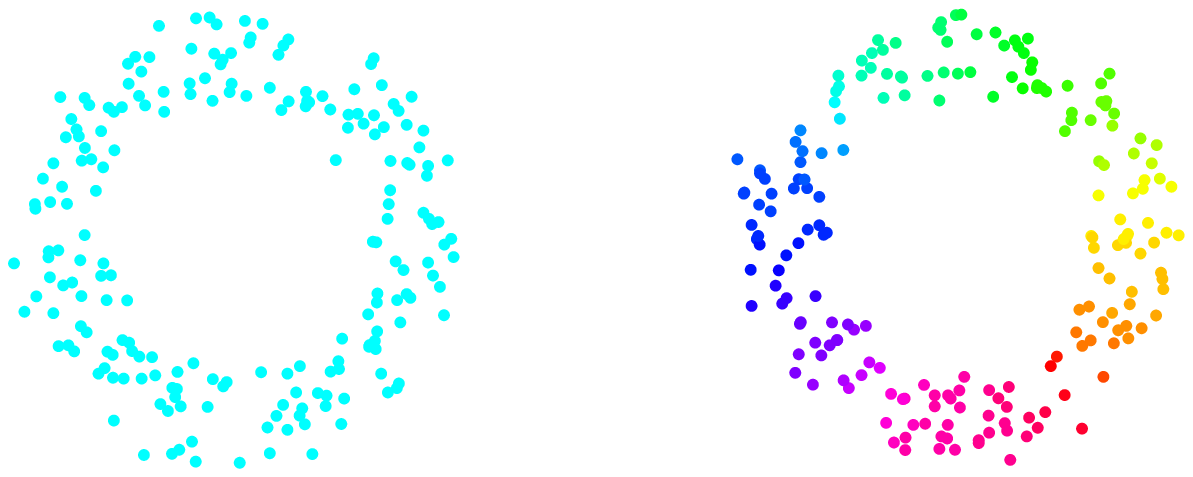}
\end{tabular}
\end{minipage}
\caption{Two conjoined circles (left); two disjoint circles (right). In each case we show the persistence diagram (top left),  the two inferred coordinates (right column), the correlation scatter plot (bottom left).}
\label{fig:wedgedisjoint}
\end{figure*}

In both cases, our method detects the two most natural circle-valued functions. The scatter plots appear very similar. In the conjoined case, there is some interference between the two circles, near their meeting point.

\subsection{Torus}
\label{sec:torus}

See Figure~\ref{fig:torus}. We picked 400 points at random in the unit square, and then used a standard parametrization to map the points onto a torus with inner and outer radii 1.0 and 3.0. These were subsequently jittered by adding a uniform
random variable from $[0.0,0.2]$ to each coordinate. We constructed a Rips
complex with maximal radius $\sqrt3$, resulting in 61522 simplices. The corresponding cohomology was computed in 209 seconds.

\begin{figure*}
\centering
\begin{minipage}{0.8\linewidth}
\centering{
  \subfigure[Persistence diagram (left); first inferred coordinate (middle); second inferred coordinate (right).]
{
  \includegraphics[width=0.25\linewidth]{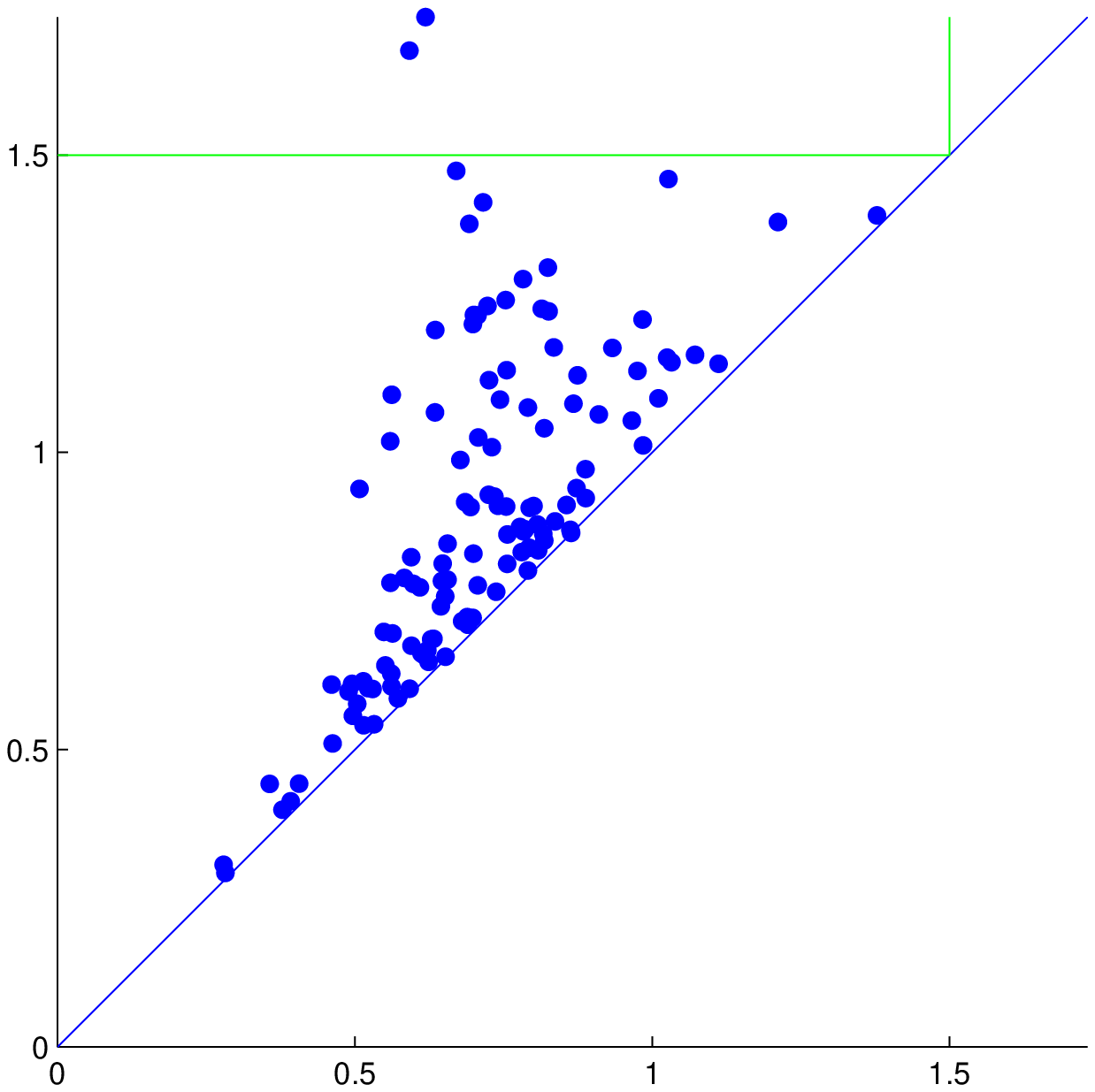} \qquad
  \includegraphics[width=0.3\linewidth]{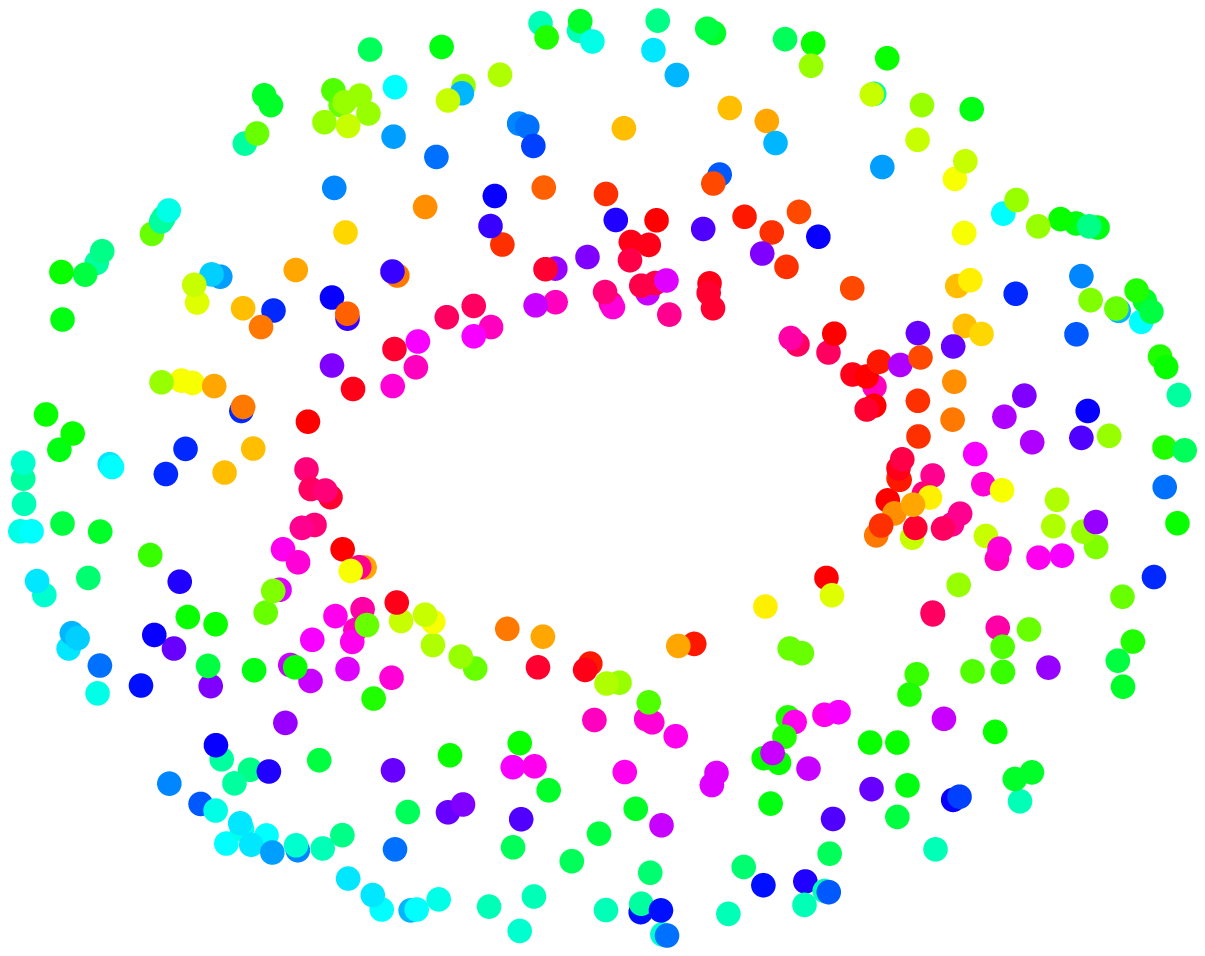} \qquad
  \includegraphics[width=0.3\linewidth]{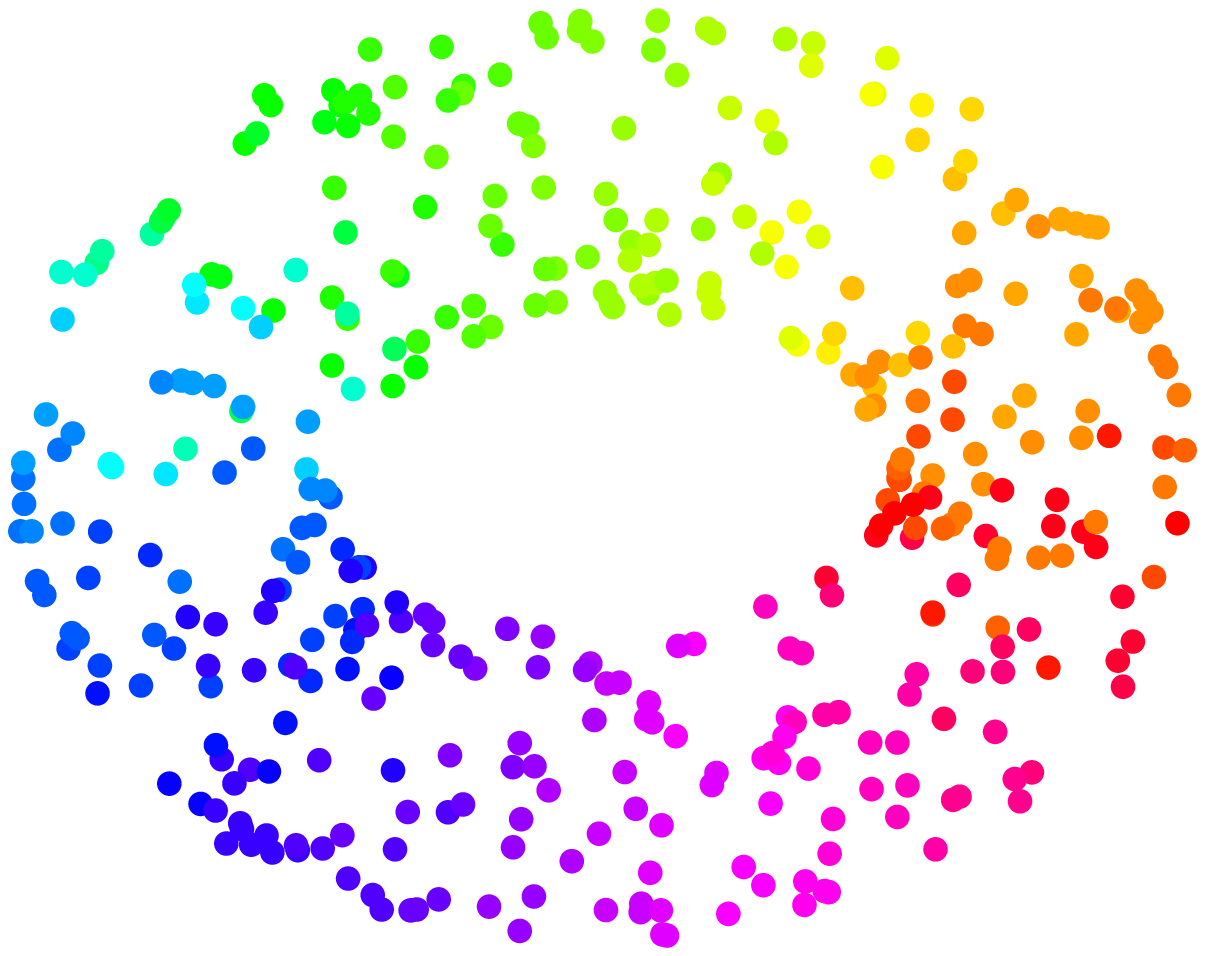}}
}
  \\
  \subfigure[Correlation scatter plots between the two original and two inferred coordinates.]{
  \begin{tabular}{cccc}
    & $\text{Inferred}_2$ & $\text{Original}_1$ &
    $\text{Original}_2$\\
    {\begin{sideways} \hspace*{2em}
        $\textrm{Inferred}_{1}$\end{sideways}} & 
    \includegraphics[width=0.3\linewidth]{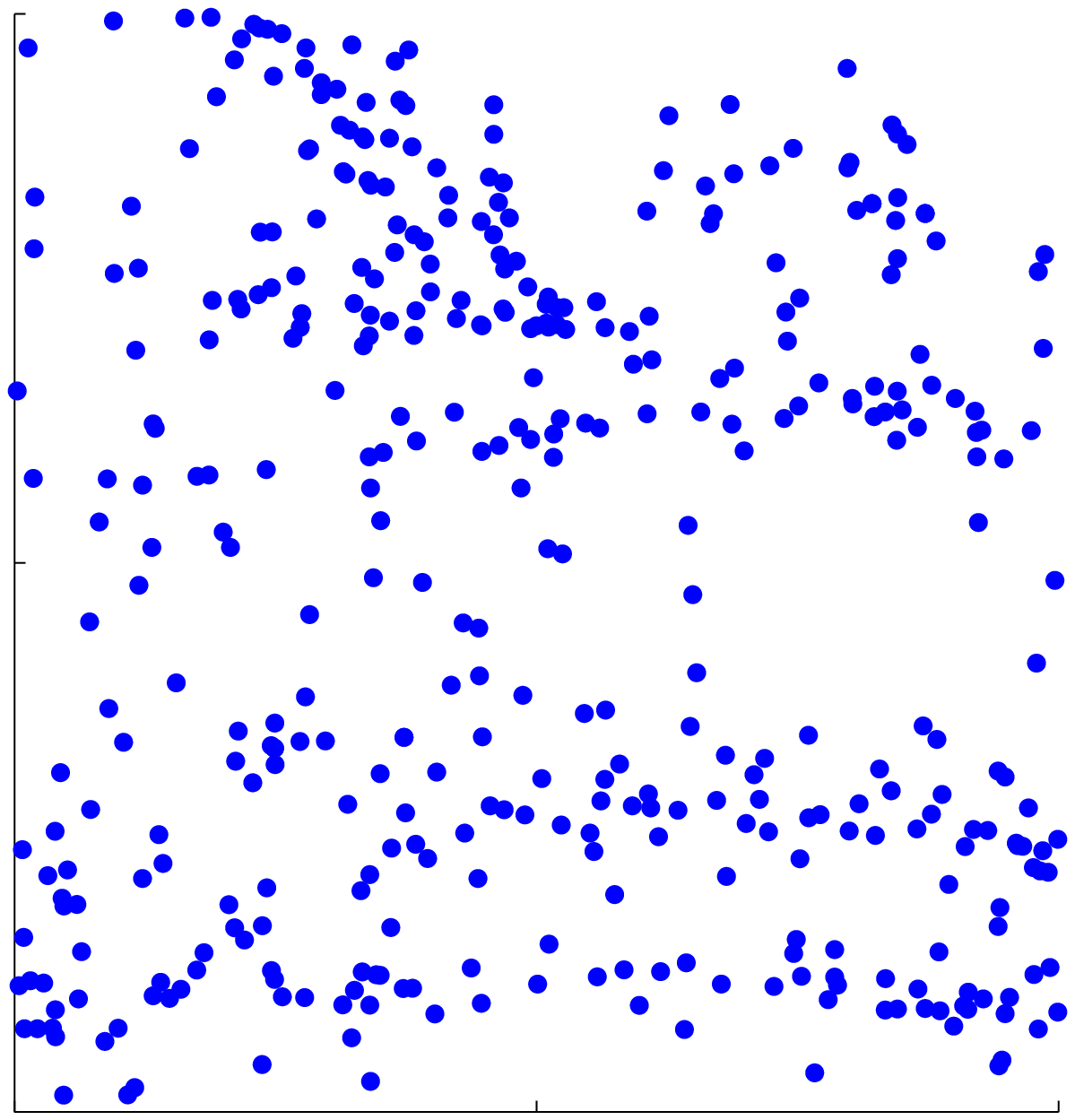} &
    \includegraphics[width=0.3\linewidth]{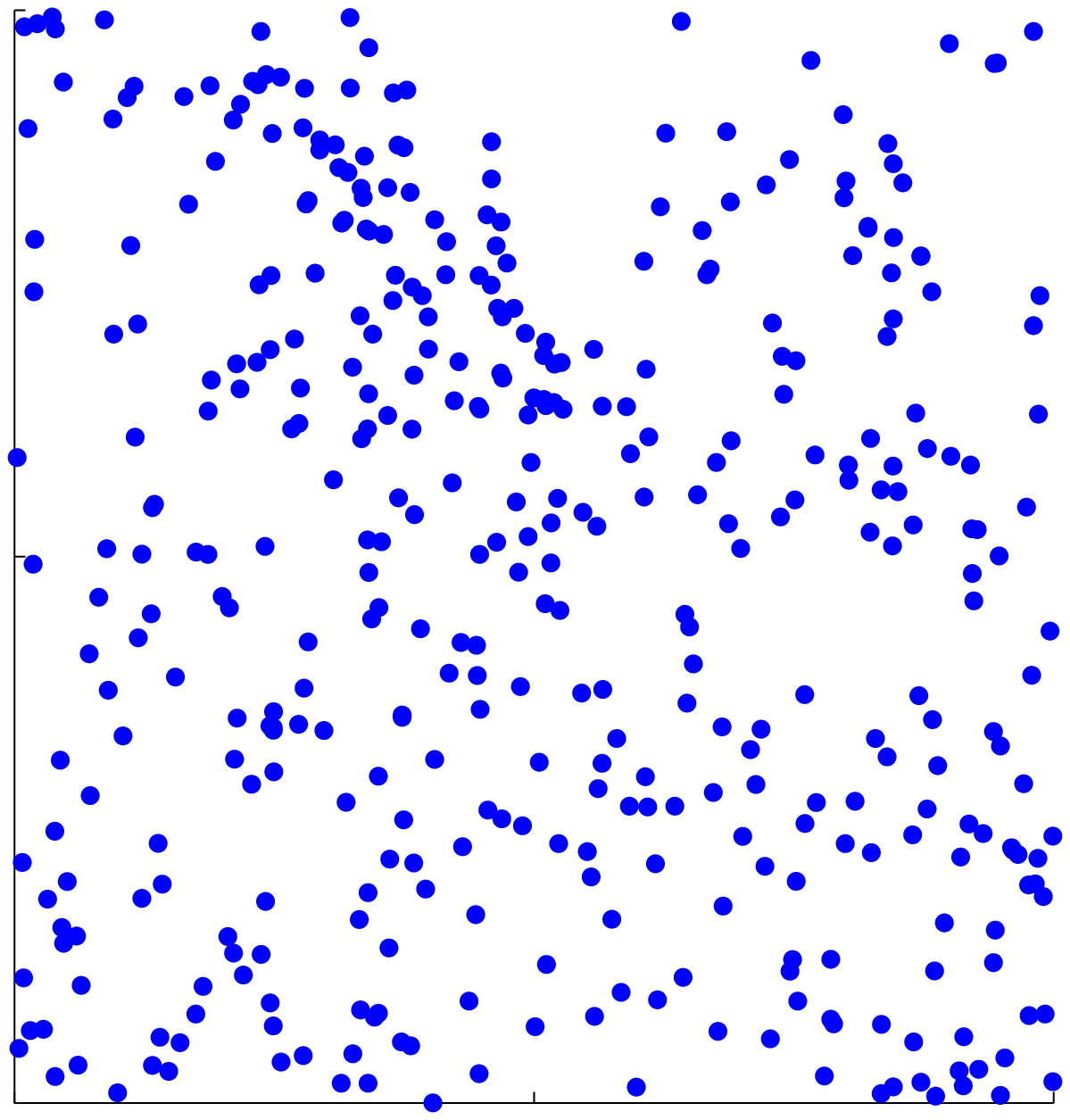} &
    \includegraphics[width=0.3\linewidth]{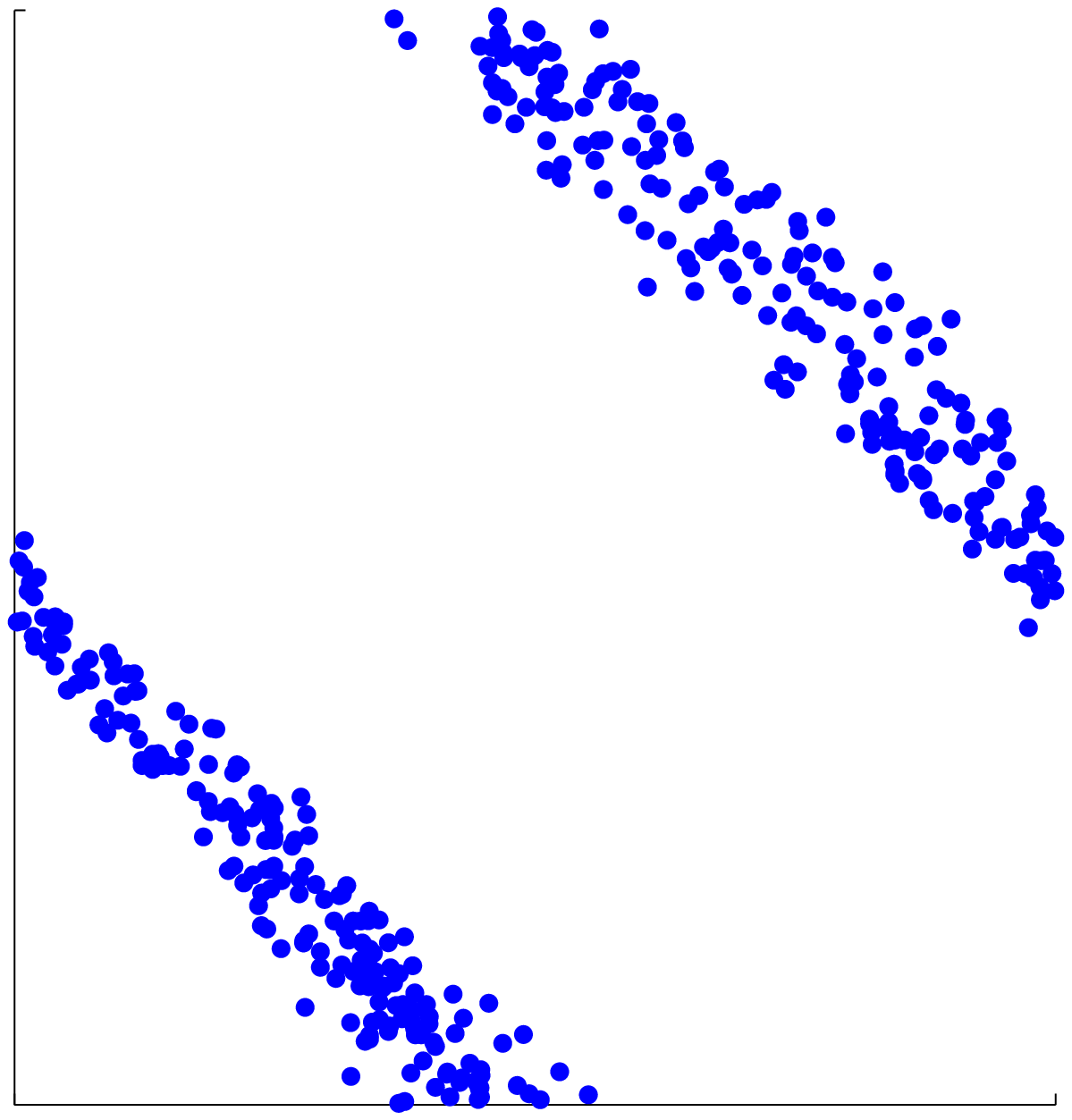} \\
    {\begin{sideways} \hspace*{2em}
        $\textrm{Inferred}_{2}$\end{sideways}} & 
    &
    \includegraphics[width=0.3\linewidth]{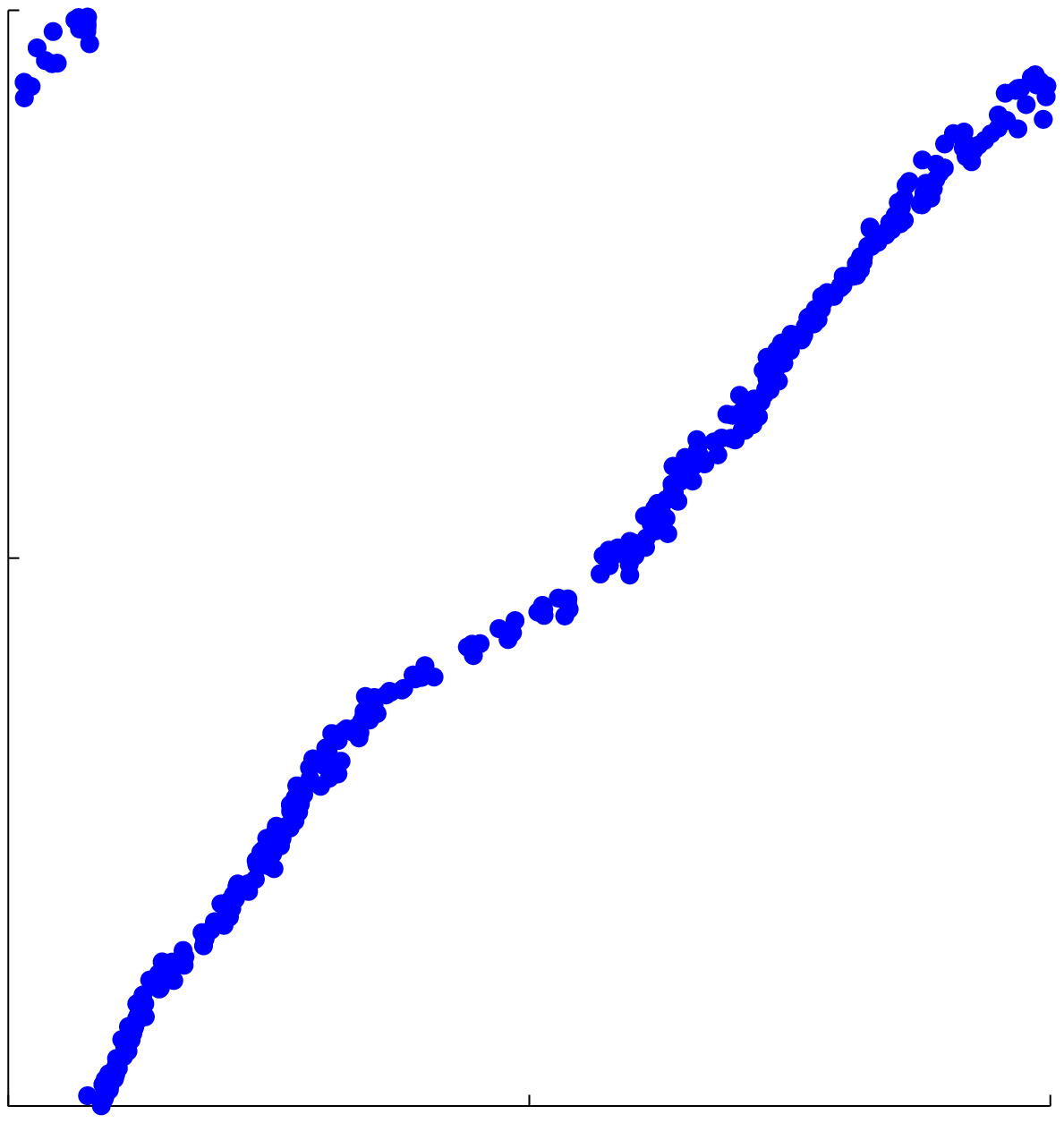} &
    \includegraphics[width=0.3\linewidth]{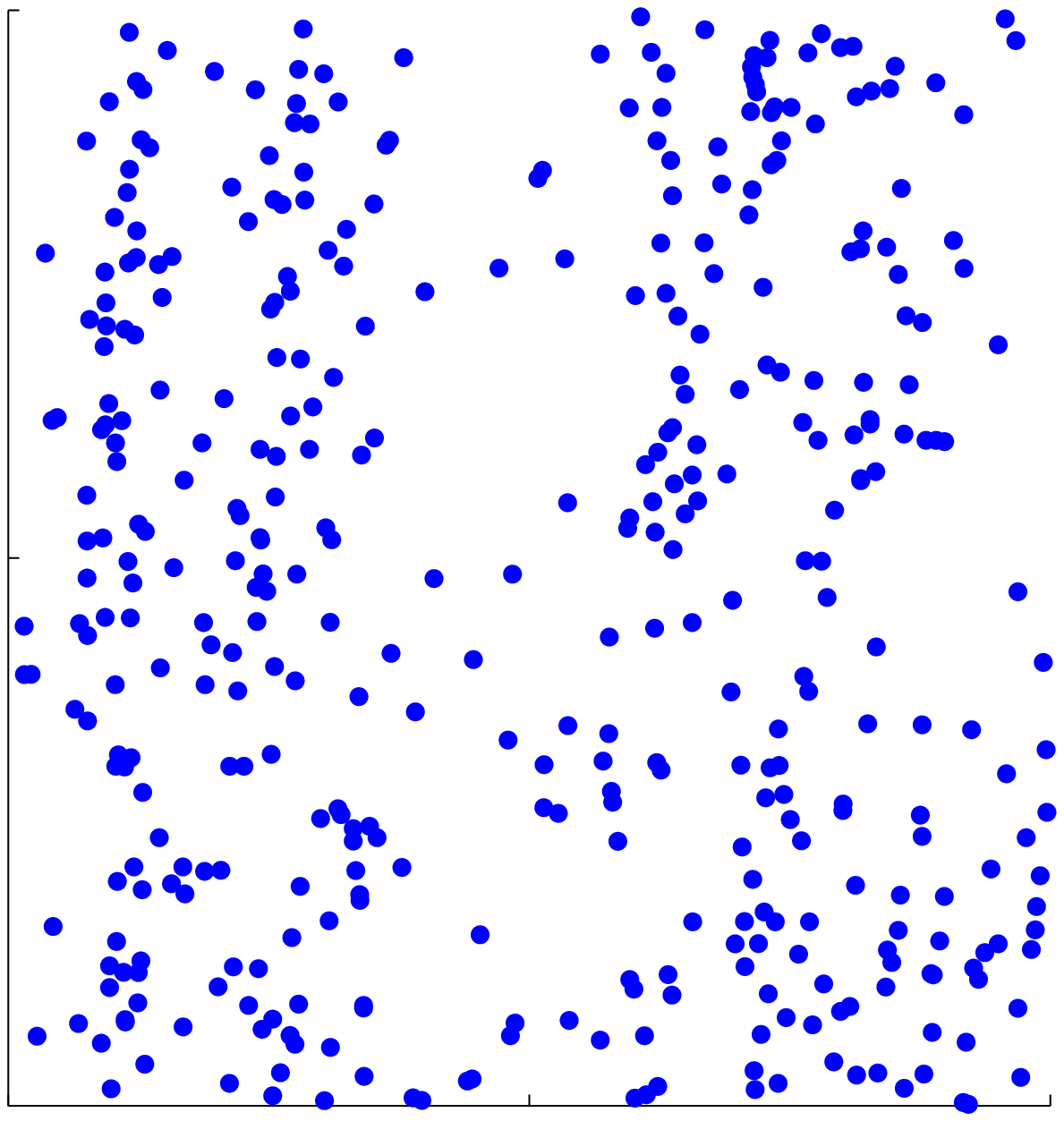} \\
    {\begin{sideways} \hspace*{2em}
        $\textrm{Original}_{1}$\end{sideways}} & 
    &
    &
    \includegraphics[width=0.3\linewidth]{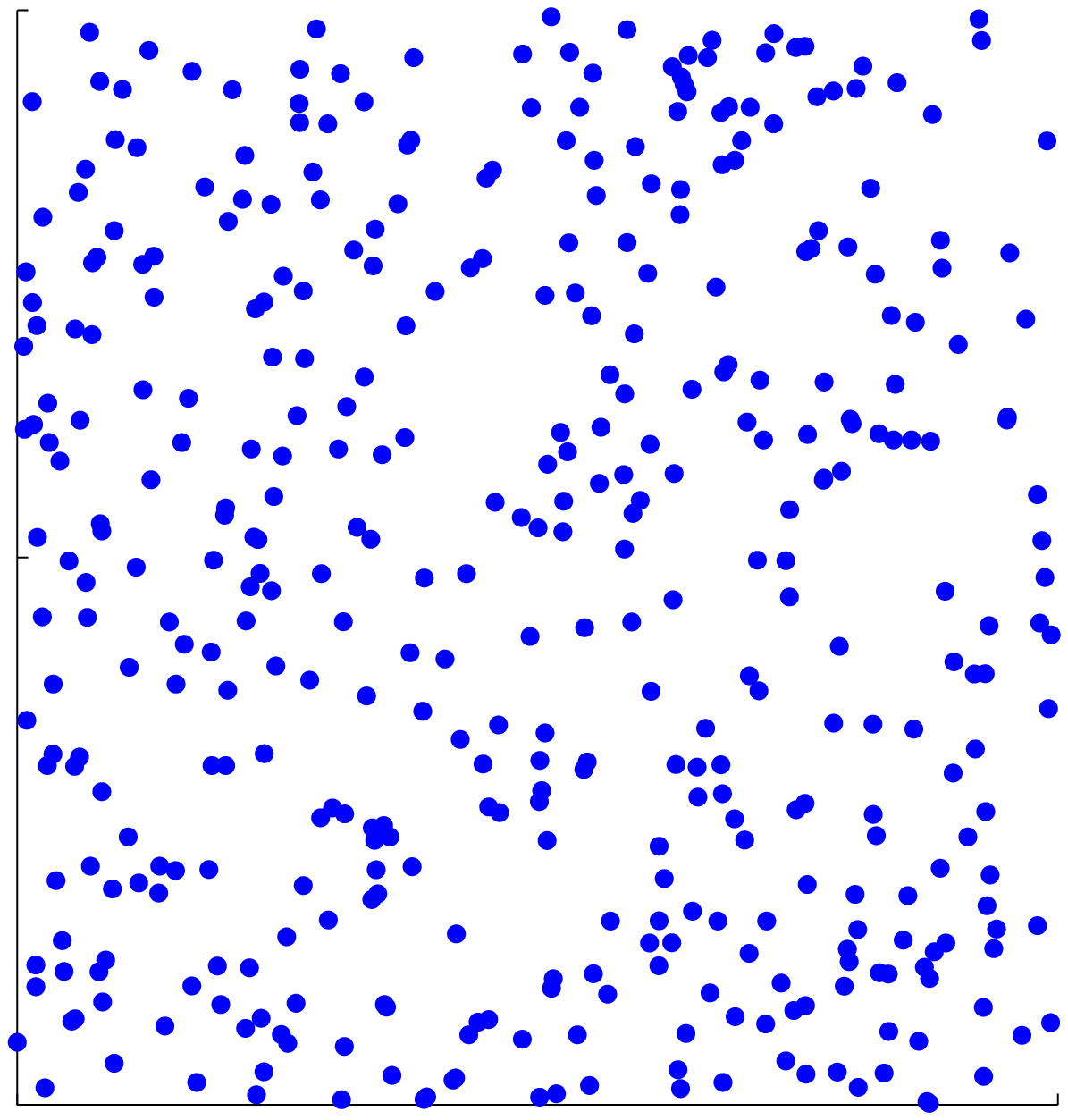} \\
  \end{tabular}
  }
  \caption{Torus in~$\Rr^3$.}
  \label{fig:torus}
\end{minipage}
\end{figure*}

The two inferred coordinates in this (fairly typical) experimental run recover the original coordinates essentially perfectly: the first inferred coordinate correlates with the meridional coordinate with topological degree~$-1$, while the second inferred coordinate correlates with the longitudinal coordinate with degree~$1$.

When the original coordinates are unavailable, the important figure is the inferred-versus-inferred scatter plot. In this case the scatter plot is fairly uniformly distributed over the entire coordinate square (i.e.\ torus). In other words, the two coordinates are decorrelated. This is slightly truer (and more clearly apparent in the scatter plot) for the two original coordinates. Contrast these with the corresponding scatter plots for a pair of circles (conjoined or disjoint).

\subsection{Elliptic curve}
\label{sec:elliptic-curve}

See Figure~\ref{fig:elliptic}. For fun, we repeated the previous experiment with a torus abstractly defined as the zero set of a homogeneous cubic polynomial in three variables, interpreted as a complex projective curve. We picked 400 points at random on $S^5 \subset \Cc^3$, subject to the cubic equation
\[
x^2y + y^2 z + z^2 x = 0.
\]
To interpret these as points in $\Cc P^2$, we used the projectively invariant metric
\[
d(\xi,\eta) = \cos^{-1}(|\bar{\xi} \cdot \eta|)
\]
for all pairs $\xi,\eta \in S^5$. With this metric we
built a Rips complex with maximal radius 0.15. The resulting complex
had 44184 simplices, and the cohomology was computed in 56
seconds. We found two dominant coclasses that survived beyond radius~0.15, and we computed our parametrizations at the 0.15 mark.

The resulting correlation plot quite clearly exhibits the decorrelation which is characteristic of the torus.

\begin{figure}
  \centering
  {\includegraphics[width=0.45\linewidth]{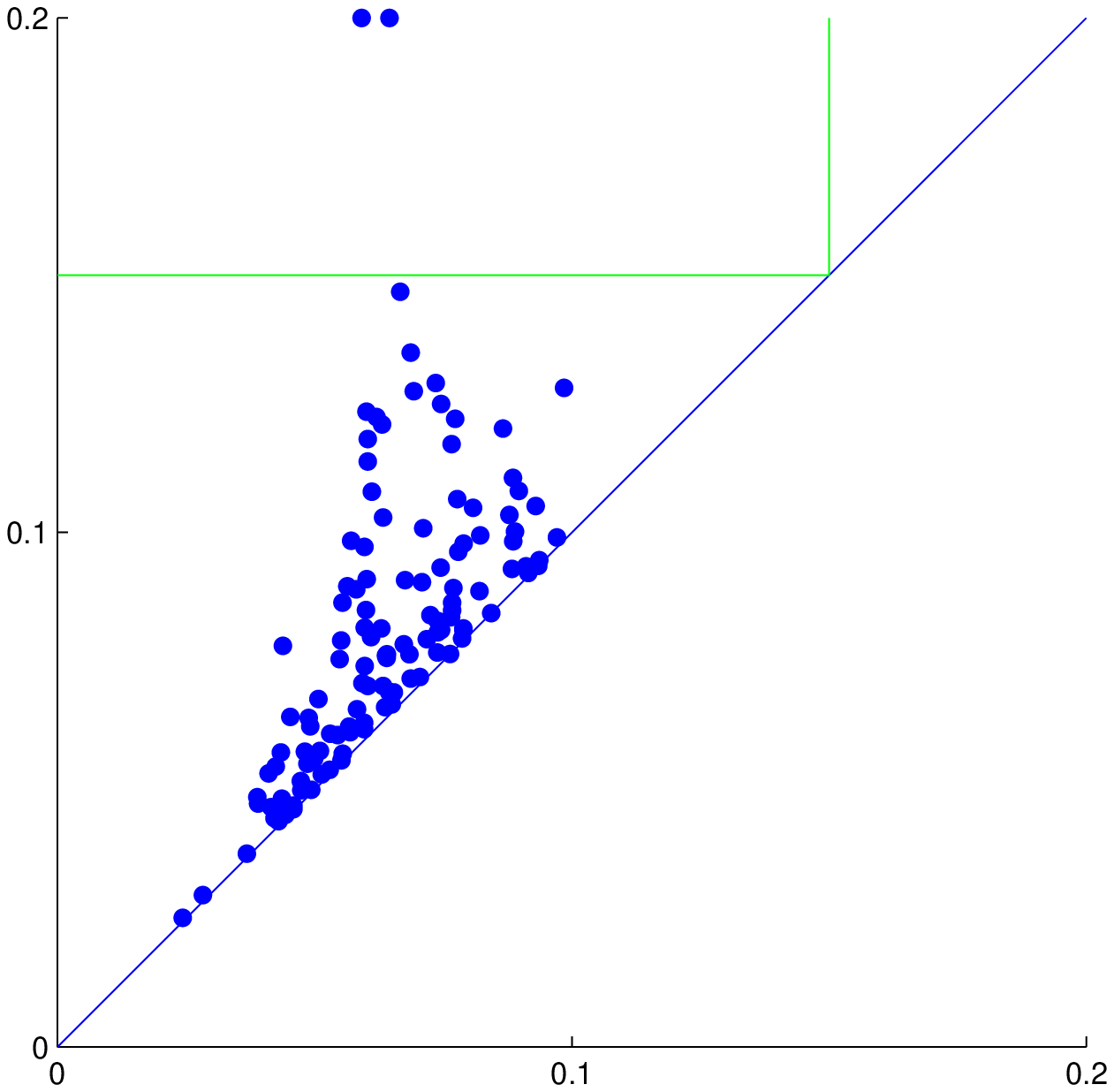}}
  \hfill
  {\includegraphics[width=0.45\linewidth]{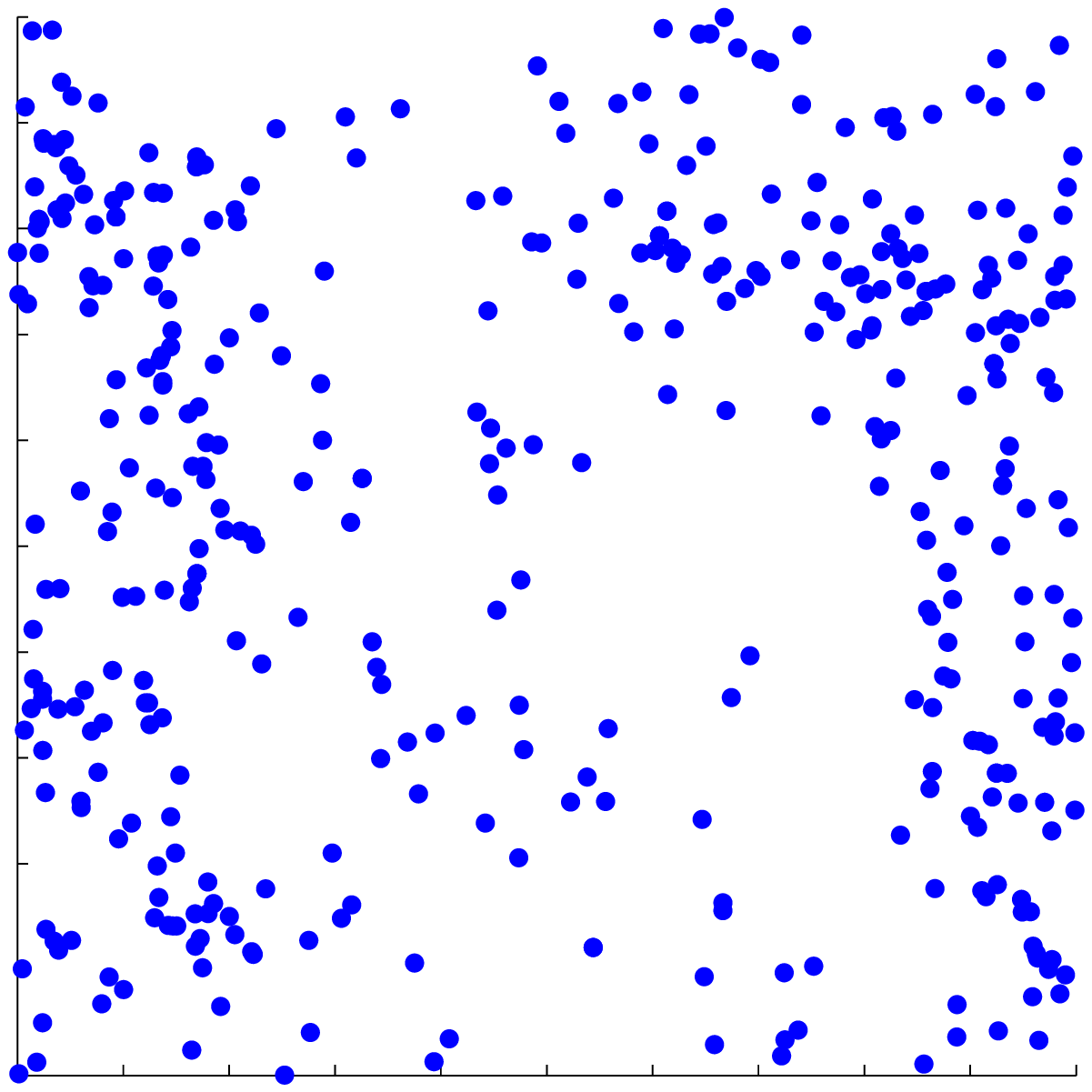} }
  \caption{Elliptic curve. Persistence diagram (left), correlation scatter plot between the two coordinates (right).}
  \label{fig:elliptic}
\end{figure}

\eject
\subsection{Double torus}
\label{sec:double-torus}

See Figure~\ref{fig:doubleTorus}. We constructed a genus-2 surface by generating 1600 points on a torus with inner and outer radii 1.0 and 3.0; slicing off part of the data set by a plane at distance 3.7 from the axis of the torus, and reflecting the remaining points in that plane. The resulting data set has 3120 points. Out of these, we pick 400 landmark points, and construct a witness complex with maximal radius 0.6. The landmark set yields a covering radius $r_{\max}=0.9982$ and a complex with 70605 simplices. The computation took 748 seconds active computer time. We identified the four most significant cocycles.

\begin{figure*}
  \centering
  \subfigure[Persistence diagram: four cocycles detected.]
  {
  \begin{minipage}{0.4\linewidth}
      \includegraphics[width=0.9\linewidth]{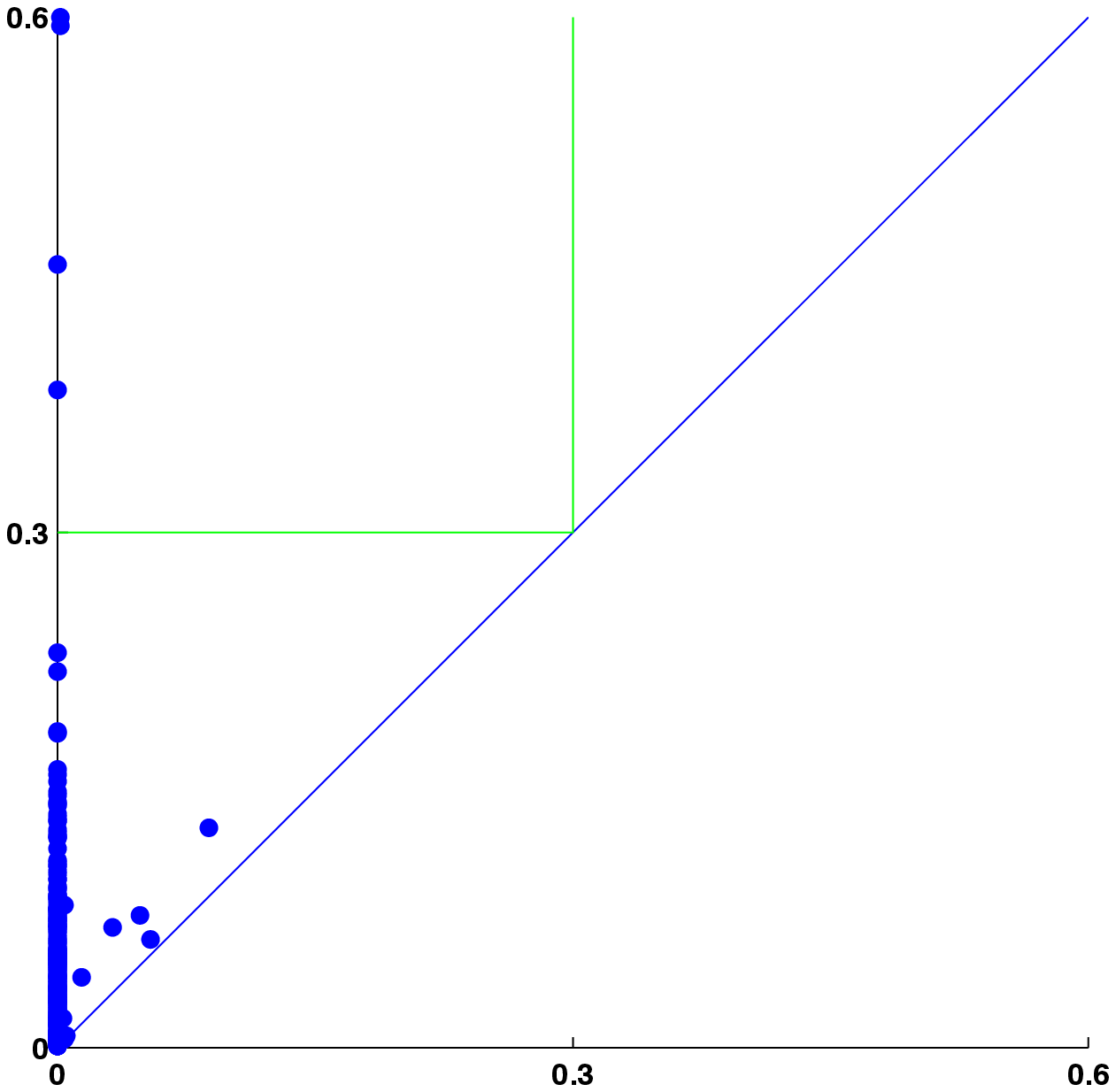}
  \end{minipage}
  }
  \\
  \subfigure[Correlation scatter plots between the four inferred coordinates.]
  {
  \begin{minipage}{0.8\linewidth}
    \begin{flushright}
     \includegraphics[width=0.225\linewidth]{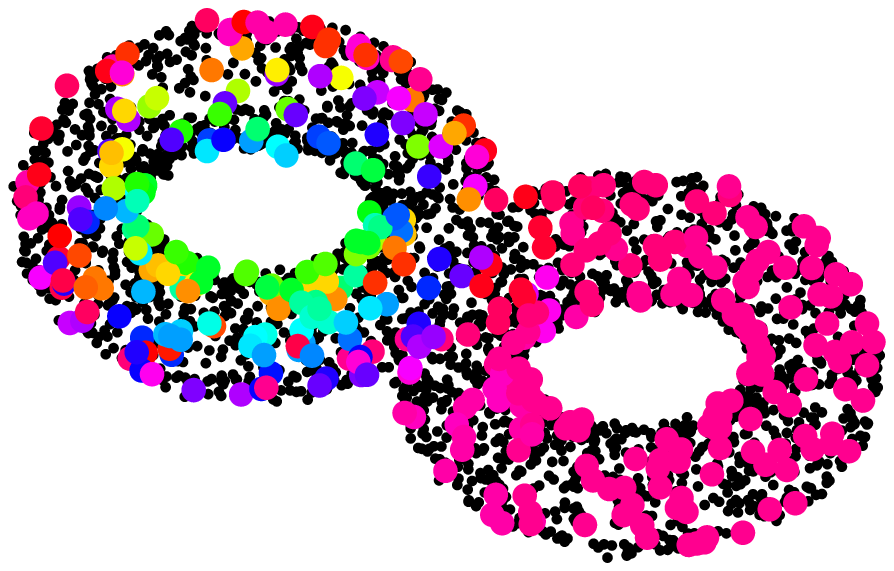}
     \hfill
     \includegraphics[width=0.225\linewidth]{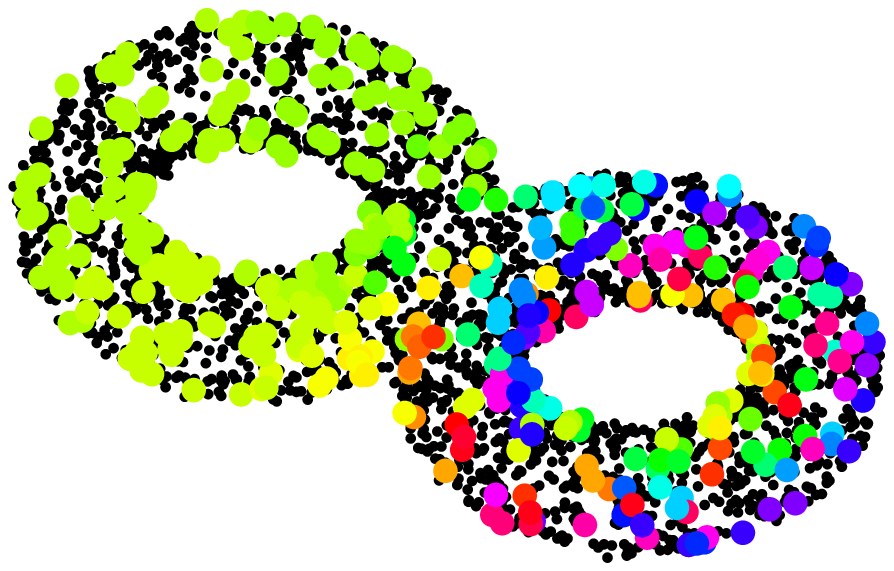}
     \hfill
     \includegraphics[width=0.225\linewidth]{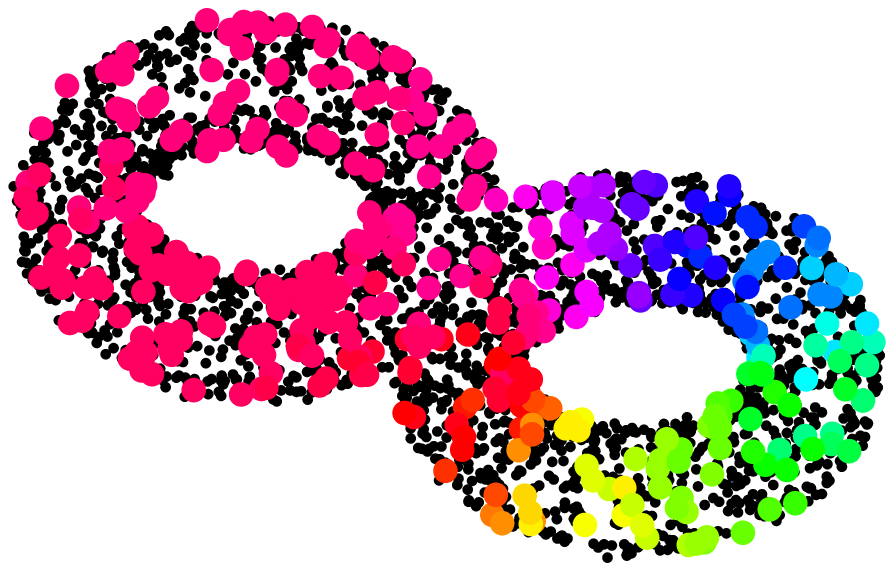}
     \hfill
     \includegraphics[width=0.225\linewidth]{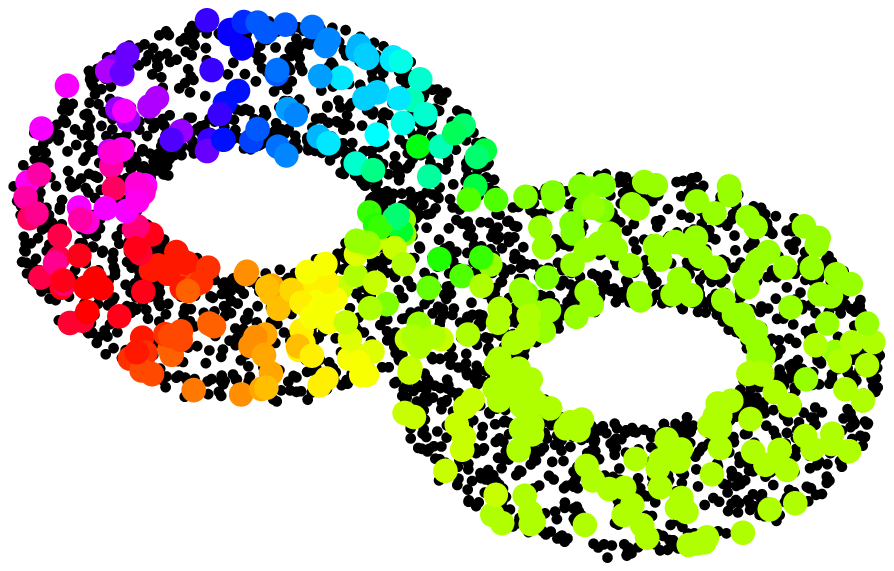}
    \\         
     \includegraphics[width=0.225\linewidth]{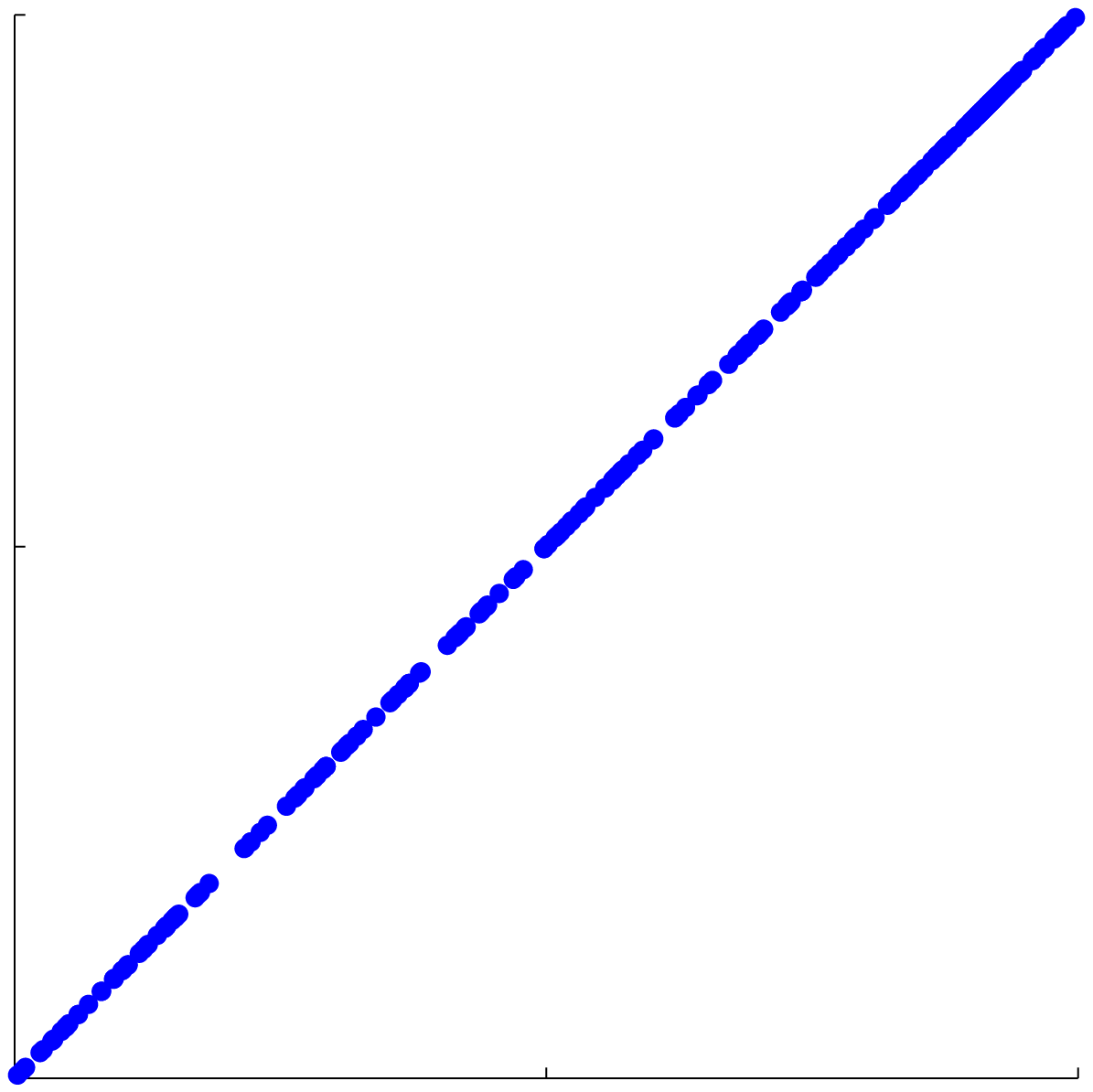}
     \hfill
     \includegraphics[width=0.225\linewidth]{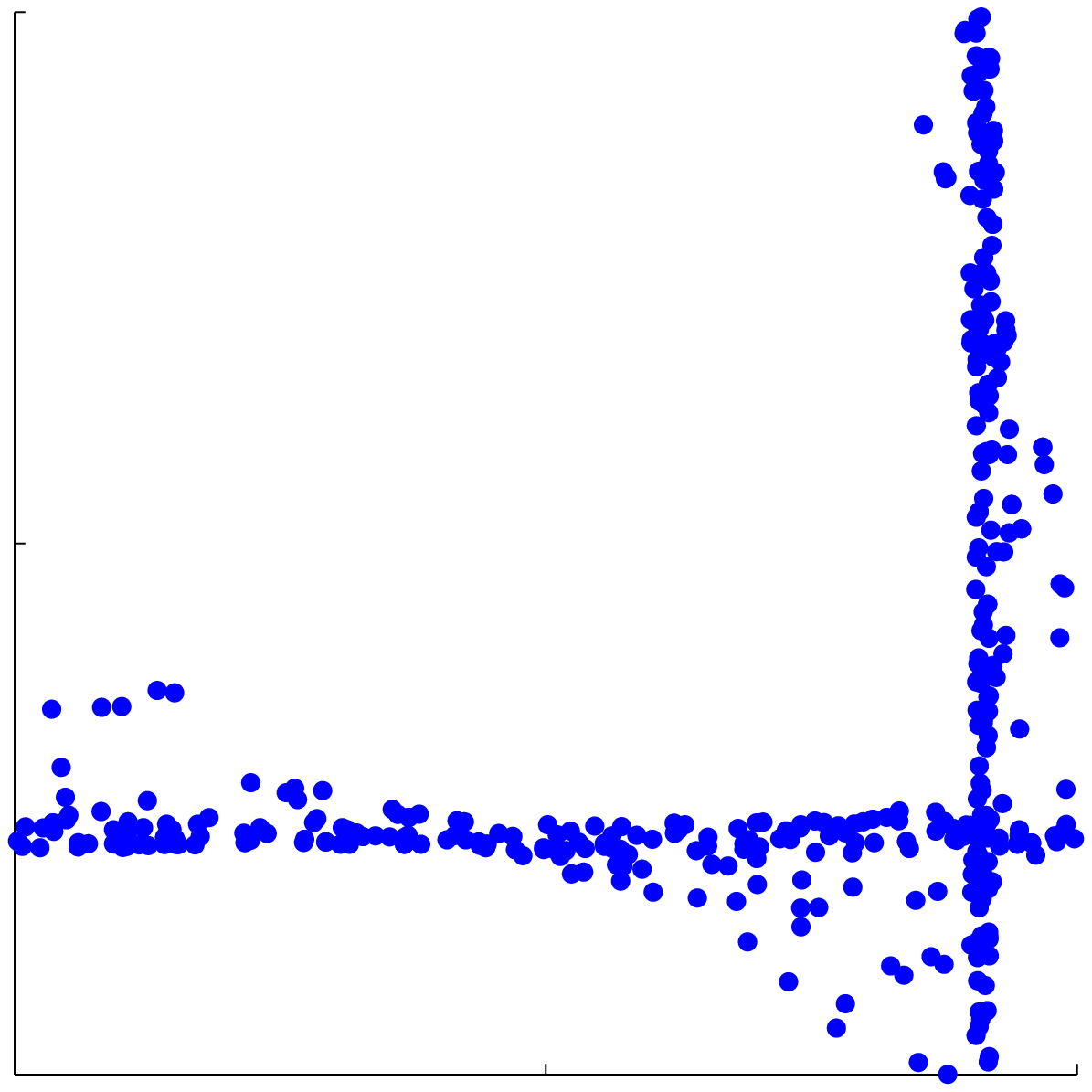}
     \hfill
     \includegraphics[width=0.225\linewidth]{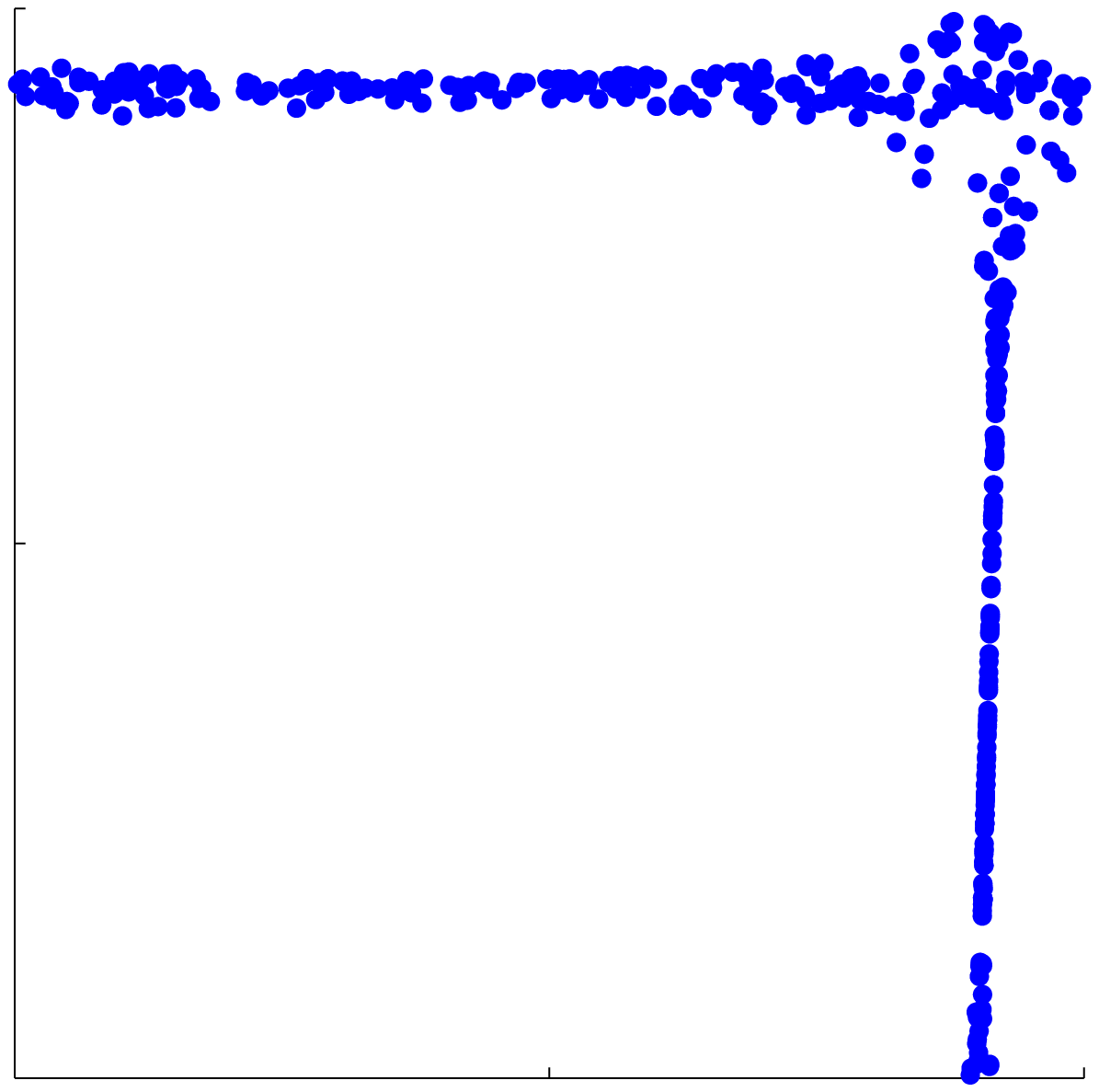}
     \hfill
     \includegraphics[width=0.225\linewidth]{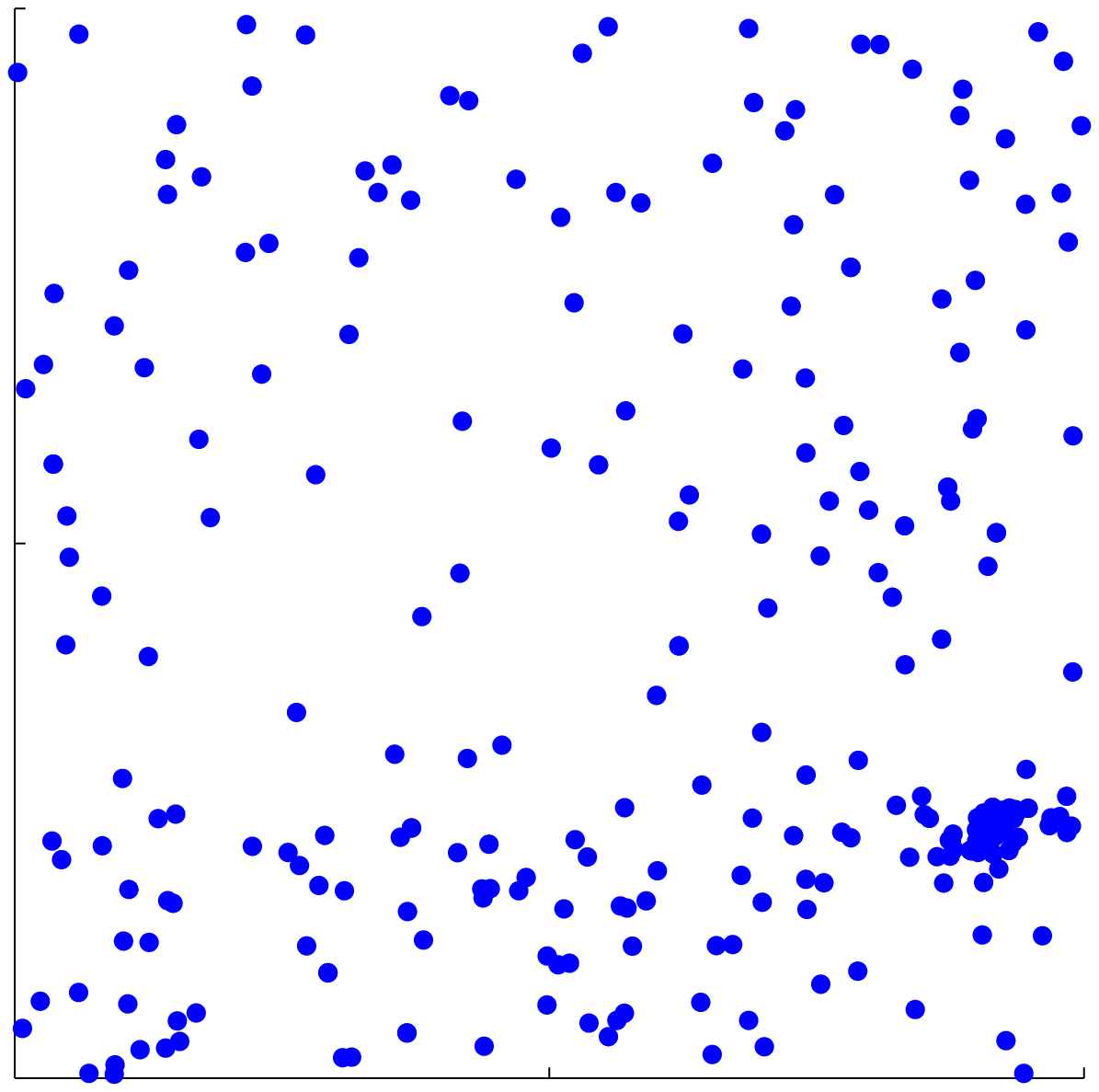}
\\         
     \includegraphics[width=0.225\linewidth]{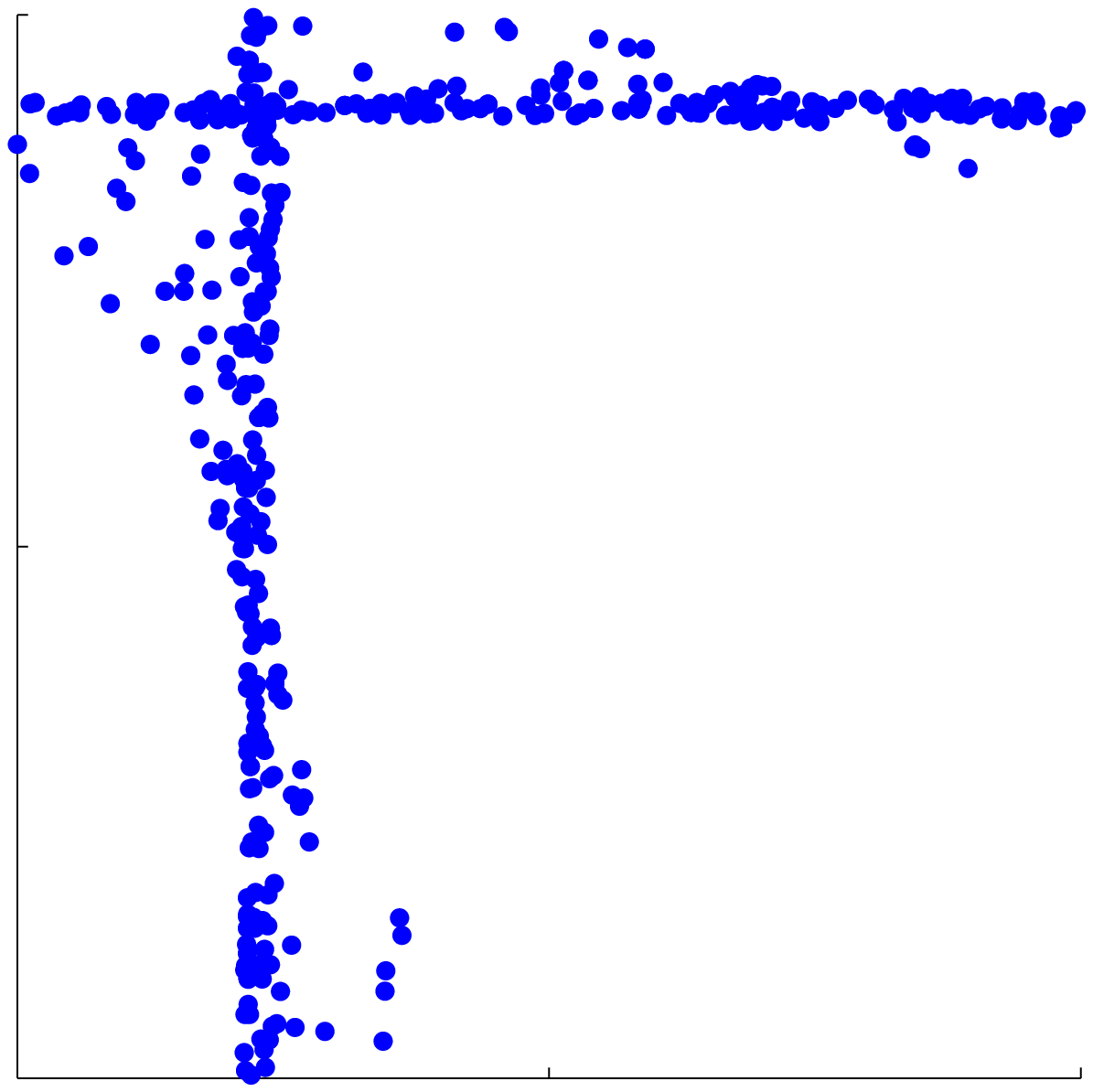}
     \hfill
     \includegraphics[width=0.225\linewidth]{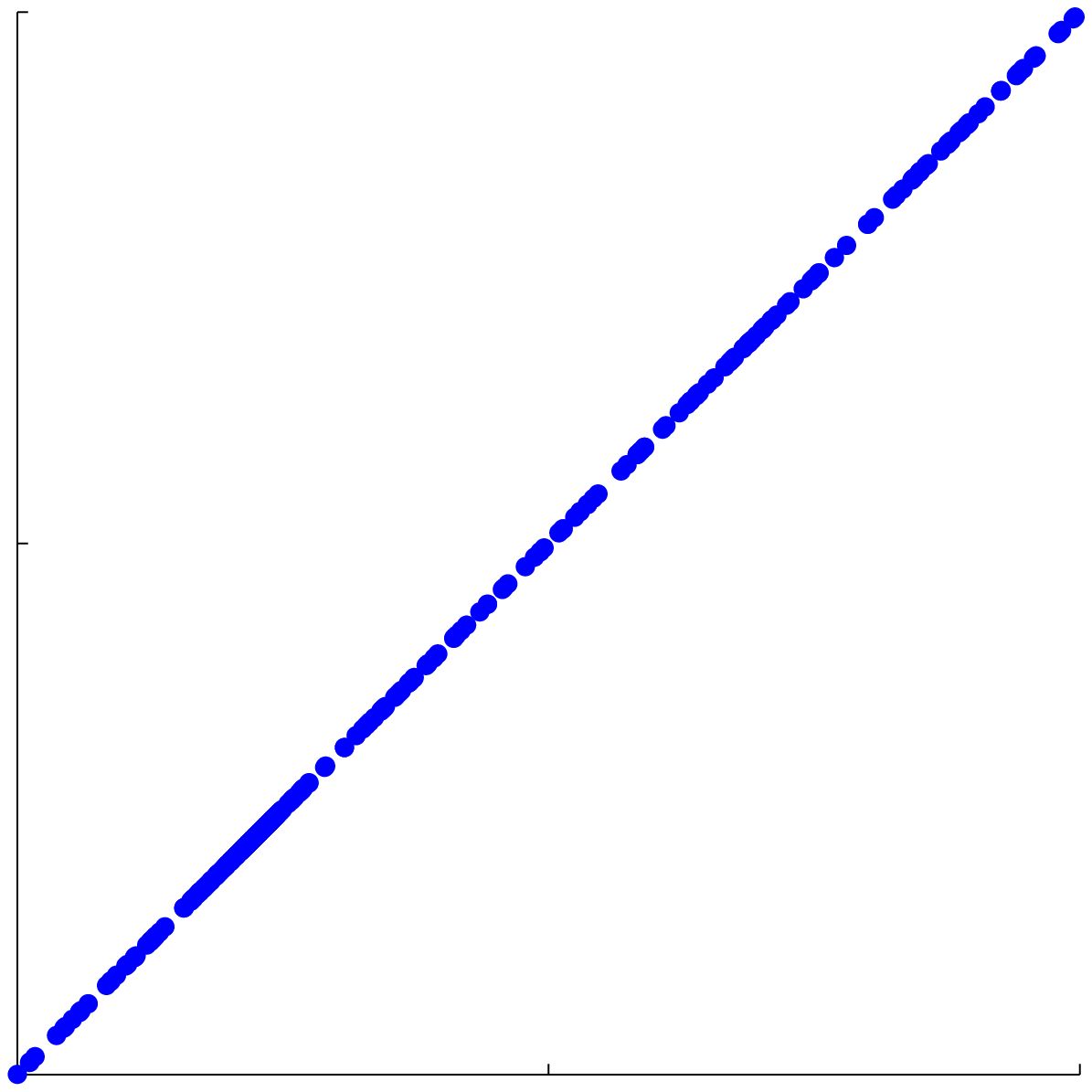}
     \hfill
     \includegraphics[width=0.225\linewidth]{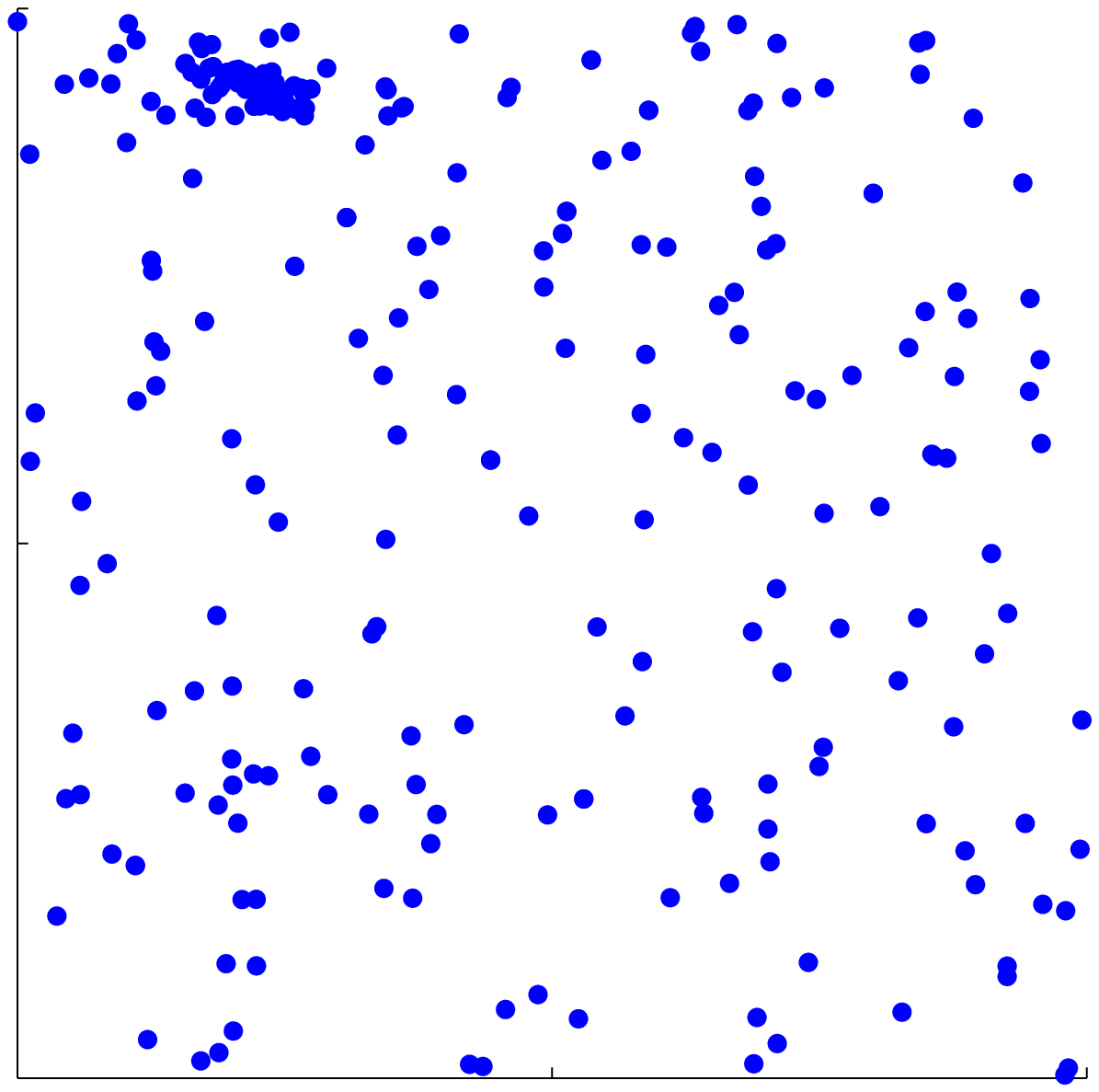}
     \hfill
     \includegraphics[width=0.225\linewidth]{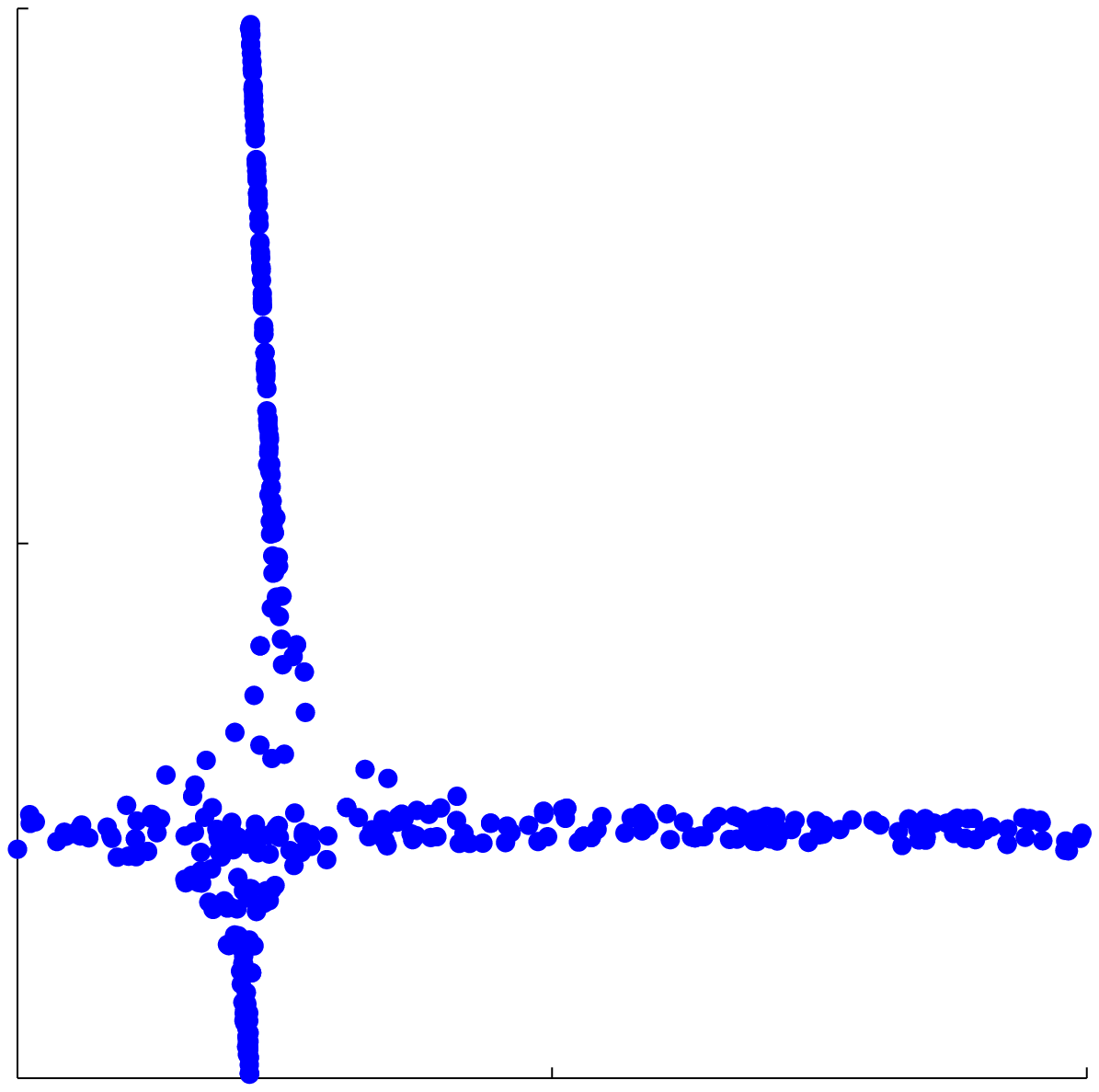}
\\         
     \includegraphics[width=0.225\linewidth]{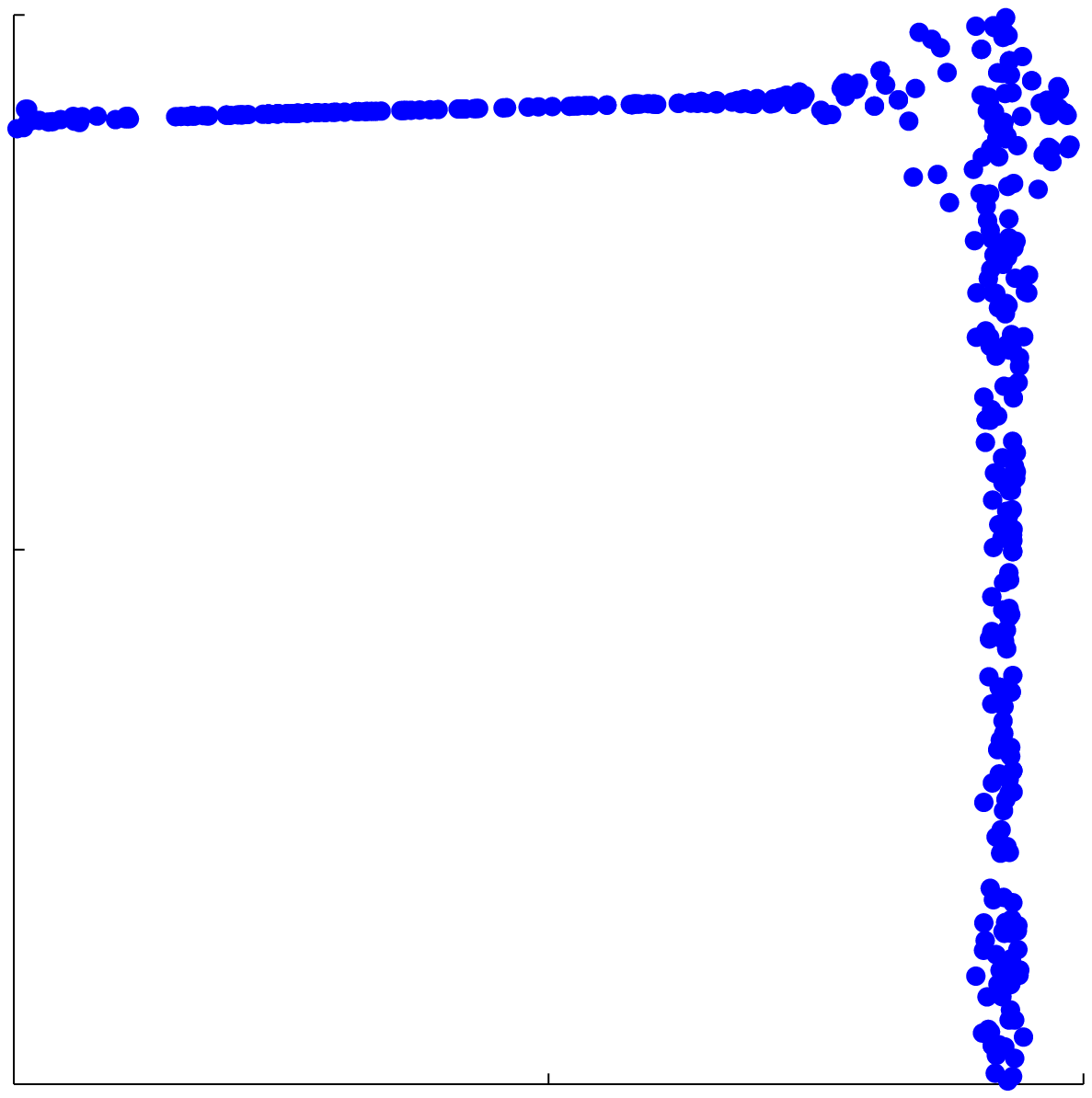}
     \hfill
     \includegraphics[width=0.225\linewidth]{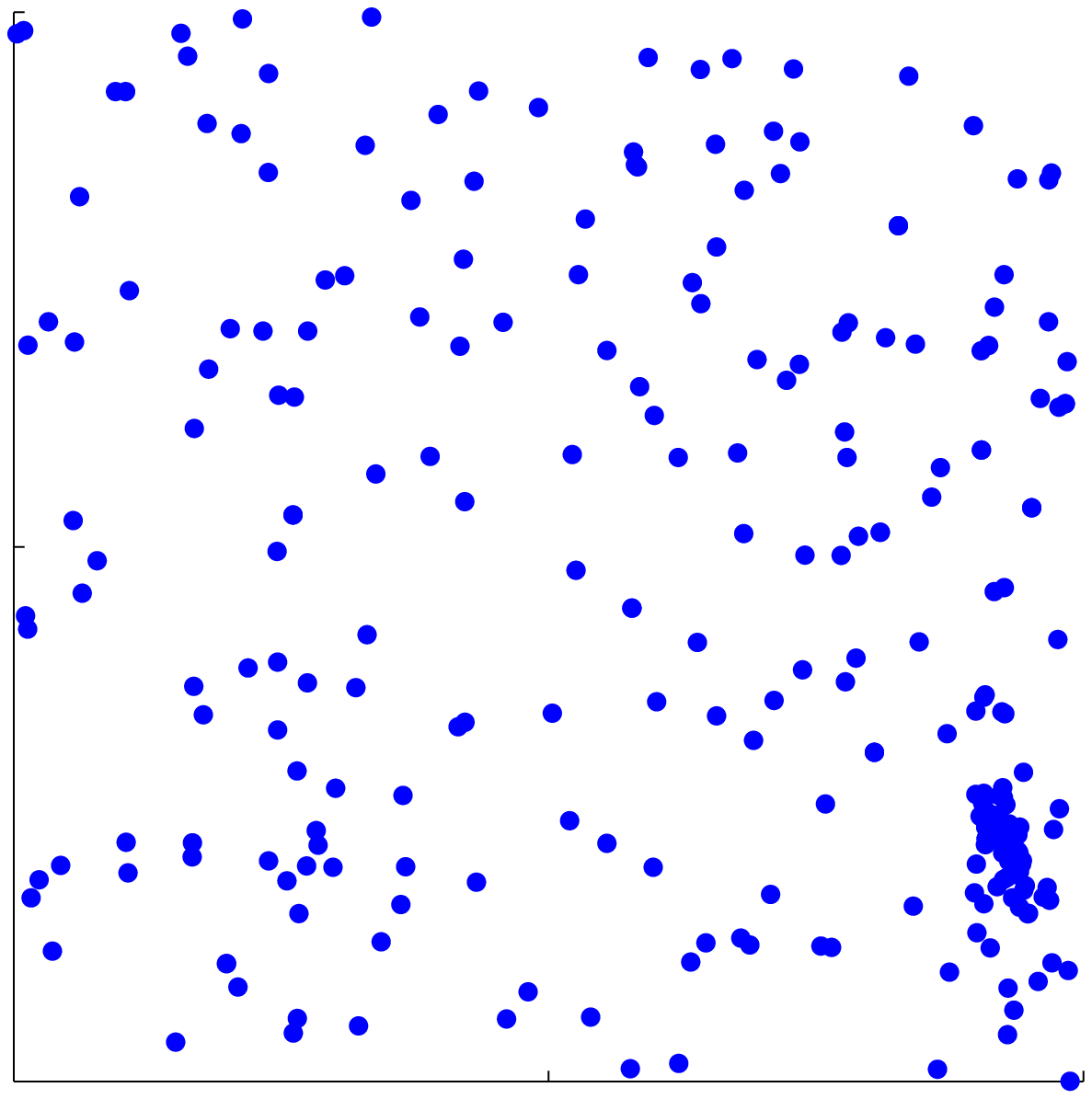}
     \hfill
     \includegraphics[width=0.225\linewidth]{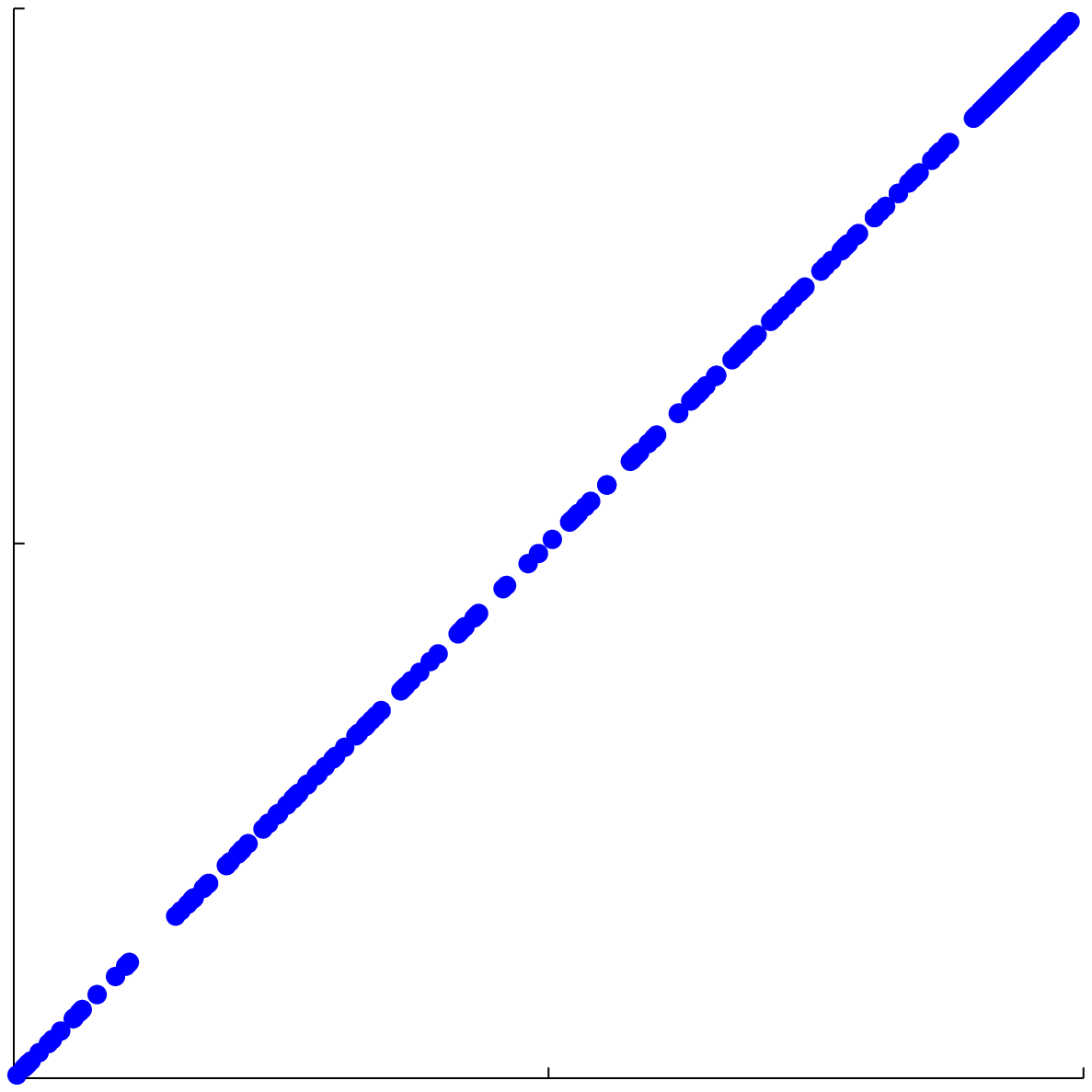}
     \hfill
     \includegraphics[width=0.225\linewidth]{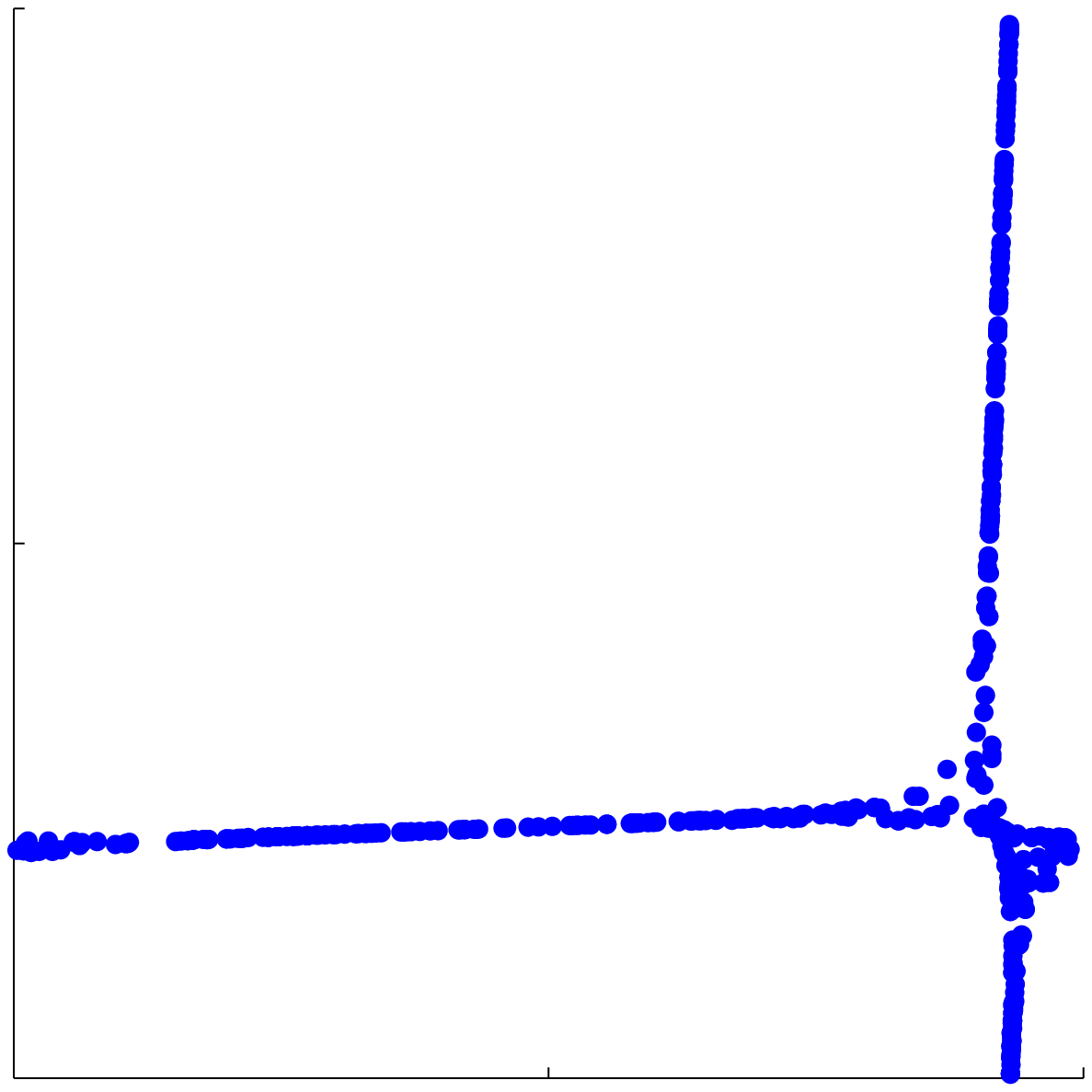}
\\         
     \includegraphics[width=0.225\linewidth]{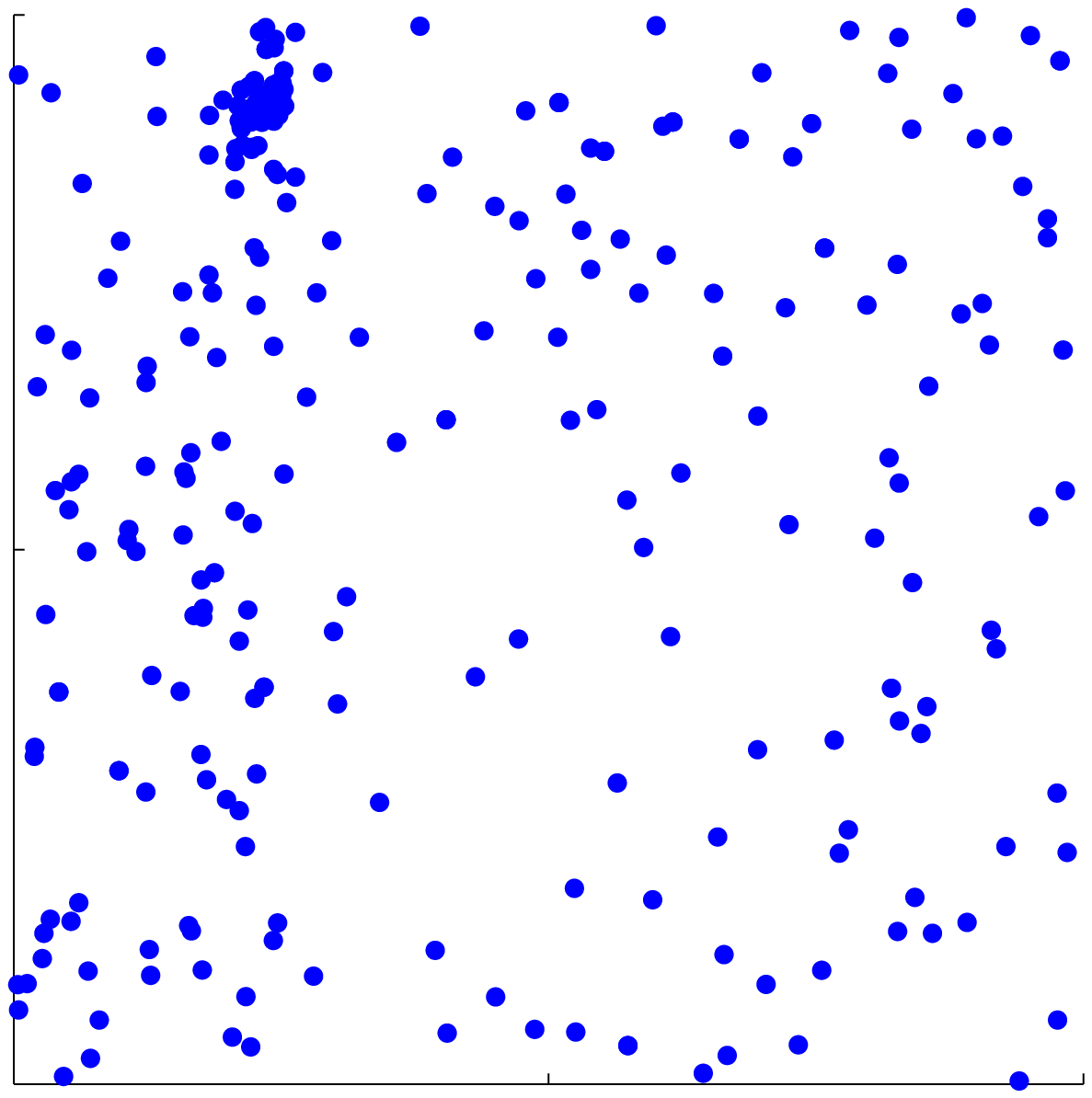}
     \hfill
     \includegraphics[width=0.225\linewidth]{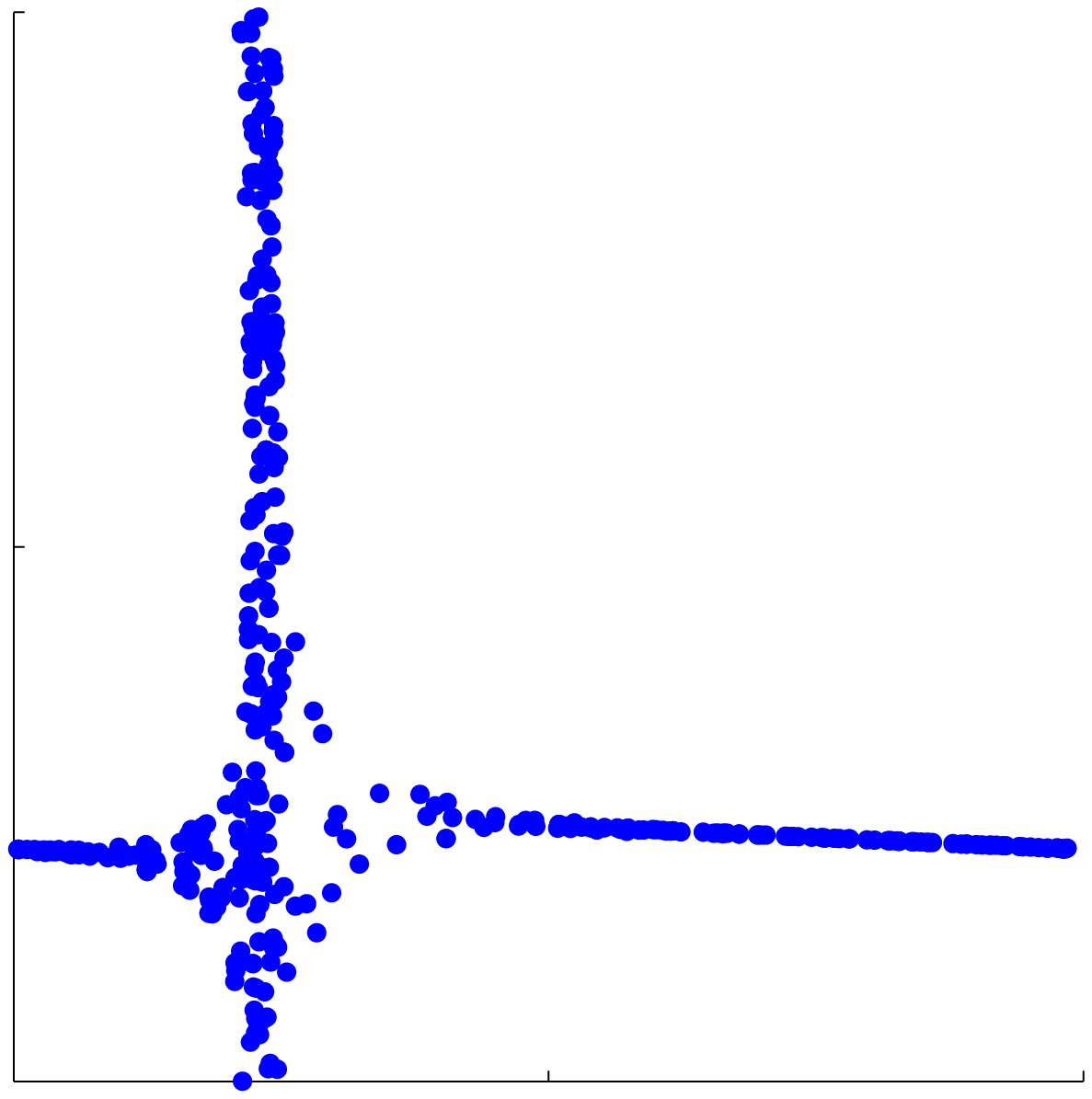}
     \hfill
     \includegraphics[width=0.225\linewidth]{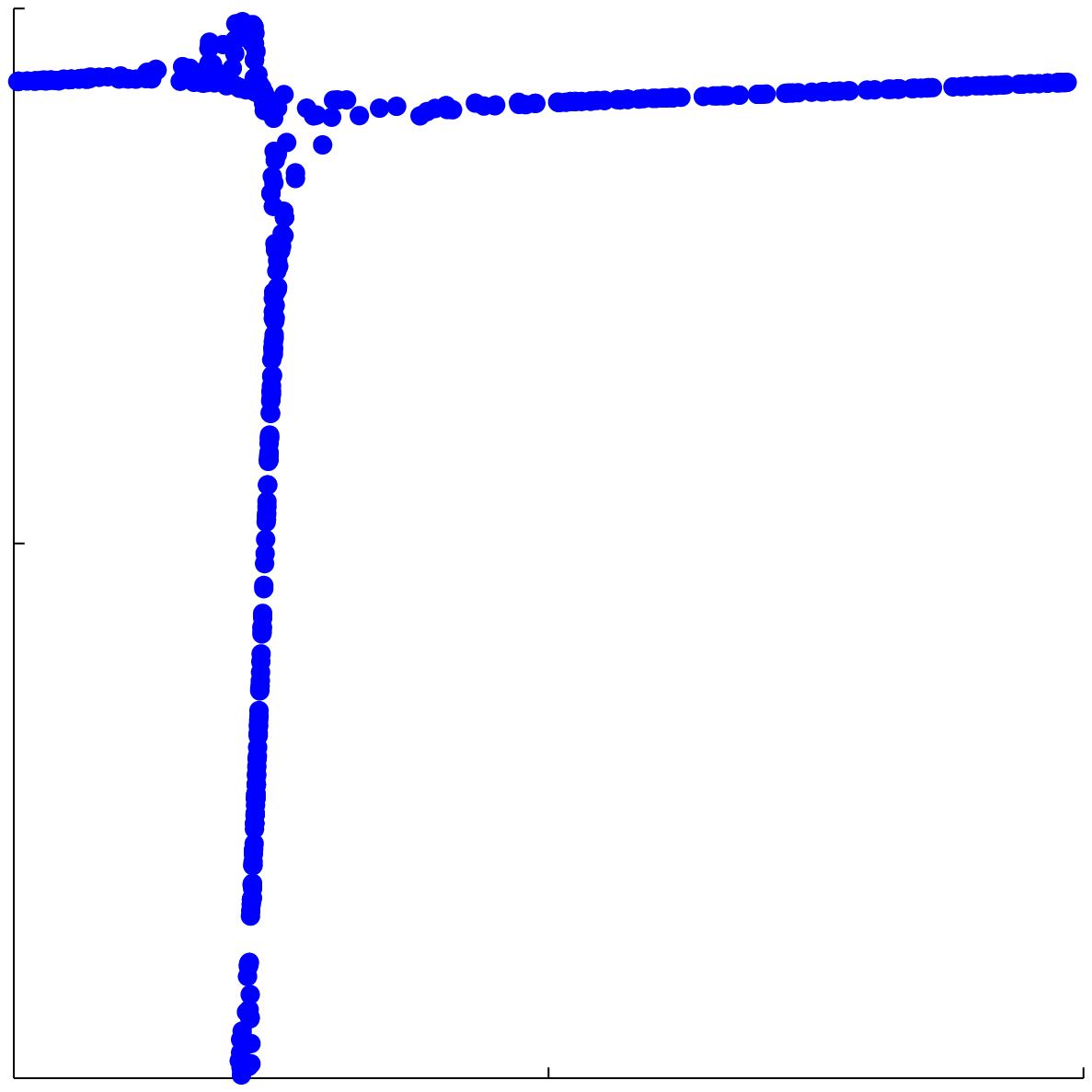}
     \hfill
     \includegraphics[width=0.225\linewidth]{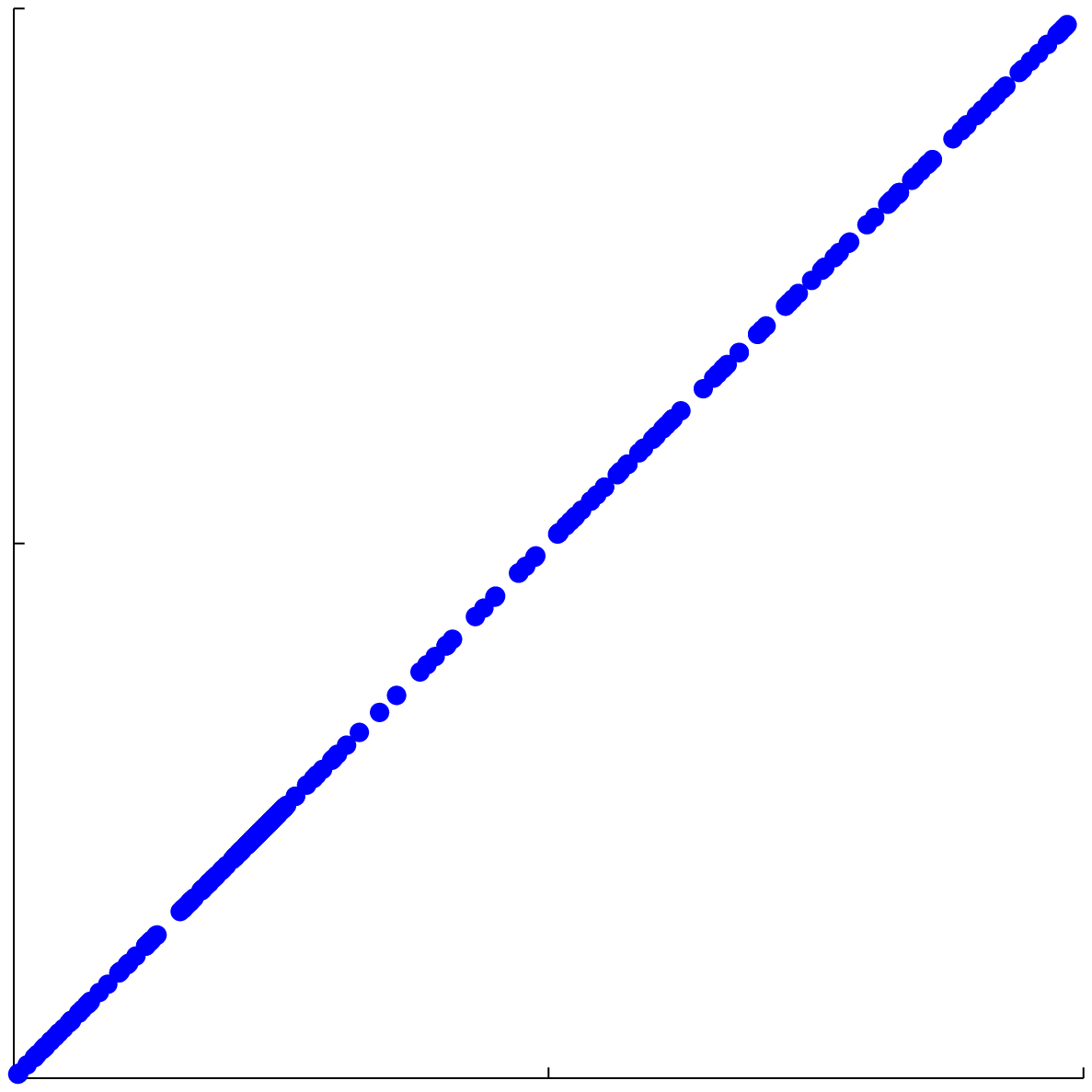}
    \end{flushright}
  \end{minipage}
  }
  \caption{Double torus in~$\Rr^3$.}
  \label{fig:doubleTorus}
\end{figure*}

Note that coordinates 1 and~4 are `coupled' in the sense that they are supported over the same subtorus of the double torus. The scatter plot shows that the two coordinates appear to be completely decorrelated except for a large mass concentrated at a single point. This mass corresponds to the other subtorus, on which coordinates 1 and~4 are essentially constant. A similar discussion holds for coordinates 2 and~3.

The uncoupled coordinate pairs (1,2), (1,3), (2,4), (3,4) produce scatter plots reminiscent of two conjoined or disjoint circles.

\eject
\section{Acknowledgements}
\label{sec:ack}

We are immensely grateful to Dmitriy Morozov: he has given us considerable assistance in implementing the algorithms in this paper. In particular we thank him for the persistent cocycle algorithm.

Thanks also to Jennifer Kloke for sharing her analysis of a visual image data set; this example did not make the present version of this paper. Finally, we thank Gunnar Carlsson, for his support and encouragement as leader of the topological data analysis research group at Stanford; and Robert Ghrist, as leader of the DARPA-funded project Sensor Topology and Minimal Planning (SToMP).

\bibliographystyle{abbrv}

\end{document}